\documentclass{memo-l}

\usepackage{}

\newtheorem{theorem}{Theorem}[chapter]
\newtheorem{lemma}[theorem]{Lemma}
\newtheorem{corollary}[theorem]{Corollary}
\newtheorem*{TTheorem}{Theorem I (1991)}
\newtheorem*{TBrennan}{Theorem II}
\newtheorem*{MTheorem}{Main Theorem I}
\newtheorem*{MTheorem2}{Main Theorem II}
\newtheorem*{ARSTheorem}{Theorem III (2009)}
\newtheorem{proposition}[theorem]{Proposition}

\newcommand{\quotes}[1]{``#1''}

\theoremstyle{definition}
\newtheorem{definition}[theorem]{Definition}
\newtheorem{example}[theorem]{Example}

\theoremstyle{remark}
\newtheorem{remark}[theorem]{Remark}

\numberwithin{section}{chapter}
\numberwithin{equation}{chapter}

\makeindex

\begin{document}

\frontmatter

\title{Approximation in the mean by rational functions}


\author{John B. Conway}
\address{Department of Mathematics, The George Washington University, Washington, DC 20052}
\curraddr{}
\email{conway@gwu.edu}
\thanks{}

\author{Liming Yang}
\address{Department of Mathematics, Virginia Polytechnic Institute and State University, Blacksburg, VA 24061}
\curraddr{}
\email{yliming@vt.edu}
\thanks{}

\date{}


\maketitle

\tableofcontents


\chapter*{Abstract}
For $1\le t < \infty$, a compact subset $K\subset\mathbb C$, and a finite positive measure $\mu$ supported on $K$, $R^t(K, \mu)$ denotes the closure in $L^t(\mu)$ of rational functions with poles off $K$. Let $\text{abpe}(R^t(K, \mu))$ denote the set of analytic bounded point evaluations.  The objective of this paper is to describe the structure of $R^t(K, \mu)$. In the work of Thomson on describing the closure in $L^t(\mu)$ of analytic polynomials, $P^t(\mu)$, the existence of analytic bounded point evaluations plays critical roles, while $\text{abpe}(R^t(K, \mu))$ may be empty. We introduce the concept of non-removable boundary $\mathcal F\subset \mathbb C$, a Borel set, such that the removable set $\mathcal R = K\setminus \mathcal F$ contains $\text{abpe}(R^t(K, \mu))$. Recent remarkable developments in analytic capacity and Cauchy transform provide us the necessary tools to describe $\mathcal F$ and obtain structural results for $R^t(K, \mu)$. 

Assume that $R^t(K, \mu)$ does not have a direct $L^t$ summand. The non-removable boundary $\mathcal F$ splits into three sets, $\mathcal F_0$ and $\mathcal F_+ \cup \mathcal F_-$ such that

(1) Cauchy transforms $\mathcal C(g\mu)$ of annihilating measures $g\mu$ ($g\perp R^t(K, \mu)$) are zero on $\mathcal F_0$, and $\mu | _{\mathcal F_0}$
has linear density zero (hence $\mathcal F_0\supset\mathbb C \setminus K$); and 

(2) $\mathcal F_+ \cup \mathcal F_-$ is
contained in a countable union of rotated Lipschitz graphs, $\mu |_{\mathcal F_+ \cup \mathcal F_-}$ is
absolutely continuous with respect to one-dimensional Hausdorff measure, and Cauchy transforms $\mathcal C(g\mu)$ of annihilating measures $g\mu$ have zero one side nontangential limits $\mu |_{\mathcal F_+ \cup \mathcal F_-}-a.a.$ in full analytic capacitary density.

Let $\mathcal L^2_{\mathcal R}$ denote the planar Lebesgue measure restricted to $\mathcal R$. Let $H^\infty (\mathcal R)$ be the weak$^*$ closure in $L^\infty (\mathcal L^2_{\mathcal R})$ of the functions $f(z)$, for which $f(z)$ is bounded analytic on $\mathbb C \setminus E_f$ for some compact subset $E_f \subset \mathcal F$. We prove:  

(1) There is a unique map $\rho$ satisfying: $\rho(f)(z)\mathcal C(g\mu)(z) = \mathcal C(fg\mu)(z)$ (except on a set of zero analytic capacity) for $f\in R^t(K, \mu)$ and $g\perp R^t(K, \mu)$. Moreover, $\rho(f)$ is continuous in full analytic capacitary density on $\mathcal R$ and has nontangential limits $\mu |_{\mathcal F_+ \cup \mathcal F_-}-a.a.$ in full analytic capacitary density. 

(2) The map $\rho$ is an isometric isomorphism and a weak$^*$ homeomorphism from $R^t(K, \mu)\cap L^\infty(\mu )$ onto $H^\infty (\mathcal R)$.

Consequently, we show that a decomposition theorem (Main Theorem II) of $R^t(K, \mu)$ holds for an arbitrary compact subset $K$ and a finite positive measure $\mu$ supported on $K$, which extends the central results regarding $P^t(\mu)$.

\bigskip

2010 Mathematics Subject Classification Primary 47A15; Secondary 30C85, 31A15, 46E15, 47B38



\mainmatter

\chapter{Introduction}
\bigskip

\section{Known results}
\bigskip

Let $\mu$ be a finite, positive Borel measure that is compactly 
supported in $\mathbb{C}$. We require that the support of $\mu$, $\mbox{spt}(\mu)$,  be contained in
some compact subset $K\subset \mathbb C$. For $1\leq t < \infty$, the analytic polynomials and functions in $\mbox{Rat}(K) := \{q:\mbox{$q$ is a rational function with poles off $K$}\}$ are members of $L^t(\mu)$. We let $P^t(\mu)$ denote the closure of the (analytic) 
polynomials in $L^t(\mu)$ and let $R^t(K, \mu)$ denote the closure of $\mbox{Rat}(K)$ in $L^t(\mu)$.
Since $K$ is bounded, the operator $f \rightarrow M_\mu f,~  M_\mu f (z) = z f (z)$, is
bounded on $L^t(\mu)$ and takes $P^t(\mu)$ (resp., $R^t(K, \mu)$) into itself. Let $S_\mu = M_\mu |_{P^t(\mu)}$ (resp., $S_\mu = M_\mu |_{R^t(K, \mu)}$).
It is well known that the spectrum of $S_\mu$, $\sigma(S_\mu)$, is contained in $K$ and $R^t(K, \mu) = R^t(\sigma(S_\mu), \mu)$ (see, for example, Proposition 1.1 in \cite{ce93}). Throughout this paper, we assume $K = \sigma(S_\mu)$. That assures, for example, that $\partial K \subset \text{spt}(\mu)$. The operator $S_\mu$ is pure if $P^t(\mu)$ (resp., $R^t(K, \mu)$) does not have a direct $L^t$ summand and is irreducible if $P^t(\mu)$ (resp., $R^t(K, \mu)$) contains no non-trivial characteristic functions.  

A point $z_0$ in $\mathbb{C}$ (resp., $z_0$ in $K$) is called a \textit{bounded point evaluation} for $P^t(\mu)$ (resp., $R^t(K, \mu)$)
if $f\mapsto f(z_0)$ defines a bounded linear functional for the analytic polynomials (resp., functions in $\mbox{Rat}(K)$)
with respect to the $L^t(\mu)$ norm. The collection of all such points is denoted $\mbox{bpe}(P^t(\mu))$ 
(resp., $\mbox{bpe}(R^t(K, \mu)$)).  If $z_0$ is in the interior of $\mbox{bpe}(P^t(\mu))$ (resp., $\mbox{bpe}(R^t(K, \mu)$)) 
and there exist positive constants $r$ and $M$ such that $|f(z)| \leq M\|f\|_{L^t(\mu)}$, whenever $|z - z_0|\leq r$ 
and $f$ is an analytic polynomial (resp., $f\in \mbox{Rat}(K)$), then we say that $z_0$ is an 
\textit{analytic bounded point evaluation} for $P^t(\mu)$ (resp., $R^t(K, \mu)$). The collection of all such 
points is denoted $\mbox{abpe}(P^t(\mu))$ (resp., $\mbox{abpe}(R^t(K, \mu)$)). Actually, it follows from Thomson's Theorem
\cite{thomson} (or see Theorem I, below) that $\mbox{abpe}(P^t(\mu))$ is the interior of $\mbox{bpe}(P^t(\mu))$. 
This also holds in the context of $R^t(K, \mu)$ as was shown by J. Conway and N. Elias in \cite{ce93}. Now, 
$\mbox{abpe}(P^t(\mu))$ is the largest open subset of $\mathbb{C}$ to which every function in $P^t(\mu)$ has an analytic 
continuation under these point evaluation functionals, and similarly in the context of $R^t(K, \mu)$. 
\smallskip

Our story begins with celebrated results of J. Thomson, in \cite{thomson}. 

\begin{TTheorem}
Let $\mu$ be a finite, positive Borel measure that is compactly supported in $\mathbb{C}$ and suppose that $1\leq t < \infty$.
Then there is a Borel partition $\{\Delta_i\}_{i=0}^\infty$ of $\mbox{spt}(\mu)$ such that 
\[
 \ P^t(\mu ) = L^t(\mu |_{\Delta_0})\oplus \bigoplus _{i = 1}^\infty P^t(\mu |_{\Delta_i})
 \]
and the following statements are true:

(a) If $i \ge 1$, then $S_{\mu |_{\Delta_i}}$ on $P^t(\mu |_{\Delta_i})$ is irreducible.

(b) If $i \ge 1$ and $U_i :=abpe( P^t(\mu |_{\Delta_i}))$, then $U_i$ is a simply connected region and $\Delta_i\subset \overline{U_i}$ ($\overline {E}$ denotes the closure of a subset $E$).

(c) If $S_\mu$ is pure (that is, $\Delta_0 = \emptyset$) and $\Omega = \text{abpe}(P^t(\mu))$, then the evaluation map $\rho: f\rightarrow f|_\Omega$ is an isometric isomorphism and a weak$^*$ homeomorphism from $P^t(\mu) \cap L^\infty(\mu)$ onto $H^\infty(\Omega)$. 
\end{TTheorem}
\smallskip

The next result in our list is due to J. Brennan in \cite{b08} and J. Conway and N. Elias in \cite{ce93}.

\begin{TBrennan} Let $\mu$ be a finite, positive Borel measure that is compactly supported in a compact subset $K\subset \mathbb{C}$. Suppose that $1\le t < \infty$ and the diameters of components of $\mathbb C\setminus K$ are bounded below. 
Then there is a Borel partition $\{\Delta_i\}_{i=0}^\infty$ of $\mbox{spt}(\mu)$ and compact subsets $\{K_i\}_{i=1}^\infty$ such that $\Delta_i \subset K_i$ for $i \ge 1$,
 \[
 \ R^t(K, \mu ) = L^t(\mu |_{\Delta_0})\oplus \bigoplus _{i = 1}^\infty R^t(K_i, \mu |_{\Delta_i}),
 \]
and the following statements are true:

(a) If $i \ge 1$, then $S_{\mu |_{\Delta_i}}$ on $R^t(K_i, \mu |_{\Delta_i})$ is irreducible.

(b) If  $i \ge 1$ and $U_i :=\mbox{abpe}( R^t(K_i, \mu |_{\Delta_i}))$, then $K_i = \overline{U_i}$.

(c) If  $i \ge 1$, then the evaluation map $\rho_i: f\rightarrow f|_{U_i}$ is an isometric isomorphism and a weak$^*$ homeomorphism from $R^t(K_i, \mu |_{\Delta_i}) \cap L^\infty(\mu |_{\Delta_i})$ onto $H^\infty(U_i)$.
\end{TBrennan}
\smallskip

Let $\mathbb D$ be the open unit disk and $\mathbb T := \partial \mathbb D$. The following remarkable results are due to Aleman, Richter, and Sundberg in \cite{ARS09}.
\smallskip

\begin{ARSTheorem} \label{ARSTheorem}
Suppose that $\mu$ is supported in $\overline{\mathbb D}$, 
$abpe (P^t(\mu )) = \mathbb{D}$, $P^t(\mu )$ is irreducible, and that $\mu (\mathbb{T})> 0$.
\newline
(a) If $f \in P^t(\mu )$, then the nontangential limit $f^*(\zeta )$ of f at $\zeta$ exists a.e. $\mu |_{\mathbb{T}}$ 
and $f^* = f |_{\mathbb{T}}$ as elements of $L^t(\mu |_{\mathbb{T}}).$
\newline
(b) Every nontrivial, closed invariant subspace $\mathcal{M}$ for the shift $S_{\mu}$ on $P^t(\mu )$ has index 1; that is, the dimension
of $\mathcal{M}/z\mathcal{M}$ is one.
\end{ARSTheorem}
\smallskip

In this paper, we extend all above results to the context of $R^t(K, \mu)$ for an arbitrary compact subset $K \subset \mathbb C$ and a finite positive measure $\mu$ supported on $K$.

\bigskip

\section{Concept of non-removable boundary}
\bigskip

We observe that the assumption of the compact subset $K$ in Theorem II is strong due to the following reasons:

(1) In Theorem II, we notice that $K = \overline {\mbox{abpe}(R^t(K,\mu))}$ if $S_\mu$ is pure ($\mu |_{\Delta_0} = 0$). 
Unfortunately, it may happen that $R^t(K,\mu) \ne L^t(\mu)$ and $\mbox{abpe}(R^t(K,\mu)) = \emptyset$. Examples of this phenomenon can be constructed, where $K$ is a Swiss cheese set (with empty interior, see \cite{b71} and \cite{f76}).    

(2) Even for a string of beads set $K$ with $\mbox{abpe}(R^t(K,\mu)) = int(K)$ (understanding that $K = \overline{int(K)}$), Theorem \ref{SOBTheorem} shows that $H^\infty (\mbox{abpe}(R^t(K,\mu)))$ is \quotes{bigger} than  $R^t(K,\mu)\cap L^\infty(\mu)$, while under the assumptions of Theorem II, the algebras $H^\infty (\mbox{abpe}(R^t(K,\mu)))$ and $R^t(K,\mu)\cap L^\infty(\mu)$ ($\mu |_{\Delta_0} = 0$) are isometrically isomorphic. Note: a string of beads set $K$ is the closed unit disk with a sequence of open disks $D_j$ removed such that the closures of the $D_j$’s are mutually disjoint, all their centers lie on the real line $\mathbb R$, and such that $K\cap \mathbb R$ is a (totally disconnected) Cantor set of positive linear Lebesgue measure.

Those examples suggest that the concept of \textit{analytic bounded point evaluation} may not be appropriate for exploring the space $R^t(K,\mu)$. We identify a Borel set  $\mathcal R$ containing $\mbox{abpe}(R^t(K,\mu))$ such that functions in $R^t(K,\mu)$ have some continuity properties related to analytic capacity. The set $\mathcal R$ is called the \textit{removable set} of $R^t(K,\mu)$. The complement $\mathcal F = \mathbb C \setminus \mathcal R$ is called the \textit{non-removable boundary}. Our examples and results evidence that $\mathcal F$ and $\mathcal R$ are applicable concepts for $R^t(K,\mu)$. 

Before stating our definitions and main results, we need to introduce some required notations. Let $E$ be a compact subset of $\mathbb{C}$. We
define the analytic capacity of $E$ by
\begin{eqnarray}\label{GammaDefinition}
\ \gamma(E) = \sup |f'(\infty)|,
\end{eqnarray}
where the supremum is taken over all those functions $f$ that are analytic in $\mathbb C_{\infty} \setminus E$ ($\mathbb C_{\infty} = \mathbb C \cup \{\infty \}$), such that
$|f(z)| \le 1$ for all $z \in \mathbb{C}_\infty \setminus E$; and
$f'(\infty) := \lim _{z \rightarrow \infty} z(f(z) - f(\infty)).$
The analytic capacity of a general subset $F$ of $\mathbb{C}$ is given by: 
 \[
 \ \gamma (F) = \sup \{\gamma (E) : E\subset F \text{ compact}\}.
 \]
Good sources for basic information about analytic
capacity are \cite{Du10}, Chapter VIII of \cite{gamelin}, \cite{Ga72}, Chapter V of \cite{C91}, and \cite{Tol14}.

Let $\nu$ be a finite complex-valued Borel measure that
is compactly supported in $\mathbb {C}$. 
For $\epsilon > 0,$ $\mathcal C_\epsilon(\nu)$ is defined by
\ \begin{eqnarray}\label{CTEDefinition}
\ \mathcal C_\epsilon(\nu)(z) = \int _{|w-z| > \epsilon}\dfrac{1}{w - z} d\nu (w).
\ \end{eqnarray} 
The (principal value) Cauchy transform
of $\nu$ is defined by
\ \begin{eqnarray}\label{CTDefinition}
\ \mathcal C(\nu)(z) = \lim_{\epsilon \rightarrow 0} \mathcal C_\epsilon(\nu)(z)
\ \end{eqnarray}
for all $z\in\mathbb{C}$ for which the limit exists. It follows from 
Corollary \ref{ZeroAC} that \eqref{CTDefinition} is defined for all $z$ except for a set of zero analytic 
capacity. Throughout this paper, the Cauchy transform of a measure always means the principal value of the transform.

Given $d \ge 0$ and $0 < \epsilon \le \infty$, for $A\subset \mathbb C$,
 \[
 \ \mathcal H^d_\epsilon (A) = \inf \left \{\sum_i \text{diam}(A_i)^d:~ A\subset \cup_i A_i,~ \text{diam}(A_i)\le \epsilon \right \}.
 \] 
The $d$-dimensional Hausdorff measure of $A$ is: 
 \[
 \ \mathcal H^d (A) = \sup_{\epsilon >0} \mathcal H^d_\epsilon (A) = \lim _{\epsilon \rightarrow 0} \mathcal H^d_\epsilon (A). 
 \]
For our purposes, we will use one-dimensional Hausdorff measure $\mathcal H^1$. 

We define the non-zero set $\mathcal N(f) = \{\lambda: ~ f(\lambda)\text{ is well defined,} ~ f(\lambda) \ne 0 \}$ and the zero set $\mathcal Z(f)  = \{\lambda: ~ f(\lambda)\text{ is well defined,}  ~ f(\lambda) = 0 \} $.
Write $\gamma-a.a.$, for a property that holds everywhere, except possibly on a set of analytic capacity zero. For $1\le t < \infty$, let $s = \frac{t}{t-1}$. For $g\in L^s(\mu)$, the statement $g\perp R^t(K, \mu)$ (or $g\in R^t(K, \mu)^\perp$) means that $\int r(z) g(z) d \mu(z) = 0 $ for all $r\in \mbox{Rat}(K)$ ($g\mu$ is an annihilating measure). For $\lambda$ in $\mathbb{C}$ and $\delta > 0$, denote $B(\lambda, \delta) = \{z\in\mathbb{C}: |z - \lambda | < \delta\}$. 

Theorem \ref{GPTheorem1} extends Theorem 3.6 in \cite{acy18} of Plemelj's formula for a finite compactly supported complex-valued measure $\nu$. That is, for a Lipschitz graph $\Gamma_0$, the nontangential limit $v^+(\nu, \Gamma_0, \lambda)$ of $\mathcal C(\nu)$ at $\lambda\in \Gamma_0$, in the sense of full $\gamma$ density, exists from top $\mathcal H^1 |_{\Gamma_0}-a.a.$. More accurately, 
 \begin{eqnarray}\label{UNLimit}
 \ \begin{aligned}
 \ & \lim_{\delta\rightarrow 0} \dfrac{\gamma(B(\lambda, \delta)\cap U_{\Gamma_0} \cap \{|\mathcal C(\nu)(z) - v^+(\nu, \Gamma_0, \lambda)| >\epsilon\})}{\delta} \\
 \ = & 0, ~ \mathcal H^1 |_{\Gamma_0}-a.a.
 \ \end{aligned}
 \end{eqnarray}
for all $\epsilon > 0$, where $U_{\Gamma_0} = \{z: ~ Im(z) > y_0, ~ (Re(z), y_0) \in \Gamma_0\}$.  
Similarly, 
 \begin{eqnarray}\label{LNLimit}
 \ \begin{aligned}
 \ &  \lim_{\delta\rightarrow 0} \dfrac{\gamma(B(\lambda, \delta)\cap L_{\Gamma_0} \cap \{|\mathcal C(\nu)(z) - v^-(\nu, \Gamma_0, \lambda)| >\epsilon\})}{\delta} \\
 \ = & 0, ~ \mathcal H^1 |_{\Gamma_0}-a.a.
 \ \end{aligned}
 \end{eqnarray}
for all $\epsilon > 0$, where $L_{\Gamma_0} = \{z: ~ Im(z) < y_0, ~ (Re(z), y_0) \in \Gamma_0\}$.

Lemma \ref{GammaExist} states that 
there is a sequence of Lipschitz functions $A_n: \mathbb R\rightarrow \mathbb R$ with $\|A_n'\|_\infty \le \frac{1}{4}$ and its (rotated) graph $\Gamma_n$ such that if 
 \begin{eqnarray}\label{RNDecom}
 \ \Gamma = \cup_n \Gamma_n\text{ and }\mu = h\mathcal H^1 |_{\Gamma} + \mu_s
 \end{eqnarray}
 is the Radon-Nikodym decomposition with respect to $\mathcal H^1 |_{\Gamma}$, where $h\in L^1(\mathcal H^1 |_{\Gamma})$ and $\mu_s\perp \mathcal H^1 |_{\Gamma}$, then 
 \[ 
 \ \lim_{\delta\rightarrow 0}\dfrac{\mu(B(\lambda, \delta))}{\delta} = 0, ~ \gamma |_{\mathbb C \setminus \mathcal N(h)}-a.a.
 \]
 This allows us to use $v^+(g\mu, \Gamma_n, \lambda)$ and $v^-(g\mu, \Gamma_n, \lambda)$ ($g\perp R^t(K, \mu)$) in our definition of $\mathcal F$ and $\mathcal R$.  

Essentially, in the proof of Theorem I for the existence of $\lambda \in \text{abpe} (P^t(\mu)$, Thomson constructed a barrier of squares (around $\lambda$) on which $|\mathcal C(g\mu)(z)|$ is \quotes{bounded below} for some $g\perp P^t(\mu)$. In other words, for $\lambda \in \mathbb C$ and $g\perp P^t(\mu)$, if $|\mathcal C(g\mu)(z)|$ is \quotes{small} on a sizable subset $E_{\delta} \subset B(\lambda, \delta)$ ($\gamma(E_{\delta})$ is comparable with $\delta$) for all $\delta > 0$, then $\lambda \notin \text{abpe} (P^t(\mu)$. The fact motivates us to define the non-removable boundary $\mathcal F$ for $R^t(K, \mu)$ as the collection of such points. More precisely, for a dense sequence $\{g_j\}_{j=1}^\infty\subset R^t(K, \mu)^\perp$, $\mathcal F$ consists of three sets:

(1) $\lambda \in \mathbb C\setminus \mathcal N(h)$ (as in \eqref{RNDecom}) and $\mathcal C(g_j\mu)(\lambda) = 0$ for all $j\ge 1$. The set $E_{\delta}$ can be constructed from Lemma \ref{CauchyTLemma};

(2) $\lambda \in \Gamma_n \cap \mathcal N(h)$ and $v^+(g_j\mu, \Gamma_n, \lambda) = 0$  for all $j\ge 1$ (though $\mathcal C(g_j\mu)(\lambda)$ may not be zero for some $g_j$). The set $E_{\delta}$ can be constructed from \eqref{UNLimit}; and 

(3) $\lambda \in \Gamma_n\cap \mathcal N(h)$ and $v^-(g_j\mu, \Gamma_n, \lambda) = 0$  for all $j\ge 1$ (though $\mathcal C(g_j\mu)(\lambda)$ may not be zero for some $g_j$). The set  $E_{\delta}$ can be constructed from \eqref{LNLimit}.
\smallskip

The detailed constructions of $E_{\delta}$ are discussed in section 3.3 (Theorem \ref{acTheorem}).
Now we formulate our definition of $\mathcal F$ as the following.    

\begin{definition}\label{NRBDef}
Let  $\{g_n\}_{n=1}^\infty \subset R^t(K,\mu) ^\perp$ be a dense subset. 
Let $\Gamma_n$, $h$, $v^+(\nu, \Gamma_n, \lambda)$, and $v^-(\nu, \Gamma_n, \lambda)$ be as in \eqref{RNDecom}, \eqref{UNLimit}, and \eqref{LNLimit}, respectively. 
Define,
 \begin{eqnarray}\label{FZeroEq}
 \ \mathcal F_0 = \bigcap_{j=1}^\infty \mathcal Z(\mathcal C(g_j\mu)),
 \end{eqnarray}
 \begin{eqnarray}\label{FPlusEq}
 \ \mathcal F_+ = \bigcup_{n \ge 1} \bigcap_{j=1}^\infty \mathcal Z(v^+(g_j\mu, \Gamma_n, . )) \cap \mathcal N(h),
 \end{eqnarray}
 \begin{eqnarray}\label{FMinusEq}
 \ \mathcal F_- = \bigcup_{n \ge 1} \bigcap_{j=1}^\infty \mathcal Z(v^-(g_j\mu, \Gamma_n, . )) \cap \mathcal N(h),
 \end{eqnarray}
and
 \[
 \ \mathcal F = \mathcal F_0 \cup \mathcal F_+ \cup \mathcal F_-.
 \]  
$\mathcal F$ is called the \textit{non-removable boundary} of $R^t(K,\mu)$. The set $\mathcal R = K \setminus \mathcal F$ is called the \textit{removable set} of $R^t(K, \mu)$. The set $\mathcal R_B = \mathcal R \setminus \mbox{abpe}(R^t(K,\mu))$ is called the \textit{removable boundary} of $R^t(K, \mu)$.
\end{definition}
\smallskip

It is proved in Theorem \ref{FCharacterization} and Corollary \ref{NRBUnique} that $\mathcal F$, $\mathcal R$, and $\mathcal R_B$ are  independent of choices of $\{\Gamma_n\}_{n=1}^\infty$ and $\{g_n\}_{n=1}^\infty$ up to a set of zero analytic capacity. Clearly, $\mu |_{\mathcal F_+ \cup \mathcal F_-}$ is absolutely continuous with respect to $\mathcal H^1|_{\Gamma}$. Theorem \ref{FRProperties} shows that 
 \[
 \ \mathcal F_0 \cap (\mathcal F_+ \cup \mathcal F_-) = \emptyset,~\gamma-a.a.\text{ and }\lim_{\delta\rightarrow 0}\frac{\mu(B(\lambda, \delta))}{\delta} = 0, ~ \gamma |_{\mathcal F_0}-a.a.
 \] 

Using the classical Hardy space $H^2(\mathbb T)$ as a simple example, we get $\mathcal F_0 = \mathbb C \setminus \overline {\mathbb D}$, $\mathcal F_+ \cup \mathcal F_- = \mathbb T$ since non-tangential limits of $\mathcal C(g_j\mu)$ ($=0$ on $\mathbb C \setminus \overline {\mathbb D}$) from outside of $\overline {\mathbb D}$ are zero, $\mathcal R = \mathbb D$, and $\mathcal R_B = \emptyset$. For special cases as in Theorem I and Theorem II, as a consequence of Proposition \ref{NFSetIsBig} (3) and Theorem \ref{ABPETheorem}, $\mathcal F = \mathbb C \setminus \mbox{abpe}(R^t(K,\mu))$, $\mathcal R = \mbox{abpe}(R^t(K,\mu))$, and $\mathcal R_B = \emptyset$. Proposition \ref{SOBRemovable} provides examples of string of beads sets $K$ such that the space $R^t(K,\mu)$ has rich non-trivial removable boundaries ( $K\cap\mathbb R \supset \mathcal R_B \ne \emptyset$ and $\mathcal R = \mathcal R_B \cup \text{int}(K)$). For a Swiss cheese set $K$ ($\text{int}(K) = \emptyset$) as in Example \ref{FCExample}, if $\partial _o K = \cup_n \partial U_n$ (outer boundary of $K$), where $\{U_n\}$ are connected components of $\mathbb C \setminus K$, then $\mathcal F = K^c \cup \partial _o K$ and $\mathcal R = \mathcal R_B = K \setminus \partial _o K$. In this case, $\mathcal R$ does not have interior and $\mbox{abpe}(R^t(K,\mu)) = \emptyset$.  

In Proposition \ref{NFSetIsBig}, we prove $K = \overline{\mathcal R}$ if $S_\mu$ is pure (though $\mbox{abpe}(R^t(K,\mu))$ may be empty). It is also shown that, in Theorem \ref{DensityCorollary} and Theorem \ref{ABPETheorem}, 
 \[
 \ \lim_{\delta\rightarrow 0 }\dfrac{\gamma(B(\lambda_0, \delta)\setminus \mathcal R)}{\delta} = 0,~\gamma|_{\mathcal R}-a.a.
 \]
 and  
 \[ 
 \ \mathcal R \cap \text{int}(K) \approx \mbox{abpe}(R^t(K,\mu)), ~ \gamma-a.a.
 \]
\bigskip

\section{Nontangential limits}
\bigskip

It is well known that for every function $f\in R^t(K,\mu)$, if $\rho(f)(\lambda )$ denotes the point evaluation functional at $\lambda \in \text{abpe}(R^t(K,\mu))$, then $\rho(f)$ is an analytic function on $\text{abpe}(R^t(K,\mu))$ and $\rho(f)(\lambda ) = f(\lambda ), ~ \mu |_{\text{abpe}(R^t(K,\mu))}-a.a.$. Clearly, $\frac{f(z)-\rho(f)(\lambda)}{z - \lambda}\in R^t(K,\mu)$. Hence,
 \[
 \ \int \dfrac{f(z)-\rho(f)(\lambda)}{z - \lambda} g(z) d \mu(z) = 0, ~\lambda \in \text{abpe}(R^t(K,\mu)), ~g\perp R^t(K,\mu). 
 \]
Consequently,
 \begin{eqnarray}\label{NTIntroEq1}
 \ \begin{aligned}
 \ &\rho(f)(\lambda)\mathcal C(g\mu)(\lambda) = \mathcal C(fg\mu)(\lambda),~ \gamma |_{\text{abpe}(R^t(K,\mu))}-a.a.,\\
 \ &g\perp R^t(K,\mu). 
 \ \end{aligned} 
\end{eqnarray}
We expand the domain of the identity \eqref{NTIntroEq1} to $\mathcal R$ and show that $\rho(f)$ has some continuity property on $\mathcal R$ (as in Definition \ref{GDefInto} below). The expansion is important because $\mathcal R \supset\mbox{abpe}(R^t(K,\mu))$ and $K = \overline{\mathcal R}$ (under the condition that $S_\mu$ is pure) as shown in Proposition \ref{NFSetIsBig} regardless of whether $\mbox{abpe}(R^t(K,\mu)) = \emptyset$ or not. It is also essential for us to describe $R^t(K,\mu)$ as in Main Theorem II. 
\smallskip

\begin{definition}\label{GDefInto}
A function $f$ is $\gamma$-continuous at $\lambda$ if $f(z)$ is well defined at $\lambda$ and on $B(\lambda, \delta_0)\setminus \mathbb Q$ with $\gamma(\mathbb Q) = 0$ for some $\delta_0 > 0$ satisfying
 \[
 \ \lim_{\delta\rightarrow 0 }\dfrac{\gamma (B(\lambda, \delta) \cap \{|f(z) - f(\lambda)| > \epsilon\})}{\delta} = 0
 \]
for all $\epsilon > 0$.
\end{definition}
\smallskip

We establish our first main result as the following.    

\begin{MTheorem}\label{MTheoremII}
If $S_\mu$ on $R^t(K, \mu )$ is pure, then the following properties hold:

(a) For $f\in R^t(K,\mu)$, there exists a unique $\rho(f)$ defined on $\mathcal R,~ \gamma-a.a.$ satisfying:
\[
\ \rho(f)(z) \mathcal C(g\mu) (z) = \mathcal C(fg\mu) (z),~ z\in\mathcal R,~\gamma-a.a. 
\]
for all $g\perp R^t(K,\mu)$.
Furthermore, there is a subset $\mathbb Q_f\subset \mathbb C$ with $\gamma(\mathbb Q_f) = 0$ such that $\rho(f)$ is $\gamma$-continuous at each $\lambda\in \mathcal R \setminus \mathbb Q_f$.    

(b) (Nontangential limits in full analytic capacitary density) For $f\in R^t(K,\mu)$, there exists  a subset $\mathbb Q_f$ with $\gamma(\mathbb Q_f) = 0$ such that for $\lambda \in \mathcal F _+\cap\Gamma_n \setminus \mathbb Q_f$, 
 \[
 \ \lim_{\delta\rightarrow 0 }\dfrac{\gamma (B(\lambda, \delta) \cap L_{\Gamma_n} \cap \{|\rho(f)(z) - f(\lambda)| > \epsilon\})}{\delta} = 0
 \]
and
 for $\lambda \in \mathcal F _- \cap\Gamma_n\setminus \mathbb Q_f$, 
 \[
 \ \lim_{\delta\rightarrow 0 }\dfrac{\gamma (B(\lambda, \delta) \cap U_{\Gamma_n} \cap \{|\rho(f)(z) - f(\lambda)| > \epsilon\})}{\delta} = 0
 \] 
for all $\epsilon > 0$, where $\Gamma_n$ is as in \eqref{RNDecom}. 
\end{MTheorem}
\smallskip 

As applications of Main Theorem I, in case that $\mbox{abpe}(R^t(K, \mu))$ is adjacent to one of these Lipschitz
graphs, one can consider usual nontangential limits and obtain Theorem \ref{NonTL}, which generalizes Theorem III (a) to the space $R^t(K, \mu)$. Theorem III (b) can also be extended to $R^t(K, \mu)$ as in Theorem \ref{IndexForIS}.

\bigskip

\section{Decomposition of $R^t(K,\mu)$ and the algebra $R^t(K,\mu)\cap L^\infty (\mu)$}
\bigskip

The main part of this paper is to examine the structure of $R^t(K, \mu)$. Our main theorem II below extends Theorem I and Theorem II to $R^t(K, \mu)$ for an arbitrary compact subset $K$ and a finite positive measure $\mu$ with $\text{spt}(\mu) \subset K$.
\smallskip 

For a Borel set $F$, define 
 \[
 \ H^\infty_F = \{f\in L^\infty(\mathbb C): ~ f \text{ is bounded analytic on } \mathbb C \setminus E_f,\text{ compact subset }E_f\subset F\}.
 \]
Let $\mathcal L^2_F$ denote the planar Lebesgue measure restricted to $F$.
\smallskip

\begin{definition}\label{hSpace}
The algebra $H^\infty (\mathcal R)$ is the weak$^*$ closure in $L^\infty(\mathcal L^2_{\mathcal R})$ of $H^\infty_{\mathcal F}$.
\end{definition}
\smallskip

We observe that functions in $H^\infty (\mathcal R)$ may not be analytic on open subsets as $\mathcal R$ may not have interior (e.g. a Swiss cheese set $K$, see Example \ref{algEqExample}). We now state our second main theorem.

\begin{MTheorem2}\label{DecompositionTheorem} 
There exists a Borel partition $\{\Delta_n\}_{n\ge 0}$ of $\text{spt}(\mu )$ and compact subsets $\{K_n\}_{i=1}^\infty$ such that $\Delta_n \subset K_n$ for $n \ge 1$,
 \[
 \ R^t(K,\mu) = L^t(\mu |_{\Delta_0})\oplus \bigoplus_{n=1}^\infty R^t(K_n, \mu |_{\Delta_n}),
 \]
and the following statements are true: 

(1) If $n \ge 1$, then $S_{\mu |_{\Delta_n}}$ on $R^t(K_n, \mu |_{\Delta_n})$ is irreducible. 

(2) If $n \ge 1$, $m \ge 1$, $n\ne m$, and $\mathcal F_n$ and $\mathcal R_n$ are the non-removable boundary and removable set for $R^t(K_n, \mu |_{\Delta_n})$, respectively, then $K_n = \overline{\mathcal R_n}$ and $K_m \subset \mathcal F_n, ~\gamma-a.a$. 

(3) If $S_\mu$ is pure ($\mu |_{\Delta_0} = 0$), then the map $\rho$ is an isometric isomorphism and a weak$^*$ homeomorphism from $R^t(K, \mu) \cap L^\infty (\mu)$ onto $H^\infty (\mathcal R)$. 
\end{MTheorem2}
\smallskip

Theorem \ref{SOBTheorem} (for string of beads sets) is an easy conclusion of Main Theorem II, which implies a certain possible splitting of the space
$R^t(K, \mu)$ does not occur under $\mathcal R_B \ne \emptyset$. In section 2.2, we use Theorem \ref{SOBTheorem} as an example to exhibit the ideas of proving our main theorems. 

Corollary \ref{DecompCorollary1} and \ref{DecompCorollary2} are applications of Main Theorem II, where the boundary of each connected component of $\text{abpe}(R^t(K,\mu))$ is a subset of $\mathcal F$. In section 7.2, we apply our Main Theorem II to the cases in which $\partial (\text{abpe}(R^t(K,\mu))) \cap \mathcal R \ne \emptyset, ~\gamma-a.a.$.   

Finally, we mention some irreducibility properties of $S_\mu$. It is shown in Theorem \ref{Lemma3} that if $S_\mu$ is irreducible, then $\underset{\delta \rightarrow 0}{\overline{\lim}}\frac{\gamma(B(\lambda, \delta)\setminus \mathcal F)}{\delta} > 0$ for $\lambda \in K$.  Theorem \ref{IrreducibilityTheorem} proves that $S_\mu$ is irreducible if and only if the removable set $\mathcal R$ is $\gamma$-connected (see Definition \ref{gammaCDefinition}).

\bigskip

\section{Historical notes}

\bigskip

D. Sarason \cite{s72} has characterized $P^\infty (\mu)$, the weak$^*$ closure in $L^\infty(\mu)$ of
the polynomials. In that case, there is a decomposition similar to Thomson's Theorem (Theorem I) and each analytic summand is the space of bounded analytic functions on the set $U_j$, which turns out to be a special kind of simply connected region (see Page 398 of \cite{c81}).

Polynomial and rational approximation in the mean has been studied extensively. Before Thomson published his remarkable results in \cite{thomson}, there were several papers such as J. Brennan \cite{B79}, S. Hruscev \cite{H79}, T. Trent \cite{Tr79}, and \cite{Tr90}, etc. As we have already pointed out, Theorem II is due to J. Conway and N. Elias \cite{ce93} and J. Brennan \cite{b08}. For a compactly supported complex Borel measure $\nu$ of $\mathbb C$, by estimating
analytic capacity of the set $\{\lambda: |\mathcal C(\nu )| > c \}$, J. Brennan \cite{B06}, A. Aleman, S. Richter, and C. Sundberg \cite{ARS09} and \cite{ARS10} provide interesting
alternative proofs of Thomson's theorem for the existence of analytic bounded point evaluations for mean polynomial approximation. Both their proofs rely on X. Tolsa's deep results on analytic
capacity. There are other related research papers recently. For example, J. Brennan and E. Militzer \cite{BM11}, L. Yang \cite{Y16} and \cite{Y18}, etc. 

The paper \cite{Y95b} suggested a connection between the index of nonzero invariant subspaces and the existence of boundary values of functions. Later, in \cite{cy98}, the authors conjectured Theorem III (b). Before A. Aleman, S. Richter, and C. Sundberg published their remarkable paper \cite{ARS09} which proved Theorem III,  
this conjecture had been supported by numerous partial results by various
authors, both prior to and subsequent to \cite{cy98}.
We mention here the work of J. Akeroyd \cite{A01}, \cite{A02}, \cite{A03}, A. Aleman
and S. Richter \cite{AR97}, T. Miller and R. Smith \cite{MS90}, T. Miller, W.
Smith, and L. Yang \cite{MSY99}, R. Olin and J. Thomson \cite{OT80}, J
Thomson and L. Yang \cite{TY95}, T. Trent \cite{Tr79}, Z. Wu and L. Yang \cite{WY98}, and L. Yang \cite{Y95a}.

Finally, we mention that the paper \cite{acy18} extends Theorem III to certain $R^t(K, \mu)$. Moreover, it provides a simpler proof of Theorem III. 
 
\bigskip

\chapter{Preliminaries}
\bigskip

\section{Analytic capacity and Cauchy transform}
\bigskip

For a finite complex-valued Borel measure $\nu$ with compact support in $\mathbb C$, $\mathcal C_\epsilon (\nu)$ is defined as in \eqref{CTEDefinition} and (principal value) Cauchy transform $\mathcal C(\nu)(z) = \lim_{\epsilon \rightarrow 0}\mathcal C_\epsilon (\nu)(z)$ whenever the limit exists as  in \eqref{CTDefinition}. It is well known that in the sense of distribution,
 \begin{eqnarray}\label{CTDistributionEq}
 \ \bar \partial \mathcal C(\nu) = - \pi \nu.
 \end{eqnarray} 
The maximal Cauchy transform is defined
 \[
 \ \mathcal C_*(\nu)(z) = \sup _{\epsilon > 0}| \mathcal C_\epsilon(\nu)(z) |
 \]
and the maximal function of $\nu$ is defined
 \[
 \ \mathcal M_\nu(z) = \sup _{\epsilon > 0}\dfrac{|\nu|(B(z,\epsilon))}{\epsilon}.
 \]

Analytic capacity $\gamma$ is defined as in \eqref{GammaDefinition}. A related capacity, $\gamma _+,$ is defined for subsets $E$ of $\mathbb{C}$ by:
\[
\ \gamma_+(E) = \sup \|\eta \|,
\]
where the supremum is taken over positive measures $\eta$ with compact support
contained in $E$ for which $\|\mathcal{C}(\eta) \|_{L^\infty (\mathbb{C})} \le 1.$ 
Since $\mathcal C\eta$ is analytic in $\mathbb{C}_\infty \setminus \mbox{spt}(\eta)$ and $(\mathcal{C}(\eta)'(\infty) = \|\eta \|$, 
we have:
$\gamma _+(E) \le \gamma (E)$
for all compact subsets $E$ of $\mathbb{C}$. 

Given three pairwise different points $x, y, z \in \mathbb C$, their Menger curvature is
\[
\ c(x, y, z) = \dfrac{1}{R(x, y, z)},
 \]
where $R(x, y, z)$ is the radius of the circumference passing through $x, y, z$ (with $R(x, y, z) = \infty,~ c(x, y, z) = 0$ if $x, y, z$ lie on a same line). If two among these points coincide, we let $c(x, y, z) = 0$. For a finite positive measure $\eta$, we set
 \[
 \ c^2_\eta(x) = \int\int c(x,y,z)^2 d\eta(y)d\eta(z)
 \]
and we define the curvature of $\eta$ as
 \[
 \ c^2(\eta ) = \int c^2_\eta (x)d\eta)(x) = \int\int\int c(x,y,z)^2 d\eta(x)d\eta(y)d\eta(z).
 \] 

For a finite complex-valued measure $\nu$, define
 \[
 \ \Theta_\nu^* (\lambda ) := \underset{\delta\rightarrow 0}{\overline\lim} \dfrac{|\nu |(B(\lambda , \delta ))}{\delta}\text{ and }\Theta_\nu (\lambda ) := \lim_{\delta\rightarrow 0} \dfrac{|\nu |(B(\lambda , \delta ))}{\delta}\text{ if the limit exists.}
 \]

A finite positive measure $\eta$ supported in $E$ is $c$-linear growth if $\eta(B(\lambda, \delta)) \le c\delta$ for $\lambda\in \mathbb C$, denoted $\eta\in \Sigma(E)$ if $c=1$. In addition, if 
$\Theta_\eta(\lambda) = 0$, we say $\eta\in \Sigma_0(E)$. The Cauchy transform $\mathcal C\eta$ is bounded from $L^2(\eta)$ to $L^2(\eta)$ if
 \[
 \ \|\mathcal C_\epsilon (f \eta )\|_{L^2(\eta)} \le C \|f\|_{L^2(\eta)} 
 \]
for some constant $C>0$ and all $\epsilon > 0$. The operator norm is denoted by $\|\mathcal C\eta\|_{L^2(\eta)\rightarrow L^2(\eta)}$.
X. Tolsa has established the following astounding results.
\smallskip

\begin{theorem} \label{TTolsa} (Tolsa's Theorem)

(1) $\gamma_+$ and $\gamma$ are actually equivalent. 
That is, there are absolute positive constants $a_T$ and $A_T$ such that 
\begin{eqnarray}\label{GammaEq}
\ \gamma (E) \le A_ T \gamma_+(E),
\end{eqnarray}
\begin{eqnarray}\label{GammaEq1}
\ a_ T \gamma(E) \le \sup \{\eta(E):~ \eta\in \Sigma(E),~ \sup_{\epsilon > 0}\|\mathcal C_\epsilon (\eta )\|_{L^\infty (\mathbb C)} \le 1\} \le  A_ T \gamma(E),
\end{eqnarray}
\begin{eqnarray}\label{GammaEq2}
\ a_ T \gamma(E) \le \sup \{\eta(E):~ \eta\in \Sigma(E),~ \|\mathcal C\eta\|_{L^2(\eta)\rightarrow L^2(\eta)} \le 1\} \le A_ T \gamma(E),
\end{eqnarray}
and
\begin{eqnarray}\label{GammaEq3}
\ a_ T \gamma(E) \le \sup \{\eta(E):~ \eta\in \Sigma(E),~ c^2(\eta) \le \eta(E) \} \le A_ T \gamma(E),
\end{eqnarray}
for all $E \subset \mathbb{C}.$ 

(2) Semiadditivity of analytic capacity:
\begin{eqnarray}\label{Semiadditive}
\ \gamma \left (\bigcup_{i = 1}^m E_i \right ) \le A_T \sum_{i=1}^m \gamma(E_i)
\end{eqnarray}
where $E_1,E_2,...,E_m \subset \mathbb{C}$ and $m$ could be $\infty$.

(3) There is an absolute positive constant $C_T$ such that, for any $a > 0$, we have:  
\begin{eqnarray}\label{WeakOneOne}
\ \gamma(\{\mathcal{C}_*(\nu)  \geq a\}) \le \dfrac{C_T}{a} \|\nu \|.
\end{eqnarray}  
\end{theorem}

\begin{proof}
(1) and (2) are from \cite{Tol03} (or see Theorem 6.1 and Corollary 6.3 in \cite{Tol14}). \eqref{GammaEq1}, \eqref{GammaEq2}, and \eqref{GammaEq3} follow from Theorem 4.14 of \cite{Tol14}.
(3) follows from Proposition 2.1 of \cite{Tol02} (or see \cite{Tol14} Proposition 4.16).
\end{proof}
\smallskip

The following corollary is a simple application of \cite{To98} (or see Theorem 8.1 in \cite{Tol14} or Corollary 3.1 in \cite{acy18}).

\begin{corollary}\label{ZeroAC}
Suppose that $\nu$ is a finite, complex-valued Borel measure with compact support in $\mathbb{C}$. Then there exists 
$\mathbb Q \subset \mathbb{C}$ with $\gamma(\mathbb Q) = 0$ such that $\lim_{\epsilon \rightarrow 0}\mathcal{C} _{\epsilon}(\nu)(z)$ 
exists for $z\in\mathbb{C}\setminus \mathbb Q$.
\end{corollary}
\smallskip

Let $\mathbb R$ be the real line. Let $A:\mathbb R\rightarrow \mathbb R$ be a Lipschitz function. The Lipschitz graph $\Gamma$ of $A$ is defined by 
 $\Gamma = \left \{(x,A(x)):~x\in \mathbb R \right \}$.
We will use $C_1, C_2, C_3, ...$ and $c_1, c_2, c_3,...$ to stand for absolute constants that may change from one step to the next.
For a finite complex-valued measure $\nu$ with compact support on $\mathbb C$, define zero and non zero linear density set as the following:
 \begin{eqnarray}\label{ZNDensity}
 \ \mathcal{ZD}(\nu) = \{ \lambda:~ \Theta_\nu (\lambda ) = 0 \}, ~ \mathcal{ND}(\nu) =  \{ \lambda:~ \Theta^*_\nu (\lambda ) > 0 \}.
 \end{eqnarray}
Set $\mathcal{ND}(\nu, n) =  \{ \lambda:~ \frac{1}{n} \le \Theta^*_\nu (\lambda ) \le n \}$.
The following results are used throughout this paper and we list them here as a lemma for the convenience of the reader.

\begin{lemma}\label{lemmaBasic0}
(1) Let $\eta$ be a compactly supported finite positive measure with $\eta\in \Sigma(\text{spt}(\eta))$. Suppose that $\mathcal C\eta$ is bounded on $L^2(\eta)$. Then for a Borel set $E\subset \mathbb C$, there exists some function $h$ supported on $E$, with $0\le h \le 1$ such that
 \[
 \ \eta (E) \le 2 \int h d\eta, ~ \|\mathcal C _\epsilon (h\eta)\|_{ L^\infty (\mathbb C)} \le c
\]
for all $\epsilon > 0$, where the constant $c$ depends only on the $L^2(\eta)$ norm of $\mathcal C\eta$.

(2) If $E\subset \mathbb C$ is a bounded Borel set, then $\mathcal L^2(E)\le 4 \pi \gamma(E)^2$.

(3) Let $p>0$ be an integer. Let $\{E_n\}$ be a family of subsets of bounded set $E$ such that every $B(z,\gamma(E))$ meets at most $p$ of $\{E_n\}$. Then 
 \[
 \ \sum_n \gamma(E_n) \le 100 p \gamma(E).
 \]

(4) If $E$ is a subset of $\mathbb C$ with $\mathcal H^1 (E) <\infty$, then
 \begin{eqnarray}\label{RectifiableEq}
 \ E \subset \mathbb Q \cup \bigcup_{n=1}^\infty \Gamma_n
 \end{eqnarray}
where $\gamma(\mathbb Q) = 0$ and $\Gamma_n$ is a (rotated) Lipschitz graph with its Lipschitz function $A_n$ satisfying $\|A'_n\|\le \frac{1}{4}$.

(5) Let $\mu$ be a finite positive measure with compact support in $\mathbb C$. For $\lambda_2 \ge \lambda_1 > 0$, let $E$ and $F$ be two bounded Borel sets such that 
 \[
 \ E \subset \{z:~ \Theta^*_\mu (z) \ge \lambda_1 \} \text{ and } F \subset \{z:~ \lambda _1 \le \Theta^*_\mu (z) \le \lambda_2 \}.
 \]
Then there exist some absolute constants $c_1$, $c_2$, and $c_3$ such that $\mathcal H^1(E) \le c_1 \frac{\mu(E)}{\lambda_1}$ and $\mu |_F = g \mathcal H^1|_F$, where $g$ is some Borel function such that $c_2\lambda_1 \le g(z) \le c_3 \lambda_2,~ \mathcal H^1|_F-a.a.$. Consequently, the following statements are true:

(a) If $\mathbb Q_\mu = \{z:~ \Theta_\mu^*(z) = \infty\}$, then $\gamma(\mathbb Q_\mu) = \mathcal H^1(\mathbb Q_\mu) = 0$, 
 \[
 \ \mathcal{ND}(\mu) = \bigcup_n \mathcal{ND}(\mu, n) \cup \mathbb Q_\mu,
 \]
 $\mu |_{\mathcal{ND}(\mu)} = g \mathcal H^1|_{\mathcal{ND}(\mu)}$, and $c_2\frac{1}{n} \le g(z) \le c_3 n,~ \mathcal H^1|_{\mathcal{ND}(\mu, n)}-a.a.$.

(b) Let $\eta$ be a compactly supported finite positive measure on $\mathbb{C}$ with $c$-linear growth. Let $\nu$ be a finite, complex-valued Borel measure with compact support in $\mathbb{C}$. If $\eta\perp|\nu|$, then $\eta(\mathcal{ND}(\nu)) = 0$. 

(6) Let $E\subset \mathbb C$ be compact with $\mathcal H^1 (E) <\infty$, and let $f : \mathbb C \setminus E\rightarrow \mathbb C$ be analytic such that $\|f\|_\infty \le 1$ and $f(\infty )=0$. Then there is a measure $\nu = b \mathcal H^1 |_E$, where $b$ is a measurable function with $|b(z)| \le 1$ for all $z\in E$, such that $f(z) = \mathcal C(\nu )(z)$ for all $z \notin E$.

(7) Let $A:\mathbb R\rightarrow \mathbb R$ be a Lipschitz function with graph $\Gamma$. Then the Cauchy transform of the measure $\mathcal H^1 |_\Gamma$ is bounded on $L^2(\mathcal H^1 |_\Gamma)$, while the norm only depends on $\|A'\|_\infty$. Hence, by Theorem \ref{TTolsa} (1), there are constants $c_\Gamma, C_\Gamma > 0$ that depend only on $\|A'\|_\infty$ such that for $E\subset \Gamma$,
 \begin{eqnarray}\label{HACEq}
 \ c_\Gamma \mathcal H^1 |_\Gamma (E) \le \gamma (E) \le C_\Gamma \mathcal H^1 |_\Gamma (E). 
 \end{eqnarray}

(8) Let $\eta$ be a finite, positive Borel measure on $\mathbb{C}$ with $\eta\in \Sigma(\text{spt}(\eta))$. Then, there exists an absolute constant $C_M > 0$, for any finite complex-valued Borel measure $\nu$ with compact support in $\mathbb{C}$, 
 \[
 \ \eta\{\lambda:~ \mathcal M_\nu (\lambda) > a\} \le \dfrac{C_M}{a}\|\nu\|.
 \]
Consequently, $\gamma\{\lambda:~ \mathcal M_\nu (\lambda) > a\} \le \frac{C_M}{a}\|\nu\|$.

(9) Let $\eta$ be a finite positive measure such that $\eta\in \Sigma(\text{spt}(\eta))$ and $\|\mathcal C_\epsilon (\eta)\| \le 1$. If $\gamma(E) = 0$, then $\eta(E) = 0$.

(10) Let $\{\nu_j\}$ be a sequence of finite complex-valued measures with compact supports. Then for $\epsilon > 0$, there exists a Borel subset $F$ such that $\gamma (F^c) < \epsilon$ and $\mathcal C_*(\nu_j)(z), ~\mathcal M_{\nu_j}(z ) \le M_j < \infty$ for $z \in F$.
\end{lemma}

\begin{proof}
(1) follows from Lemma 4.7 in \cite{Tol14}.
See Theorem 2.3 on page 200 for (2) and Theorem 2.7 on page 202 for (3) in \cite{gamelin}. 
Combining Theorem 1.26 in \cite{Tol14}, David’s Theorem (\cite{D98} or Theorem 7.2 in \cite{Tol14}), Proposition 4.13 in \cite{Tol14}, and the proof of Lemma 4.11 in \cite{Tol14}, we obtain (4).

(5) follows from Lemma 8.12  on page 307 in \cite{Tol14}. Notice that 
 \[
 \ \mathcal H^1(\mathbb Q_\mu) \le \mathcal H^1\{z:~ \Theta^*_\mu (z) \ge n\} \le \dfrac{c_1}{n}\|\mu\|.
 \]
So (5)(a) follows from (5). For (5) (b): By (5), $\mathcal H^1(\mathcal{ND}(\eta,m)\cap \mathcal{ND}(\nu, n)) = 0$ as $\eta\perp|\nu|$. Therefore, $\mathcal H^1(\mathcal{ND}(\eta)\cap \mathcal{ND}(\nu)) = 0$. On the other hand, $\mathcal H^1(\mathcal{ND}(\nu, n)) < \infty$, which implies $\eta |_{\mathcal{ND}(\nu)}$ is absolutely continuous with respect to $\mathcal H^1  |_{\mathcal{ND}(\nu)}$ since $\eta$ is linear growth. Hence, $\eta(\mathcal{ZD}(\eta) \cap \mathcal{ND}(\nu)) = 0$, which implies $\eta(\mathcal{ND}(\nu)) = 0$.    

For (6), see Proposition 6.5 in \cite{Tol14}.
For (7), see \cite{CMM82} or Theorem 3.11 in \cite{Tol14}.
For (8), see Theorem 2.6 in \cite{Tol14} and Theorem \ref{TTolsa} (1).

(9): Suppose $\eta(E) > 0$. By Proposition 3.3 in \cite{Tol14}, $c^2(\eta) < \infty$. If $\eta_0 = \sqrt{\frac{\eta(E)}{c^2(\eta)}}\eta|_E$, then  
 \[
	 \ c^2(\eta_0) \le \left ( \sqrt{\dfrac{\eta(E)}{c^2(\eta)}} \right )^3 c^2(\eta) = \eta_0(E).   
	 \]
	By Theorem \ref{TTolsa} (1) \eqref{GammaEq3}, we get $\eta_0 (E) \le A_T \gamma (E)$, which contradicts to $\eta(E) > 0$.
	
	The proof of (10) is an application of (8) and Theorem \ref{TTolsa} (2), (3). In fact, let $A_j = \{\mathcal C_*(\nu_j)(z) \le M_j\}$ and $B_j = \{\mathcal M_{\nu_j}(z ) \le M_j\}$. By Theorem \ref{TTolsa} (3) and (8), we can select $M_j>0$ so that $\gamma(A_j^c) < \frac{\epsilon}{2^{j+2}A_T}$ and $\gamma(B_j^c) < \frac{\epsilon}{2^{j+2}A_T}$. Set $F = \cap_{j=1}^\infty (A_j \cap B_j)$. Then applying Theorem \ref{TTolsa} (2), we get
\[
 \ \gamma (F^c) \le A_T \sum_{j=1}^\infty (\gamma(A_j^c) + \gamma(B_j^c)) < \epsilon.
 \] 
\end{proof}

\bigskip

\section{Examples: String of beads}
\bigskip

Our ideas of proving our main theorems stemmed from careful analysis of a special set such as a string of beads set. Therefore, in this section, we use the set as an example to illustrate our ideas. 

Let $\{\lambda_n\}_{n = 1}^\infty\subset \mathbb R$ such that $B(\lambda _n, r_n) \subset \mathbb D$ and $\overline{B(\lambda _i, r_i)} \cap \overline{B(\lambda _j, r_j)} = \emptyset$ for $i \ne j$. Denote 
 \[
 \ K = \overline{\mathbb D} \setminus \left (\bigcup_{n = 1}^\infty B(\lambda _n, r_n) \right )
 \]
 and $E = K\cap \mathbb R$. Assume $E$ has no interior with positive linear measure. In this case, $\sum _{n = 1}^\infty r_n < 2$. $K$ is called a string of beads set. 

Set $G = int(K)$. 
Define
$G_U = \{z\in G:~Im(z) > 0\}, ~ G_L = G\setminus G_U$, $K_U = \overline{G_U}$, and $K_L = \overline{G_L}$.
The set $K$ does not satisfy the hypothesis of Theorem II.

The string of beads set $K$ has only one Gleason part (see Page 145 in \cite{gamelin}) even though $K$ has two disjoint interiors $G_U$ and $G_L$. Hence, the algebra $R(K)$, uniform closure of $Rat(K)$, contains no non-trivial characteristic functions. For $R^\infty(K, \mu)$, the weak$^*$ closure of $Rat(K)$ in $L^\infty(\mu)$, the envelope $E(K,\mu)$ is the set of complex numbers $a$ for which there exists a finite measure $\mu_a = h \mu$ with $\mu (\{a\}) = 0$ such that $f(a) = \int f d\mu_a$ for $f\in R(K)$. Chaumat's Theorem (see \cite{Ch74} or page 288 in \cite{C91}) states that if $R^\infty(K, \mu)$ has no $L^\infty$ summand and $G \subset E(K,\mu)$, then the identity map of $R(K)$ extends an isomorphism of $R^\infty(K, \mu)$ to $R^\infty(E(K,\mu), \mathcal L^2|_{E(K,\mu)})$. This implies that $R^\infty(K, \mu)$ contains no non-trivial characteristic functions. However, for $1\le t <\infty$, the situations for $R^t(K, \mu)$ are very different. The following example constructs a measure $\mu$ such that $S_\mu$ on $R^2(K, \mu)$ is pure, $\text{abpe}(R^2(K, \mu)) = G$, and $R^2(K, \mu)$ splits.
\smallskip

\begin{example}\label{SOBExample1}
Let $\mathcal L^2_W$ be the weighted planar Lebesgue measure $W(z)d\mathcal L^2(z)$ on $G$, where
 \[
 \ W(z) = \begin{cases}  \text{dist}(z, \mathbb R)^4, & z \in G_U\cup G_L;\\0. & z \in (G_U\cup G_L)^c.  \end{cases}
 \]
Then,
 \begin{eqnarray}\label{SOBEq1}
 \ R^2(K, \mathcal L^2_W) = R^2(K_U, \mathcal L^2_W |_{G_U}) \oplus R^2(K_L, \mathcal L^2_W |_{G_L}).
 \end{eqnarray}  
\end{example} 

\begin{proof}
It is easy to verify that $S_{\mathcal L^2_W}$ is pure and $\text{abpe}(R^2(K, \mathcal L^2_W)) = G$.
Let $\lambda\in E$. Choose a subsequence $\{\lambda_{n_k}\}$ such that $\lambda_{n_k}\rightarrow \lambda$. For $g\perp R^2(K, \mathcal L^2_W)$, we have
 \[
 \ \begin{aligned}
 \ & \int \left | \dfrac{1}{z-\lambda_{n_k}} - \dfrac{1}{z-\lambda} \right | |g|d\mathcal L^2_W \\
 \ \le &|\lambda_{n_k} - \lambda| \left(\int _G \dfrac{W(z)}{|z-\lambda_{n_k}|^2 |z - \lambda|^2}d\mathcal L^2(z) \right )^{\frac{1}{2}} \|g\|_{L^2(\mathcal L^2_W)} \\
 \ \le &\sqrt{\pi}\|g\|_{L^2(\mathcal L^2_W)}|\lambda_{n_k} - \lambda| \\
 \ & \rightarrow 0.
 \ \end{aligned} 
\]
Thus, $\mathcal C(g\mathcal L^2_W)(\lambda) = 0$ for $\lambda\in E$. Let $\nu$ be $\dfrac{1}{2\pi i}dz$ restricted to upper circle of $|z| = 2$ and the interval $[-2,2]$. By Fubini's theorem, we see that 
 \[
 \ \int \mathcal C(\nu)(z) g(z) d\mathcal L^2_W(z) = - \int \mathcal C(g\mathcal L^2_W)(\lambda) d\nu (\lambda) = 0.
 \]
Hence, $\chi_{G_U}(z) \in R^2(K, \mathcal L^2_W)$, where $\chi_{A}$ is the characteristic function for the set $A$. \eqref{SOBEq1} is proved.   
\end{proof}
\smallskip

Now let us look at a simple example of $\mu$ such that $S_\mu$ is irreducible.

\begin{example}\label{SOBExample2}
Let $a = \frac{1}{2}i$ and $b = - \frac{1}{2}i$, where $i = \sqrt{-1}$. Let $\omega_a$ and $\omega_b$ be the harmonic measures of $G_U$ and $G_L$ evaluated at $a$ and $b$ $b$, respectively. Let $\omega = \omega_a + \omega_b$.  Then, for every $f\in R^2(K, \omega)$, the nontangential limit from top on $E$, $f_+(z)$,  and the nontangential limit from bottom on $E$, $f_-(z)$, exist $m_E-a.a.$, and $f_+(z) = f_-(z) = f(z),~ m_E-a.a.$, where $m_E$ is the Lebesgue measure on $E$. Consequently, $S_\omega$ is irreducible.     
\end{example}

\begin{proof}
It is clear that $\int r(z)(z-a)(z-b) d \omega(z) = 0$ for $r\in Rat(K)$. So $S_\omega$ is pure. Let $\{r_n\}\subset Rat(K)$ such that $\|r_n - f\|_{L^2(\omega)} \rightarrow 0$. Hence, $f\in H^2(G_U)$ and $f\in H^2(G_L)$. Thus, $f_+(z)$ exists, $f_+(z) = f(z),~ \omega_a |_E-a.a.$ and $f_-(z)$ exists, $f_-(z) = f(z),~ \omega_b |_E-a.a.$. Since $\partial G_U$ and $\partial G_L$ are rectifiable Jordan curves, the measures $\omega_a |_E$, $\omega_b |_E$, and $m_E$ are mutually absolutely continuous. Therefore, 
 \[
 \ f_+(z) = f_-(z) = f(z),~ m_E-a.a.
 \]     
\end{proof}
\smallskip

Above examples suggest that, in order to ensure that $S_\mu$ on $R^t(K, \mu)$ is irreducible, values of every function $f\in R^t(K, \mu)$ on $G_U$ and $G_L$ must be associated in some ways. Our next theorem, which points out the relationship, is a special case of Main Theorem II. We explain the proof to provide ideas of proving our main theorems. First we need to introduce some notations.         

For $\lambda\in \mathbb R$ and $0 < \alpha \le 1$, define the upper cone 
\[
\ U(\lambda, \alpha) = \{z:~ |Re(z - \lambda)| < \alpha Im(z)\}
\]
and the lower cone 
\[
\ L(\lambda, \alpha) = \{z:~ |Re(z - \lambda)| < - \alpha Im(z)\}.
\]
Set $U_n(\lambda,\alpha) = U(\lambda,\alpha)\cap B(\lambda,\frac{1}{n})$ and $L_n(\lambda,\alpha) = L(\lambda,\alpha)\cap B(\lambda,\frac{1}{n})$. 

It is easy to verify that for almost all $\lambda \in E$ with respect to $m_E$, there exists $n$ such that $U_n(\lambda,\alpha) \cup L_n(\lambda,\alpha) \subset G$. 

For a function $f$, we define nontangential limits of $f$ from top and bottom by
 \[
 \ f_+(\lambda ) = \underset{z\in U(\lambda, \alpha)\rightarrow \lambda}{\lim} f(z), ~ f_-(\lambda ) = \underset{z\in L(\lambda, \alpha)\rightarrow \lambda}{\lim} f(z).  
 \]
Let $F\subset E$ be a Borel subset. Define
 \[
 \ H^\infty_F(G) = \{f\in H^\infty(G): ~ f_+(z) = f_-(z),\text{ for } z\in F, ~m_E-a.a.\}.
 \]
\smallskip

\begin{theorem}\label{SOBTheorem}
Let $1 \le t < \infty$. If $S_\mu$ on $R^t(K,\mu)$ is pure and $\text{abpe}(R^t(K,\mu)) = G$, then there exists a Borel subset $F\subset E$ such that 
 \[
 \ \rho: f\in R^t(K, \mu) \cap L^\infty(\mu) \rightarrow f|_G \in H^\infty(G) 
 \]
is an isometric isomorphism and a weak$^*$ homeomorphism from $R^t(K, \mu) \cap L^\infty(\mu)$ onto $H^\infty_F(G)$. 
\end{theorem}
\smallskip

Certainly, the set $F$ is unique up to a zero $m_E$ measure set. 
The set $F$ is the removable boundary $\mathcal R_B$ for $R^t(K,\mu)$ defined as in Definition \ref{NRBDef}.

Our plan to describe the ideas is the following: 

Step I: Determine the points $\lambda\in E$ for which $f\in R^t(K,\mu)$ has non-tangential limits from top (or bottom). Construct the set $F$ and prove $\rho (f) \in H^\infty_F(G)$ for $f \in R^t(K, \mu) \cap L^\infty(\mu)$.  

Step II: Describe the key analytic capacity estimation that allows us to apply Vitushkin's scheme for $f\in H^\infty_F(G)$ and prove the map $\rho$ is onto.
 
Now let us start with Step I. Our approach relies on a generalized Plemelj's formula that is proved by Theorem 3.6 in \cite{acy18} as the following.
\smallskip

\begin{lemma}\label{GPTheorem}
Let $\nu$ be a finite, complex-valued Borel measure with 
compact support in $\mathbb{C}$. Suppose that $\nu = hm_E + \sigma$ is the Radon-Nikodym 
decomposition with respect to $m_E$, where $h\in L^1(m_E)$ and $\sigma\perp m_E$. Then there exists a subset $\mathbb Q\subset \mathbb C$ with $\gamma(\mathbb Q) = 0$, such that the following hold:

(a) $\mathcal C(\nu ) (\lambda) = \lim_{\epsilon\rightarrow 0} \mathcal C_{\epsilon}(\nu)(\lambda)$ exists for $\lambda\in \mathbb C\setminus \mathbb Q$,

(b) for $\lambda \in E \setminus \mathbb Q$ and $\epsilon > 0$, $v^+(\nu, \lambda):= \mathcal{C}(\nu )(\lambda ) +\pi i h(\lambda)$,
\begin{eqnarray}\label{GPTheoremEq1}
 \ \lim_{\delta \rightarrow 0} \dfrac{\gamma(U(\lambda, \alpha) \cap B(\lambda, \delta)\cap \{ |\mathcal{C}(\nu )(z ) - v^+(\nu, \lambda)| > \epsilon \})}{\delta } = 0, 
 \end{eqnarray}
and

(c) for $\lambda \in E \setminus \mathbb Q$ and $\epsilon > 0$, $v^-(\nu, \lambda):= \mathcal{C}(\nu )(\lambda ) -\pi i h(\lambda)$,
\begin{eqnarray}\label{GPTheoremEq2}
 \ \lim_{\delta \rightarrow 0} \dfrac{\gamma(L(\lambda, \alpha) \cap B(\lambda, \delta)\cap \{ |\mathcal{C}(\nu )(z ) - v^-(\nu, \lambda)| > \epsilon \})}{\delta } = 0. 
 \end{eqnarray}
\end{lemma}
\smallskip

For our purposes with $R^t(K, \mu)$ in general, we will develop a generalized version of above Plemelj's formula for an arbitrary measure (see Theorem \ref{GPTheorem1}).

Let $\{g_j\} \subset R^t(K,\mu)^\perp$ be a dense set.
Let $\mu  = hm_E  + \mu_s$ be the Radon-Nikodym decomposition with respect to $m_E$ on $\mathbb R$, where $h \in L^1(m_E)$ and $\mu_s \perp  m_E$. For $f\in R^t(K, \mu)$, we see that $f(z)$ is analytic on $\text{abpe}(R^t(K,\mu)) = G$ and
 \[
 \ f(z) \mathcal C(g_j\mu) (z) = \mathcal C(fg_j\mu) (z), ~ z\in G_U, ~ \gamma-a.a.
 \]  
Using Lemma \ref{GPTheorem}, we may assume, for $\lambda\in E$, \eqref{GPTheoremEq1} holds.
If there exists $j_1$ such that $v^+ (g_{j_1}\mu, \lambda) \ne 0$, then one can prove
 \[
 \ \lim_{\delta\rightarrow 0} \dfrac{\gamma(U(\lambda, \alpha) \cap B(\lambda, \delta)\cap \{|f (z) - f(\lambda)) | > \epsilon\})}{\delta} = 0
 \]
which implies $f_+(\lambda) = f(\lambda)$ by Lemma \ref{lemmaARS}, where
\[
 \ f(\lambda) = \dfrac{v^+ (fg_{j_1}\mu, \lambda)}{v^+ (g_{j_1}\mu, \lambda)}.
 \]
The details are contained in the proof of Theorem \ref{MTheorem2}. 

Let $F_U$ consist of all $\lambda\in E$ for which there exists $j_1$ such that the principal value $\mathcal C(g_{j_1}\mu) (\lambda)$ exists and $v^+ (g_{j_1}\mu, \lambda) \ne 0$. Let $F_L$ consist of all $\lambda\in E$ for which there exists $j_2$ such that the principal value $\mathcal C(g_{j_2}\mu) (\lambda)$ exists and $v^- (g_{j_2}\mu, \lambda) \ne 0$. Set $F = F_U\cap F_L$. Then 
 \[
 \ f_+(\lambda) = f_-(\lambda)  = f(\lambda),~ \lambda\in F,~m_E-a.a.
 \]

Clearly, for $\lambda \in E \setminus F$, either $v^+ (g_{j}\mu, \lambda) = 0$ for all $j \ge 1$ or $v^- (g_{j}\mu, \lambda) = 0$ for all $j \ge 1$. The ideas here help us to achieve the definition of $\mathcal F$ and $\mathcal R$ as in Definition \ref{FRDefinition1}. For a string of beads set $K$ above, 
 \[
 \ \mathcal F = (\partial G \setminus F) \cup K^c,~ \mathcal R = G\cup F,\text{ and }\mathcal R_B = F.
\]

Now let us discuss the key Step II. That is, for $f\in H^\infty_F(G)$, we want to find $\hat f \in R^t(K, \mu) \cap L^\infty(\mu)$ such that $f = \hat f|_G$. It suffices to assume that $f$ is analytic off $E$. From Lemma \ref{lemmaBasic0} (6), one can find a bounded complex-valued function $w(z)$ on $E$ such that $f(z) = \mathcal C(wm_E)(z)$ for $z\in G$. Since, by Lemma \ref{GPTheorem},
$f_+(z) = \mathcal C(wm_E)(z) + \pi i w(z)$
and
$f_-(z) = \mathcal C(wm_E)(z) - \pi i w(z)$
$m_E-a.a.$ which implies that $w(z) = 0$ on $z\in F~ m_E-a.a.$. Let $S$ be a square with length $l$ and $S\cap E \ne \emptyset$. Let $\varphi$ be a smooth function supported in $S$ satisfying $0 \le \varphi \le 1$ and $\|\frac{\partial \varphi}{\partial \bar z} \|_\infty \le \dfrac{100}{l}$. $\mathcal C(\varphi w m_E)(\infty) = 0$ and
 \begin{eqnarray}\label{SBEst}
 \ \begin{aligned}
 \ |\mathcal C(\varphi wm_E) '(\infty)| = & \left |\int \varphi w(z) d m_E \right | \\
 \ \le &m(S\cap (E\setminus F)| \\
 \ \le &4 \gamma(S\cap (E\setminus F)).
 \ \end{aligned}
 \end{eqnarray}
Theorem \ref{MLemma1} indeed shows the above inequality holds in general.
In order to apply Vitushkin's scheme, we need to construct a function $f_0$ that is analytic off $S$ satisfying:
 
(1) $\|f_0\|_{L^\infty(\mathbb C)} \le C_3$;

(2) $f_0(\infty ) = 0$;

(3) $f_0'(\infty) = \gamma(S\cap (E\setminus F))$; and 

(4)  $f_0\in R^t(K, \mu) \cap L^\infty(\mu)$.

If there exists $f_0$ satisfying (1)-(4), then, by \eqref{SBEst},  $\mathcal C(\varphi w m_E) - \frac{\mathcal C(\varphi w m_E)'(\infty)}{f_0'(\infty)}f_0$ is bounded  and has double zero at $\infty$. Then the standard Vitushkin scheme can be used here to construct $F\in R^t(K, \mu) \cap L^\infty(\mu)$ with $F|_G = f$. Therefore, $\rho$ is onto. Applying Theorem \ref{MLemma1} and a modified Vitushkin scheme of P. Paramonov, we prove in Theorem \ref{MLemma2} that the result actually holds in general.  

The existence of such a function $f_0$ is a simple outcome of Theorem \ref{acTheorem} and Lemma \ref{hFunction}. The proof depends on the property that $\mathcal C(g_j\mu)$ is close to zero near the set $F$ (though $\mathcal C(g_j\mu)$ may not be zero on $F$ for some $g_j$) as in Lemma \ref{lemmaBasic4}.    

Before closing this section, we provide the following proposition that combines Example \ref{SOBExample1} and \ref{SOBExample2}.

\begin{proposition}\label{SOBRemovable}
If $F\subset E$ is a compact subset with $m_E(F) > 0$, then there exists a finite positive measure $\mu$ supported on a string of beads set $K$ such that $S_\mu$ on $R^2(K, \mu)$ is irreducible, $\text{abpe}(R^2(K, \mu)) = G$, and $F$ is its removable boundary. 
\end{proposition}

\begin{proof}
Let $\Gamma_U$ be a closed rectifiable Jordan curve in $K_U$ such that $\Gamma_U\cap \partial G_U = F$. Let $\Gamma_L$ be a closed  rectifiable Jordan curve in $K_L$ such that $\Gamma_L\cap \partial G_L = F$. Let $\Omega_U$ and $\Omega_L$ be the regions surrounded by $\Gamma_U$ and $\Gamma_L$, respectively. Let $\omega_a$ and $\omega_b$ be the harmonic measures of $\Omega_U$ and $\Omega_L$, respectively, where $a\in \Omega_U$ and $b\in \Omega_L$. Let $\mathcal L^2_W$ be the weighted planar Lebesgue measure as in Example \ref{SOBExample1}. Let $\mu = \omega_a + \omega_b + \mathcal L^2_W$. 

Similar to the proof of Example \ref{SOBExample2}, we see that $f_+(z) = f_-(z) = f(z)$ on $F$ $m_E-a.a.$. Hence, $S_\mu$ is irreducible. Obviously, $\text{abpe}(R^2(K, \mu)) = G$.

Suppose that $F$ is not a maximal removable boundary. Then there exists a compact subset $F_0\subset E\setminus F$ that is a part of removable boundary and $m_E(F_0) > 0$.  Let $f$ be the admissible function of $F_0$ for analytic capacity.  Using Lemma \ref{lemmaBasic0} (6), we have $f = \mathcal C(wm_E)$, where $w$ is bounded and supported in $F_0$. As the same proof of Example \ref{SOBExample1}, we see that $\dfrac{1}{z-\lambda} \in R^2(K, \mu)$ for $\lambda \in F_0$. Moreover,
 \[
 \ \begin{aligned}
 \ & \int \dfrac{1}{|z-\lambda|} |g_j|d\mu \\
 \ \le & \int _{\Gamma_U}\dfrac{1}{|z-\lambda|} |g_j|d\omega_a + \int _{\Gamma_L}\dfrac{1}{|z-\lambda|} |g_j|d\omega_b + \int _G\dfrac{W(z)}{|z-\lambda|} |g_j|d\mathcal L^2 \\
 \ \le & \dfrac{\|g_j\|}{\text{dist}(\lambda, \Gamma_U)} + \dfrac{\|g_j\|}{\text{dist}(\lambda, \Gamma_L)} + \|g_j\| \\
 \ \end{aligned}
 \]
for $\lambda\in F_0$. Using Fubini's theorem, we have
 \[
 \ \int f(z) g_j(z) d\mu(z) = - \int \mathcal C( g_j\mu)(\lambda) w(\lambda) d m_E(\lambda) = 0. 
 \]
Hence, $f\in R^2(K,\mu)$. This contradicts 
 to Theorem \ref{SOBTheorem} 
because, by Lemma \ref{GPTheorem}, $f_+$ and $f_-$ are not identically equal with respect to $m_E$ on $F_0$.
\end{proof}

\bigskip

\section{Generalized Plemelj's formula}
\bigskip

\begin{definition}
Let $f(\lambda)$ be a function on $B(\lambda _0, \delta)\setminus \mathbb Q$ with $\gamma(\mathbb Q) = 0$. The function $f$ has a $\gamma$-limit $a$ at $\lambda _0$ if
\[  
 \  \lim_{\delta \rightarrow 0} \dfrac{\gamma(B(\lambda _0, \delta) \cap \{|f(\lambda) - a| > \epsilon\})} {\delta}= 0
\]
for all $\epsilon > 0$. If in addition, $f(\lambda_0)$ is well defined and $a = f(\lambda_0)$, then $f$ is $\gamma$-continuous at $\lambda_0$. 
\end{definition}
\smallskip

The following lemma is from Lemma 3.2 in \cite{acy18}.

\begin{lemma}\label{CauchyTLemma} 
Let $\nu$ be a finite, complex-valued Borel measure that is compactly
supported in $\mathbb{C}$ and assume that for some $\lambda _0$ in $\mathbb C$ we have:
\begin{itemize}
\item[(a)] $\Theta_{\nu} (\lambda _0 ) = 0$ and 
\item[(b)] $\mathcal{C} (\nu)(\lambda _0) = \lim_{\epsilon \rightarrow 0}\mathcal{C} _{\epsilon}(\nu)(\lambda _0)$ exists.
\end{itemize}

Then:

(1) $\mathcal{C}(\nu)(\lambda) = \lim_{\epsilon \rightarrow 0}\mathcal{C} _{\epsilon}(\nu)(\lambda )$ 
exists for $\lambda\in \mathbb Q^c$ with $\gamma(\mathbb Q) = 0$ and

(2) Cauchy transform $\mathcal{C}(\nu)(\lambda)$ is $\gamma$-continuous at $\lambda_0$.
\end{lemma}
\smallskip

In order to state our generalized Plemelj's formula, we need to introduce some further notations.
 
Let $A:\mathbb R\rightarrow \mathbb R$ be a Lipschitz function with graph $\Gamma$.
We define the open upper cone (with vertical axis)  
	 \[
	 \ UC (\lambda, \alpha ) = \left \{z \in \mathbb C :~ |Re(z) -  Re(\lambda )| < \alpha (Im(z) -  Im(\lambda ))\right \}
 \]
	and the open lower cone (with vertical axis)
	 \[
	 \ LC (\lambda, \alpha ) = \left \{z \in \mathbb C :~ |Re(z) -  Re(\lambda )| < - \alpha (Im(z) -  Im(\lambda ))\right \}.
  \]
	Set $UC (\lambda, \alpha, \delta ) = UC (\lambda, \alpha)\cap B(\lambda ,  \delta)$	and
	$LC (\lambda, \alpha, \delta ) = LC (\lambda, \alpha)\cap B(\lambda ,  \delta)$.
		
Observe that if $\alpha < \frac{1}{\|A'\|_\infty},$ then, for almost all $\lambda\in \Gamma,$ there exists $\delta > 0$ such that 
	\[
	 \ UC (\lambda, \alpha, \delta) \subset U_\Gamma:=\left\{z\in \mathbb C:~ Im(z) > A(Re(z))\right\}
	 \]
	 and 
	 \[
	 \ LC (\lambda, \alpha, \delta)\subset L_\Gamma:=\left\{z\in \mathbb C:~ Im(z) < A(Re(z))\right\}.
	 \]
	We consider the usual complex-valued measure
	 \[
	 \ \dfrac{1}{2\pi i} dz_{\Gamma} = \dfrac{1 + iA'(Re(z))}{2\pi i (1 + A'(Re(z))^2)^{\frac{1}{2}}} d\mathcal H^1 |_{\Gamma} = L(z)d\mathcal H^1 |_{\Gamma}.
	 \]
Notice that
$|L(z)| = \frac{1}{2\pi}.$ 
Plemelj's formula for Lipschitz graph as in Theorem 8.8 of \cite{Tol14} is the following.
\smallskip
	
	\begin{lemma}\label{PFLipschitz}
	Let $A : \mathbb R \rightarrow \mathbb R$ be a Lipschitz function with its graph $\Gamma$. Let $0 < \alpha < \frac{1}{\|A'\|_\infty}$. Then, for all $f\in L^1(\mathcal H^1 |_{\Gamma})$, the non-tangential limits of $\mathcal C(f dz |_{\Gamma})(\lambda )$ exists for $\mathcal H^1 |_{\Gamma}-a.a.$ and the following identities hold $\mathcal H^1 |_{\Gamma}-a.a.$ for $\lambda\in \Gamma$:
	\[
	 \ \dfrac{1}{2\pi i} \lim_{z\in UC (\lambda, \alpha ) \rightarrow \lambda}\mathcal C(f dz|_{\Gamma})(z) = \dfrac{1}{2\pi i}\mathcal C(f dz|_{\Gamma})(\lambda) + \dfrac{1}{2} f(\lambda),
	\]
	and
	\[
	 \ \dfrac{1}{2\pi i} \lim_{z\in LC (\lambda, \alpha ) \rightarrow \lambda}\mathcal C(f dz|_{\Gamma})(z) = \dfrac{1}{2\pi i}\mathcal C(f dz|_{\Gamma})(\lambda) - \dfrac{1}{2} f(\lambda).
	\]
\end{lemma}
\smallskip

We are now ready to improve Lemma \ref{GPTheorem} as the following (replacing $UC (\lambda, \alpha,\delta )$ by $U_\Gamma \cap B(\lambda, \delta )$ and $LC (\lambda, \alpha,\delta )$ by $L_\Gamma \cap B(\lambda, \delta )$).
\smallskip

\begin{theorem}\label{GPTheorem1}
(Plemelj's Formula for an arbitrary measure) Let $\alpha < \frac{1}{\|A'\|_\infty}$. Let $\nu$ be a finite, complex-valued Borel measure with 
compact support in $\mathbb{C}$. Suppose that $\nu = h\mathcal H^1 |_{\Gamma} + \sigma$ is the Radon-Nikodym 
decomposition with respect to $\mathcal H^1 |_{\Gamma}$, where $h\in L^1(\mathcal H^1 |_{\Gamma})$ and $\sigma\perp \mathcal H^1 |_{\Gamma}$. Then there exists a subset $\mathbb Q\subset \mathbb C$ with $\gamma(\mathbb Q) = 0$, such that the following hold:

(a) $\mathcal C(\nu ) (\lambda) = \lim_{\epsilon\rightarrow 0} \mathcal C_{\epsilon}(\nu)(\lambda)$ exists for $\lambda\in \mathbb C\setminus \mathbb Q$,

(b) for $\lambda_0 \in \Gamma \setminus \mathbb Q$, $v^+(\nu, \Gamma, \lambda_0) := \mathcal C(\nu ) (\lambda_0) + \frac{1}{2}h(\lambda_0) L(\lambda_0)^{-1},$
\[
 \ \lim_{\delta \rightarrow 0} \dfrac{\gamma(U_\Gamma \cap B(\lambda, \delta ) \cap \{|\mathcal{C}(\nu )(\lambda ) - v^+(\nu, \Gamma, \lambda_0)| > \epsilon\})} {\delta}= 0  
 \]
for all $\epsilon > 0$;

(c) for $\lambda_0 \in \Gamma \setminus \mathbb Q$, $v^-(\nu, \Gamma, \lambda_0) := \mathcal C(\nu ) (\lambda_0) - \frac{1}{2}h(\lambda_0) L(\lambda_0)^{-1},$
\[
 \ \lim_{\delta \rightarrow 0} \dfrac{\gamma(L_\Gamma \cap B(\lambda, \delta ) \cap \{|\mathcal{C}(\nu )(\lambda ) - v^-(\nu, \Gamma, \lambda_0)| > \epsilon\})} {\delta}= 0  
 \]
for all $\epsilon > 0$;
and

(d) for $\lambda_0 \in \Gamma \setminus \mathbb Q$, $v^0(\nu, \Gamma, \lambda_0) := \mathcal C(\nu ) (\lambda_0),$
\[
 \ \lim_{\delta \rightarrow 0} \dfrac{\gamma(\Gamma \cap B(\lambda, \delta ) \cap \{|\mathcal{C}(\nu )(\lambda ) - v^0(\nu, \Gamma, \lambda_0)| > \epsilon\})} {\delta}= 0  
 \]
for all $\epsilon > 0$.
\end{theorem}

\begin{proof} We restrict $\Gamma$ to the portion of the graph inside $B(0,r)$ (write $\Gamma \subset B(0,r)$), where $\text{spt}(\nu) \subset B(0,r)$. 
Since $\sigma\perp \mathcal H^1|_{\Gamma}$, using Lemma \ref{lemmaBasic0} (5) (b) and Lemma \ref{CauchyTLemma}, we conclude that (b), (c), and (d) holds for the measure $\sigma$. So we may assume that $d\nu = hd\mathcal H^1 |_{\Gamma}$.

First we prove (d): From Lemma \ref{lemmaBasic0} (7), we see  that Cauchy transform of $\mathcal H^1 |_{\Gamma}$ is bounded on $L^2(\mathcal H^1 |_{\Gamma})$. So we get $\mathcal C\mathcal H^1 |_{\Gamma}\in L^2(\mathcal H^1 |_{\Gamma})$ as $\Gamma \subset B(0,r)$. Also $h\in L^1(\mathcal H^1 |_{\Gamma})$. Let $\lambda _0$ be a Lebesgue point of $h$ and $\mathcal C\mathcal H^1 |_{\Gamma}$. That is,
\begin{eqnarray}\label{LPoint1}
 \ \lim_{\delta\rightarrow 0} \dfrac{1}{\delta}\int_{B(\lambda_0,\delta)} |h(z) - h(\lambda_0)|d\mathcal H^1 |_{\Gamma}(z) = 0 
\end{eqnarray}
and 
\begin{eqnarray}\label{LPoint2}
 \ \lim_{\delta\rightarrow 0} \dfrac{1}{\delta}\int_{B(\lambda_0,\delta)} |\mathcal C\mathcal H^1 |_{\Gamma} (z) - \mathcal C\mathcal H^1 |_{\Gamma} (\lambda_0)|d\mathcal H^1 |_{\Gamma}(z) = 0 
\end{eqnarray}
as $\mathcal H^1 |_{\Gamma}(B(\lambda_0,\delta)) \approx \delta$ by Lemma \ref{lemmaBasic0} (7). 
 
For $\epsilon > 0$, from Theorem \ref{TTolsa} (2), we get (assuming $h(\lambda_0)\ne 0$):
 \[
 \ \begin{aligned}
 \ & \gamma\left( \{|\mathcal C\nu (\lambda) - \mathcal C\nu (\lambda_0) | > \epsilon \} \cap B(\lambda_0, \delta) \cap \Gamma\right )\\
 \ \le & A_T  \gamma\left (\{|\mathcal C((h-h(\lambda_0))\mathcal H^1 |_\Gamma )(\lambda) - \mathcal C((h-h(\lambda_0))\mathcal H^1 |_\Gamma )(\lambda_0) |  > \dfrac{\epsilon}{2} \} \cap B(\lambda_0, \delta) \cap \Gamma \right )\\  
 \ &+ A_T\gamma \left ( \{|\mathcal C\mathcal H^1 |_{\Gamma} (\lambda) - \mathcal C\mathcal H^1 |_{\Gamma} (\lambda_0) | > \dfrac{\epsilon}{2|h(\lambda_0)|} \}\cap B(\lambda_0, \delta) \cap \Gamma \right).
 \ \end{aligned}
 \]
From \eqref{HACEq}, we see that
 \[
 \ \begin{aligned}
 \ & \gamma \left ( \{|\mathcal C\mathcal H^1 |_{\Gamma} (\lambda) - \mathcal C\mathcal H^1 |_{\Gamma} (\lambda_0) | > \dfrac{\epsilon}{2|h(\lambda_0|} \}\cap B(\lambda_0, \delta) \cap \Gamma \right)\\
 \ \le & C_\Gamma \mathcal H^1 |_{\Gamma} \left (\{|\mathcal C\mathcal H^1 |_{\Gamma} (\lambda) - \mathcal C\mathcal H^1 |_{\Gamma} (\lambda_0) | > \dfrac{\epsilon}{2|h(\lambda_0)|} \} \cap B(\lambda_0, \delta) \cap \Gamma \right ) \\
\ \le & \dfrac{2 C_\Gamma |h(\lambda_0)|}{\epsilon} \int _{B(\lambda_0, \delta)}|\mathcal C\mathcal H^1 |_{\Gamma} (z) - \mathcal C\mathcal H^1 |_{\Gamma} (\lambda_0) | d\mathcal H^1 |_{\Gamma}(z).
 \ \end{aligned}
 \]
It follows from  \eqref{LPoint2} that 
 \[
 \ \lim_{\delta\rightarrow 0} \dfrac{\gamma \left ( \{|\mathcal C\mathcal H^1 |_{\Gamma} (\lambda) - \mathcal C\mathcal H^1 |_{\Gamma} (\lambda_0) | > \dfrac{\epsilon}{2|h(\lambda_0)|} \}\cap B(\lambda_0, \delta) \cap \Gamma \right)}{\delta} = 0. 
 \]
Therefore, we may assume that $h(\lambda_0) = 0$. In this case, $\nu$ satisfies the assumptions of Lemma \ref{CauchyTLemma} as \eqref{LPoint1} holds and we prove (d) because $\mathcal C(\nu)(z)$ is $\gamma$-continuous at $\lambda_0$ by Lemma \ref{CauchyTLemma}.

Now let us prove (b). For $a>0$, let $A_n$ be the set of $\lambda\in \mathcal N(h) (=\{h(z) \ne 0\})$ such that 
 \[
 \ |\mathcal C (\nu)(z) - \mathcal C (\nu)(\lambda) - \dfrac{1}{2}h(\lambda)L(\lambda)^{-1}| < \dfrac{a}{2}
 \]
for $z\in UC(\lambda, \alpha, \frac{1}{n})$. Then from Lemma \ref{PFLipschitz}, we obtain
 \begin{eqnarray}\label{ANEq}
 \ \mathcal H^1 |_{\Gamma} \left (\mathcal N(h) \setminus \bigcup_{n=1}^\infty A_n \right ) = 0.
 \end{eqnarray}
Let $\lambda_0$ be a Lebesgue point of $A_n$ and $h(\lambda)L(\lambda)^{-1}$, that is,
 \[
 \ \lim_{\delta\rightarrow 0} \dfrac{\mathcal H^1 |_{\Gamma}(A_n^c \cap B(\lambda_0, \delta))}{\delta} = 0
 \]
and
 \[
 \ \lim_{\delta\rightarrow 0 } \dfrac{\mathcal H^1 |_{\Gamma} \left ( \{\lambda : |h(\lambda)L(\lambda)^{-1} - h(\lambda_0)L(\lambda_0)^{-1}| \ge \frac{a}{2} \} \cap B (\lambda _0, \delta) \right)}{\delta} = 0.
 \]
Assume that (d) holds for $\lambda_0$, that is,
 \[
 \ \lim_{\delta\rightarrow 0 } \dfrac{\mathcal H^1 |_{\Gamma} \left ( \{\lambda : |\mathcal C(\nu)(\lambda) - \mathcal{C}(\nu)(\lambda _0)| \ge \frac{a}{4} \} \cap B (\lambda _0, \delta) \right)}{\delta} = 0.
 \]
Therefore,
 \[
 \ \lim_{\delta\rightarrow 0 } \dfrac{\mathcal H^1 |_{\Gamma} \left ( (A_n^c \cup \{\lambda : |v^+(\nu, \Gamma, \lambda) - v^+(\nu, \Gamma, \lambda_0)| \ge \frac{a}{2} \}) \cap B (\lambda _0, \delta) \right)}{\delta} = 0.
 \]
For $\epsilon > 0$, there exists $\delta_0 > 0$ ($\delta_0 < \frac{1}{2n}$) and an open set $O\subset B (\lambda _0, \delta)$ for $\delta < \delta_0$ such that 
 \[
 \ (A_n^c \cup \{\lambda : |v^+(\nu, \Gamma, \lambda) - v^+(\nu, \Gamma, \lambda_0)| \ge \frac{a}{2} \}) \cap B (\lambda _0, \delta) \subset O
 \]
and $\mathcal H^1 |_{\Gamma} (O) < \epsilon\delta$. Let $E = \Gamma \cap (\overline{B(\lambda_0,\delta)}\setminus O)$. Denote
 \[
 \ B = \bigcup_{\lambda\in E} \overline{UC(\lambda, \alpha,\frac{1}{n})}.
 \]
It is easy to check that $B$ is closed. By construction, we see that
 \[
 \ |\mathcal C (\nu)(z) - v^+(\nu, \Gamma, \lambda_0)| < a
 \]
for $z\in B\setminus \Gamma$. Write 
 \[
 \ U_\Gamma \cap B(\lambda_0, \delta) \setminus B = \cup _{k=1}^\infty G_k
 \]
where $G_k$ is a connected component. It is clear that $diam(G_k) \approx \mathcal H^1 |_{\Gamma} (\overline{G_k}\cap\Gamma)$. Let $I_k$ be the interior of $\overline{G_k}\cap\Gamma$ on $\Gamma$. Then $\mathcal H^1 |_{\Gamma} (\overline{G_k}\cap\Gamma) = \mathcal H^1 |_{\Gamma} (I_k)$ and $\cup_k I_k \subset O\cap \Gamma$. Hence, from Theorem \ref{TTolsa} (2), we get 
 \begin{eqnarray}\label{GEDensity}
 \ \begin{aligned}
 \ \gamma ( U_\Gamma \cap B(\lambda_0, \delta) \setminus B )\le &A_T\sum_{k=1}^\infty \gamma (G_k)  \\ 
\ \le & A_T\sum_{k=1}^\infty diam (G_k)  \\
 \ \le &A_TC_6\epsilon\delta.
 \ \end{aligned}
 \end{eqnarray}
This implies
 \begin{eqnarray}\label{ANEq1}
 \ \lim_{\delta\rightarrow 0 } \dfrac{\gamma \left ( \{\lambda : |\mathcal C (\nu)(\lambda) - v^+(\nu, \Gamma, \lambda_0)| > a \} \cap B (\lambda _0, \delta)\cap U_\Gamma \right)}{\delta} = 0.
 \end{eqnarray}
So from \eqref{ANEq}, we have proved that there exists $\mathbb Q_a$ with $\mathcal H^1 (\mathbb Q_a) = 0$ such that for each $\lambda_0\in \mathcal N(h) \setminus \mathbb Q_a$, \eqref{ANEq1} holds. Set $\mathbb Q_0 = \cup_{m=1}^\infty \mathbb Q_{\frac{1}{m}}$. Then \eqref{ANEq1} holds for $\lambda_0\in \mathcal N(h) \setminus \mathbb Q_0$ and all $a > 0$. This proves (b).

The proof (c) is the same as (b).  
\end{proof}
\smallskip

Before closing this section, we point out if $\Gamma$ is a rotation of a Lipschitz graph (with rotation angle $\beta$ and Lipschitz function $A$) at $\lambda_0$, then
 \begin{eqnarray}\label{VPlusBeta}
 \ v^+(\nu, \Gamma, \lambda) := \mathcal C(\nu ) (\lambda) + \frac{1}{2} e^{-i\beta}h(\lambda) L((\lambda - \lambda_0)e^{-i\beta} + \lambda_0)^{-1}.
 \end{eqnarray}
Similarly,
 \begin{eqnarray}\label{VMinusBeta}
 \ v^-(\nu, \Gamma, \lambda) := \mathcal C(\nu ) (\lambda) - \frac{1}{2} e^{-i\beta}h(\lambda) L((\lambda - \lambda_0)e^{-i\beta} + \lambda_0)^{-1}.
 \end{eqnarray}

\bigskip

\chapter{Removable and non-removable boundaries}
\bigskip

In this chapter, we assume that $1\le t <\infty$, $\mu$ is a finite positive Borel measure supported on a compact subset $K\subset \mathbb C$, $S_\mu$ on $R^t(K,\mu)$ is pure, and $K=\sigma(S_\mu)$. Let $\{g_n\}_{n=1}^\infty \subset R^t(K,\mu) ^\perp$ be a dense subset. Notice that for $t=1$, $\{g_n\}_{n=1}^\infty$ is a weak$^*$ dense set in $R^1(K,\mu)^\perp\subset L^\infty(\mu)$. Set $\nu_j = g_j\mu$.

We introduce the concept of non-removable boundary $\mathcal F\subset \mathbb C$, a Borel set, and the removable set $\mathcal R = K\setminus \mathcal F$ for $R^t(K,\mu)$. The set $\mathcal F$ consists of three sets, $\mathcal F_0$ and $\mathcal F_+ \cup \mathcal F_-$ such that $\mathcal C(g_n\mu)$ are zero on $\mathcal F_0$ and 
$\mathcal F_+ \cup \mathcal F_-$ is contained in a countable union of rotated Lipschitz graphs $\Gamma_n$, $\mu |_{\mathcal F_+ \cup \mathcal F_-}$ is
absolutely continuous with respect to $\mathcal H^1|_{\cup_n \Gamma_n}$, and $\mathcal C(g_n\mu)$ has zero one side nontangential limits $\mu |_{\mathcal F_+ \cup \mathcal F_-}-a.a.$ in full analytic capacitary density. We prove that $\mathcal F$, $\mathcal R$, $\mathcal F_0$, $\mathcal F_+$, and $\mathcal F_-$ are  independent of choices of $\{\Gamma_n\}$ and $\{g_n\}$ up to a set of zero analytic capacity. We discuss important properties related to the sets $\mathcal F$, $\mathcal R$, and $\text{abpe}(R^t(K,\mu))$ in sections 3,4,\&5 (required for proving our main theorem, Theorem \ref{DecompTheorem}). 
\bigskip

\section{Definition of non-removable boundary}
\bigskip   

Let 
 \[ 
 \ \mathcal N(f) = \{\lambda: ~ f(\lambda)\text{ is well defined,} ~ f(\lambda) \ne 0 \}
 \]
 and 
 \[
 \ \mathcal Z(f)  = \{\lambda: ~ f(\lambda)\text{ is well defined,}  ~ f(\lambda) = 0 \}
 \]
 be the non-zero set and zero set of a function $f$, respectively. Set
 \[
 \ \mathcal N(\nu_j, 1 \le j < \infty) = \bigcup_{j = 1}^\infty \mathcal N(\mathcal C(\nu_j)),
 \]
\[
 \ \mathcal Z(\nu_j,1 \le j < \infty) = \bigcap_{j = 1}^\infty \mathcal Z(\mathcal C(\nu_j)),
 \]
 \[
 \ \mathcal Z(\nu_j,1 \le j  < \infty, \Gamma, +) = \bigcap_{j = 1}^\infty \mathcal Z(v^+(\nu_j, \Gamma,.)), 
 \]
 \[
 \ \mathcal Z(\nu_j,1 \le j < \infty, \Gamma, -) = \bigcap_{j = 1}^\infty \mathcal Z(v^-(\nu_j, \Gamma,.)),
 \]
 \[
 \ \mathcal N(\nu_j,1 \le j < \infty, \Gamma, +) = \bigcup_{j = 1}^\infty \mathcal N(v^+(\nu_j, \Gamma,.)),
 \]
 and 
 \[
 \ \mathcal N(\nu_j,1 \le j < \infty, \Gamma, -) = \bigcup_{j = 1}^\infty \mathcal N(v^-(\nu_j, \Gamma,.)),
 \]
where $\Gamma$ is a (rotated) Lipschitz graph.
Denote ($\mathcal{ZD}(\nu)$ and $\mathcal{ND}(\nu)$ as in \eqref{ZNDensity})
 \[
 \ \mathcal{ZD}(\nu_j,1\le j < \infty) = \bigcap_{j=1}^\infty\mathcal{ZD}(\nu_j) 
 \]
and
 \[
 \ \mathcal{ND}(\nu_j,1\le j < \infty) = \bigcup_{j=1}^\infty\mathcal{ND}(\nu_j). 
 \]

\begin{lemma}\label{GammaExist}
There is a sequence of Lipschitz functions $A_n: \mathbb R\rightarrow \mathbb R$ with $\|A_n'\|_\infty \le \frac{1}{4}$ and its (rotated) graph $\Gamma_n$ such that if $\Gamma = \cup_n \Gamma_n$ and $\mu = h\mathcal H^1 |_{\Gamma} + \mu_s$ is the Radon-Nikodym decomposition with respect to $\mathcal H^1 |_{\Gamma}$, where $h\in L^1(\mathcal H^1 |_{\Gamma})$ and $\mu_s\perp \mathcal H^1 |_{\Gamma}$, then
\begin{eqnarray}\label{FCEq}
\ \mathcal {ND}(\mu ) \approx \mathcal N (h), ~\gamma-a.a., 
\end{eqnarray}
\begin{eqnarray}\label{FCEq01}
 \ \mathcal {ND}(\nu_j, 1\le j <\infty ) \approx \mathcal {ND}(\mu),~\gamma-a.a.,
 \end{eqnarray}
and
\begin{eqnarray}\label{FCEq02}
 \ \mathcal {ZD}(\nu_j, 1\le j <\infty )  \approx \mathcal {ZD}(\mu),~\gamma-a.a..
 \end{eqnarray}
\end{lemma}

\begin{proof} From Lemma \ref{lemmaBasic0} (5) (a), we have
 \[
 \ \mathcal{ND}(\mu) = \bigcup_n \mathcal{ND}(\mu, n) \cup \mathbb Q_\mu.
 \] 
Thus, \eqref{FCEq} follows from Lemma \ref{lemmaBasic0} (5) (a). The existence of $\{\Gamma_n\}$ follows from Lemma \ref{lemmaBasic0} (4).
 
It is clear, from Lemma \ref{lemmaBasic0} (5), that $\mathcal {ND}(\nu_j) \approx \mathcal N(g_jh),~\gamma-a.a.$ for $j\ge 1$. Hence,
 \[
 \ \mathcal {ND}(\nu_j, 1\le j <\infty ) \approx \bigcup_{j=1}^\infty \mathcal N(g_jh) \approx \mathcal N(h) \approx \mathcal {ND}(\mu),~\gamma-a.a.,
 \]
as $S_\mu$ is pure.
Therefore, \eqref{FCEq01} and \eqref{FCEq02} follow.
\end{proof}  
\smallskip

With Lemma \ref{GammaExist}, we can define the non-removable boundary $\mathcal F$ and the removable set $\mathcal R$ as the following.

\begin{definition}\label{FRDefinition1} Let $\Gamma_n$ and $h$ be as in Lemma \ref{GammaExist}. Define:
\[
\ \mathcal F_0 := \mathcal Z(\nu_j,1 \le j <\infty ),
\]
\[
 \ \mathcal F_+ :=  \bigcup_{n = 1}^\infty \mathcal Z(\nu_j,1 \le j < \infty, \Gamma_n, +) \cap \mathcal{ND}(\mu),
 \]
 \[
 \ \mathcal F_- :=  \bigcup_{n = 1}^\infty \mathcal Z(\nu_j,1 \le j < \infty, \Gamma_n, -) \cap \mathcal{ND}(\mu),
 \]
 \[
 \ \mathcal R_0 :=  \mathcal {ZD} (\mu) \cap \mathcal N(\nu_j,1 \le j <\infty ),
 \]
 \[
 \ \mathcal R_1 :=  \bigcup_{n = 1}^\infty \left ( \mathcal N(\nu_j,1 \le j < \infty, \Gamma_n, +) \cap \mathcal N(\nu_j,1 \le j < \infty, \Gamma_n, -) \right ) \cap \mathcal {ND} (\mu),
 \]
 \[
 \ \mathcal F = \mathcal F_0\cup \mathcal F_+ \cup \mathcal F_-,
 \]
and
 \[
 \ \mathcal R :=  \mathcal R_0 \cup \mathcal R_1.
 \]
$\mathcal F$ is called the non-removable boundary and $\mathcal R$ is called the removable set for $R^t(K, \mu)$. The set $\mathcal R_B = \mathcal R \setminus \text{abpe} (R^t(K,\mu))$ is called the removable boundary. 
\end{definition}
\smallskip

Theorem \ref{FCharacterization} and Corollary \ref{NRBUnique} ensure that $\mathcal F$ and $\mathcal R$ are independent of choices of $\{\Gamma_n\}$ and $\{g_n\}$ up to a set of zero analytic capacity. Therefore, they are well defined.

\begin{theorem}\label{FRProperties}
The following properties hold:
 \begin{eqnarray}\label{FCEq7}
 \ \mathcal F_0 \subset \mathcal {ZD} (\mu), ~\mathcal F_0\cap (\mathcal F_+\cup \mathcal F_-) = \emptyset,~\gamma-a.a.,
 \end{eqnarray}
and 
\begin{eqnarray}\label{FCEq8}
 \mathcal R_0\cap \mathcal R_1 =\emptyset,~ \mathcal R \cap \mathcal F =\emptyset,~ \mathcal R \cup \mathcal F =\mathbb C, ~\gamma-a.a..  
 \end{eqnarray}
\end{theorem}
\smallskip

We need a couple of lemmas to complete the proof. The lemmas will also be used in later sections.

\begin{lemma} \label{lemmaBasic6}
Suppose that $E$ is a compact subset with $\gamma(E) > 0$. Then there exists a finite positive measure $\eta$ satisfying:

(1) $\text{spt}(\eta)\subset E $ and $\mathcal C_*(g_n\mu ), \mathcal M_{g_n\mu}  \in L^\infty(\eta)$;

(2) $\eta\in \Sigma(E)$, Cauchy transform of $\eta$, $\mathcal C\eta$, is a bounded operator on $L^2(\eta)$ with norm $1$ and 
$c_4 \gamma(E) \le \|\eta\| \le C_4 \gamma(E)$;

(3) if $w\in L^\infty (\eta)$ and $\|\mathcal C_{\epsilon}(w\eta)\|_{L^\infty (\mathbb C)}\le 1$ for all $\epsilon > 0$,  then there exists a subsequence $f_k(z) = \mathcal C_{\epsilon_k}(w\eta)(z)$ such that $f_k$ converges to $f\in L^\infty(\mu)$ in weak$^*$ topology, $f_k(\lambda) $ converges to $f(\lambda) = \mathcal C(w\eta)(\lambda)$ uniformly on any compact subset of $E^c$ as $\epsilon_k\rightarrow 0$, 
 \begin{eqnarray}\label{lemmaBasic6Eq1}
 \ \int f(z) g_n(z)d\mu (z) = - \int \mathcal C(g_n\mu) (z) w(z) d\eta (z),
 \end{eqnarray}
and for $\lambda\in E^c$,
 \begin{eqnarray}\label{lemmaBasic6Eq2}
 \ \int \dfrac{f(z) - f(\lambda)}{z - \lambda} g_n(z)d\mu (z) = - \int \mathcal C(g_n\mu) (z) w(z) \dfrac{d\eta (z)}{z - \lambda}.
 \end{eqnarray} 
\end{lemma}

\begin{proof}
From Lemma \ref{lemmaBasic0} (10), we find $E_1\subset E$ such that $\gamma(E\setminus E_1) < \frac{\gamma(E)}{2A_T}$ and $\mathcal C_*(g_n\mu )(z), \mathcal M_{g_n\mu}(z) \le M_n < \infty$ for $z\in E_1$. Using Theorem \ref{TTolsa} (2), we get
 \[
 \ \gamma(E_1) \ge \dfrac{1}{A_T}\gamma(E) - \gamma(E\setminus E_1) \ge \dfrac{1}{2A_T}\gamma(E).
 \]
Using Theorem \ref{TTolsa} (1) \eqref{GammaEq2}, there exists a positive measure $\eta$ with $\text{spt}(\eta)\subset E_1 \subset E$ such that $\eta\in \Sigma(E)$, $\mathcal C\eta$ is bounded on $L^2(\eta)$ with norm $1$, and $\|\eta\| \ge c_4 \gamma (E)$.  (1) and (2) are proved.

Clearly,
 \begin{eqnarray}\label{lemmaBasicEq3}
 \  \int \mathcal C_\epsilon(w\eta)(z) g_nd\mu  = - \int \mathcal C_\epsilon(g_n\mu)(z) wd\eta
 \end{eqnarray}
for $n = 1,2,..$. We can choose a sequence $f_k(\lambda) = \mathcal C_{\epsilon_k}(w\eta)(\lambda)$ that converges to $f$ in $L^\infty(\mu)$ weak$^*$ topology and $f_k(\lambda)$ uniformly tends to $f(\lambda)$ on any compact subset of $E^c$. On the other hand, $|\mathcal C_{\epsilon_k}(g_n\mu)(z) | \le M _n,~ \eta |_E-a.a.$ and $\lim_{k\rightarrow \infty} \mathcal C_{\epsilon_k}(g_n\mu)(z)  = \mathcal C(g_n\mu)(z) ,~ \eta -a.a.$. We apply Lebesgue dominated convergence theorem to the right hand side of \eqref{lemmaBasicEq3} and get equation \eqref{lemmaBasic6Eq1} for $n=1,2,..$. For equation \eqref{lemmaBasic6Eq2}, let $\lambda\notin E$ and $d = \text{dist}(\lambda, E)$,
for $z\in B(\lambda, \frac{d}{2})$ and $\epsilon < \frac{d}{2}$, we have
 \[
 \ \left |\dfrac{\mathcal C_\epsilon (w\eta)(z) - f(\lambda)}{z - \lambda} \right |\le \left |\mathcal C_\epsilon  \left (\dfrac{w(s)\eta (s)}{s - \lambda} \right ) (z) \right | \le 
\dfrac{2}{d^2} \|w\|_\infty \|\eta\|. 
 \]
For $z\notin B(\lambda, \frac{d}{2})$ and $\epsilon < \frac{d}{2}$,
 \[
 \ \left |\dfrac{\mathcal C_\epsilon (w\eta)(z) - f(\lambda)}{z - \lambda} \right | \le \dfrac{4}{d}.
 \]
Thus, we can replace the above proof for the measure $\dfrac{w(s)\eta (s)}{s - \lambda}$. In fact, we can choose a subsequence  $\{\mathcal C_{\epsilon_{k_j}} (w\eta)\}$ such that $e_{k_j}(z) = \dfrac{\mathcal C_{\epsilon_{k_j}} (w\eta)(z) - f(\lambda)}{z - \lambda}$ converges to $e(z)$ in weak$^*$ topology. Clearly, $(z-\lambda)e_{k_j}(z)  + f(\lambda) = \mathcal C_{\epsilon_{k_j}} (w\eta)(z)$ converges to $(z-\lambda)e(z)  + f(\lambda) = f(z)$ in weak$^*$ topology.  
On the other hand, the equation \ref{lemmaBasicEq3} becomes
 \begin{eqnarray}\label{lemmaBasicEq31}
 \  \int \mathcal C_{\epsilon_{k_j}}(\dfrac{w(s)\eta(s)}{s-\lambda})(z) g_nd\mu  = - \int \mathcal C_{\epsilon_{k_j}}(g_n\mu)(z) \dfrac{w(z)d\eta(z)}{z-\lambda}
 \end{eqnarray}
and for $\epsilon_{k_j} < \frac{d}{2}$, we have
 \begin{eqnarray}\label{lemmaBasicEq32}
 \ \begin{aligned}
\ & \left | \mathcal C_{\epsilon_{k_j}}(\dfrac{w(s)\eta(s)}{s-\lambda})(z) - e_{k_j}(z) \right |\\
\ \le &\begin{cases}0, & z\in B(\lambda, \frac{d}{2}), \\ \dfrac{2}{d^2}\|w\|_\infty \eta(B(z, \epsilon_{k_j})), & z\notin B(\lambda, \frac{d}{2})\end{cases},
\ \end{aligned}
 \end{eqnarray}
which goes to zero as $\epsilon_{k_j} \rightarrow 0$. Combining \eqref{lemmaBasicEq31}, \eqref{lemmaBasicEq32}, and Lebesgue dominated convergence theorem, we prove the equation \eqref{lemmaBasic6Eq2}. (3) is proved.    
\end{proof}
\smallskip

\begin{lemma}\label{lemmaBasic7} Suppose $f\in R^t(K,\mu)$ and $\{r_n\}\subset Rat(K)$ satisfying: 
 \[
 \ \|r_n - f\|_{L^t(\mu)}\rightarrow 0,~ r_n \rightarrow  f, ~\mu-a.a..
 \]
 Then:

(1) For $g\perp R^t(K, \mu)$ and $\epsilon > 0$, there exists a subset $A_\epsilon$ with $\gamma(A_\epsilon) < \epsilon$ and a subsequence $\{r_{n_k}\}$ such that $\{r_{n_k}\mathcal C(g\mu)\}$ uniformly converges to $\mathcal C(fg\mu)$ on $K \setminus A_\epsilon$.

(2) For $g\perp R^t(K, \mu )$, there exists a subset $\mathbb Q$ with $\gamma (\mathbb Q) = 0$  and a subsequence $\{r_{n_m}\}$ such that $r_{n_m}(\lambda ) \mathcal C(g\mu)(\lambda )$ converges to $\mathcal C(fg\mu)(\lambda )$ for $\lambda \in K\setminus \mathbb Q$.

(3) For $\epsilon > 0$, there exists a subset $A_\epsilon$ with $\gamma(A_\epsilon) < \epsilon$ and a subsequence $\{r_{n_k}\}$ such that $\{r_{n_k}\mathcal C(\nu_j)\}$ uniformly converges to $\mathcal C(f\nu_j)$ on $K \setminus A_\epsilon$ for all $j\ge 1$.
\end{lemma}

\begin{proof} From Corollary \ref{ZeroAC}, we let $\mathbb Q_1\subset K$ with $\gamma(\mathbb Q_1) = 0$ such that the principal values of $\mathcal{C}(g\mu)(z )$ and $\mathcal{C}(fg\mu)(z )$ exist for $z\in K\setminus \mathbb Q_1$. Hence,
\[
 \ r_n(z)\mathcal{C}(g\mu)(z ) = \mathcal{C}(r_ng\mu)(z )
 \]
for $z\in K\setminus \mathbb Q_1$. Define:
 \[
 \ A_{nm} = \left\{z\in K\setminus \mathbb Q_1:~ |r_n(z)\mathcal{C}(g\mu)(z ) - \mathcal{C}(fg\mu)(z )|\ge \dfrac{1}{m} \right\}
 \]
Since
 \[
 \begin{aligned}
 \ |r_n(z)\mathcal{C}(g\mu)(z ) - \mathcal{C}(fg\mu)(z )| & = \lim_{\epsilon \rightarrow 0} | \mathcal{C} _{\epsilon}(r_ng\mu)(z ) - \mathcal{C} _{\epsilon}(fg\mu)(z )|\\
& \le \mathcal{C} _{*}((r_n - f)g\mu)(z ). 
 \end{aligned}
 \]
Applying Theorem \ref{TTolsa} (3), we get
 \[
 \ \gamma (A_{nm}) \le \gamma \left\{\mathcal{C} _{*}((r_n - f)g\mu)(z ) \ge \dfrac{1}{m} \right\} \le C_Tm \|g\|\|r_n-f\|_{L^t(\mu)}.  
 \]
Choose $n_m$ so that $\|r_{n_m}-f\|_{L^t(\mu)} \le \frac{1}{m\|g\|2^m}$ and we have
$\gamma (A_{n_mm}) \le \frac{C_T}{2^m}$.

(1): Set $B_k = \cup_{m=k}^\infty A_{n_mm}$.
Applying Theorem \ref{TTolsa} (2), there exists $k_0$ so that $\frac{A_TC_T}{2^{k_0 - 1}} < \epsilon$ and 
 \[
 \ \gamma (B_{k_0}\cup \mathbb Q_1) \le A_T\sum_{m=k}^\infty \gamma (A_{n_mm})\le A_TC_T \sum_{m=k_0}^\infty \dfrac{1}{2^m} < \epsilon.
 \]
Then on $(B_{k_0}\cup \mathbb Q_1)^c$, $r_{n_m}(\lambda ) \mathcal C(g\mu)(\lambda )$ converges to $\mathcal C(fg\mu)(\lambda )$ uniformly. 

(2): 
Set $\mathbb Q_2 = \cap_{k=1}^\infty B_k$. Clearly, $\gamma (\mathbb Q_2) = 0$. If $\mathbb Q=\mathbb Q_1\cup \mathbb Q_2$, then on $K\setminus \mathbb Q$, $r_{n_m}(\lambda ) \mathcal C(g\mu)(\lambda )$ converges to $\mathcal C(fg\mu)(\lambda )$. 

(3) Using (1), we find $A^1_\epsilon$ and a subsequence $\{r_{n,1}\}$ of $\{r_{n}\}$ such that $\gamma(A^1_\epsilon) < \frac{\epsilon}{2A_T}$ and $\{r_{n,1}\mathcal C(\nu_1)\}$ uniformly converges to $\mathcal C(f\nu_1)$ on $K \setminus A^1_\epsilon$. Then we find $A^2_\epsilon$ and a subsequence $\{r_{n,2}\}$ of $\{r_{n,1}\}$ such that $\gamma(A^2_\epsilon) < \frac{\epsilon}{2^2A_T}$ and $\{r_{n,2}\mathcal C(\nu_2)\}$ uniformly converges to $\mathcal C(f\nu_2)$ on $K \setminus A^2_\epsilon$. Therefore, we have a subsequence $\{r_{n,n}\}$ such that $\{r_{n,n}\mathcal C(\nu_j)\}$ uniformly converges to $\mathcal C(f\nu_j)$ on $K \setminus A_\epsilon$ for all $j \ge 1$, where $A_\epsilon = \cup_n A^n_\epsilon$ and $\gamma(A_\epsilon) < \epsilon$ by Theorem \ref{TTolsa} (2).     
\end{proof}
\smallskip

\begin{proof} (Theorem \ref{FRProperties}):
Using Theorem \ref{GPTheorem1}, one can easily verify that $\mathcal R_1\cap (\mathcal F_+\cup \mathcal F_-)= \emptyset, ~ \gamma-a.a.$ (or see the proof of Theorem \ref{FCharacterization} under $\lambda \in \mathcal {ND}(\mu)$).
Therefore, we only need to show 
 \[
 \ \mathcal F_0 \subset  \mathcal{ZD}(\mu), ~ \gamma-a.a.
 \]
and others are trivial.
Suppose that 
 \begin{eqnarray}\label{FRPropertiesAssump}
 \ \gamma(E:= \mathcal Z(\nu_j, 1\le j < \infty) \cap \mathcal{ND}(\mu)) > 0.
 \end{eqnarray}
Without loss of generality, we assume that $\mu (E\cap \Gamma_1) > 0$ and $E\cap \Gamma_1 \subset \mathcal N(h)$, thus, $\mathcal H^1(E\cap \Gamma_1) > 0$. Since $S_\mu$ is pure, we get $\mathcal H^1(E\cap \Gamma_1\cap (\cup_n \mathcal N(g_n h))) > 0$. Without loss of generality, we assume that 
 \[
 \ \mathcal H^1(E\cap \Gamma_1\cap \mathcal N(g_1 h )) > 0.
 \]

Applying Lemma \ref{lemmaBasic6}, we can find a compact set $E_1\subset E\cap \Gamma_1\cap \mathcal N(g_1 h )$ and a measure $\eta$ with $\text{spt}(\eta) \subset E_1$ satisfying (1)-(3) of Lemma \ref{lemmaBasic6}. Clearly, $\eta$ is absolutely continuous with respect to $\mathcal H^1|_{\Gamma_1}$ since $\eta$ is linear growth.      
From Lemma \ref{lemmaBasic0} (1), we can find a function $w$ supported on $E_1$ with $0\le w \le 1$ such that
 \[
 \ \eta(E_1) \le 2 \int w(z) d\eta(z), ~ \|\mathcal C_\epsilon(w\eta)\|_{L^\infty(\mathbb C)} \le C_5. 
 \]
From (3) of Lemma \ref{lemmaBasic6}, we have a function $f\in L^\infty(\mu)$ so that 
\[
 \ \int f(z) g_n(z)d\mu (z) = - \int \mathcal C(g_n\mu) (z) w(z) d\eta (z) = 0
 \]
and
 \[
 \ \int \dfrac{f(z) - f(\lambda)}{z - \lambda} g_n(z)d\mu (z) = - \int \mathcal C(g_n\mu) (z) w(z) \dfrac{d\eta (z)}{z - \lambda} = 0
 \]
for $\lambda\in E_1^c$, where $f(\lambda) = \mathcal C(w\eta)(\lambda)$. Hence,
\begin{eqnarray}\label{FCForEq6}
 \ f(\lambda)\mathcal C (g_n\mu)(\lambda) = \mathcal C (fg_n\mu)(\lambda),~ \lambda\in E_1^c,~ \gamma-a.a.
 \end{eqnarray}

In the remaining proof, we apply Theorem \ref{GPTheorem1} to $\mathcal C (g_1\mu)$ and $\mathcal C (fg_1\mu)$ for $\lambda\in E_1$ and use \eqref{FCForEq6} to get a contradiction. 

Clearly,
 \[
 \ \mathcal C (g_1\mu)(\lambda) = 0, ~\lambda\in E_1,~ \mathcal H^1 |_{\Gamma_1}-a.a.
 \]
and from Lemma \ref{lemmaBasic7}, we have
\[
 \ \mathcal C (fg_1\mu)(\lambda) = \lim_{m\rightarrow \infty}r_{n_m}(\lambda)\mathcal C (g_1\mu)(\lambda) = 0, ~\lambda\in E_1,~ \mathcal H^1 |_{\Gamma_1}-a.a.,
 \]  
where we use the fact that $\gamma  |_{\Gamma_1}$ is equivalent to $\mathcal H^1 |_{\Gamma_1}$ (see Lemma \ref{lemmaBasic0} (7)). Without loss of generality, we assume that the rotation angle for $\Gamma_1$ is zero. Hence, 
$v^+ (g_1\mu, \Gamma_1, \lambda) = \dfrac{1}{2}(g_1hL^{-1})(\lambda)$, 
$v^+ (fg_1\mu, \Gamma_1, \lambda) = \dfrac{1}{2}(fg_1hL^{-1})(\lambda)$,
$v^- (g_1\mu, \Gamma_1, \lambda) = -\dfrac{1}{2}(g_1hL^{-1})(\lambda)$,
 and 
$v^- (fg_1\mu, \Gamma_1, \lambda) = -\dfrac{1}{2}(fg_1hL^{-1})(\lambda)$
 for $\lambda \in E_1$ $\mathcal H^1 |_{\Gamma_1}-a.a.$.

Set $G_\epsilon = \{|\mathcal C (g_1\mu) - v^+ (g_1\mu, \Gamma_1, \lambda)| >\epsilon\}$ and $F_\epsilon = \{|\mathcal C (fg_1\mu) - v^+ (fg_1\mu, \Gamma_1, \lambda)| >\epsilon\}$. From Theorem \ref{GPTheorem1}, for $\lambda \in E_1$ $\mathcal H^1 |_{\Gamma_1}-a.a.$, 
 \[
 \ \lim_{\delta\rightarrow 0}\dfrac{\gamma(UC(\lambda, \alpha)\cap B(\lambda, \delta)\cap(G_\epsilon\cup F_\epsilon))}{\delta} = 0, 
 \]
and then
 \[
 \ \underset{\delta\rightarrow 0}{\overline{\lim}}\dfrac{\gamma(UC(\lambda, \alpha)\cap B(\lambda, \delta)\cap G_\epsilon ^c \cap F_\epsilon ^c)}{\delta} > 0. 
 \]
Therefore, there exists $\lambda_k\in UC(\lambda, \alpha)\rightarrow \lambda$ such that $f(\lambda_k)\mathcal C (g_1\mu)(\lambda_k) = \mathcal C (fg_1\mu)(\lambda_k)$ 
(from \eqref{FCForEq6}),
$\lim_{k\rightarrow \infty} \mathcal C (g_1\mu)(\lambda_k) = \frac{1}{2}(g_1hL^{-1})(\lambda)$,
and $\lim_{k\rightarrow \infty} \mathcal C (fg_1\mu)(\lambda_k) = \dfrac{1}{2}(fg_1hL^{-1})(\lambda)$.
This implies $\lim_{k\rightarrow \infty} f(\lambda_k) = \dfrac{1}{2}f(\lambda)$. Similarly, there exists a sequence $\theta_k\in LC(\lambda, \alpha)\rightarrow \lambda$ such that $\lim_{k\rightarrow \infty} f(\theta_k) = \dfrac{1}{2}f(\lambda)$. Therefore,
 \[
 \ \lim_{k\rightarrow \infty} f(\lambda_k) = \lim_{k\rightarrow \infty} f(\theta_k),~ \lambda \in E_1,~ \mathcal H^1 |_{\Gamma_1}-a.a.. 
 \]
Now we apply the classical Plemelj's formula to $f$, Lemma \ref{PFLipschitz}, we get
$v^+(w\eta, \Gamma_1, \lambda) = v^-(w\eta, \Gamma_1, \lambda)$
which implies that
 \[
 \ (w\dfrac{d\eta}{d\mathcal H^1|_{\Gamma_1}})(\lambda) = 0, ~ \lambda\in E_1, ~ \mathcal H^1|_{\Gamma_1}-a.a.. 
 \] 
This is a contradiction to \eqref{FRPropertiesAssump}.  The proof is finished.
\end{proof}
\smallskip

\begin{example}\label{FCExample} 
Let $K = \overline{\mathbb D} \setminus \cup_{n=1}^\infty B(\lambda_n,\delta_n)$, where $B(\lambda_n,\delta_n) \subset \mathbb D$, $\overline{B(\lambda_n,\delta_n)} \cap \overline{B(\lambda_m,\delta_m)} = \emptyset$ for $n \ne m$, and $\sum \delta_n < \infty$. If $\mu$ is the sum of the arclength measures of the unit circle and all small circles, then $S_\mu$ on $R^t(K, \mu)$ is pure, 
 \[
 \ \mathcal F_0 = K^c, ~ \mathcal F_+ = \emptyset,~ \mathcal F_- = \mathbb T \cup \bigcup_{n=1}^\infty \partial B(\lambda_n,\delta_n),
 \]
and 
 \[
 \ \mathcal R_0 = K\setminus \mathcal F = K\setminus \mathcal F_-,~\mathcal R_1 = \emptyset
 \]
(if $K$ has no interior, then $K$ is a Swiss cheese set).  
\end{example}

\begin{proof}
Let $dz$ be the usual complex measure on $\mathbb T$ (counter clockwise) and on each $\partial B(\lambda_n,\delta_n)$ (clockwise). Then $\int r(z) d z = 0$ for $r\in Rat(K)$. If $g = \frac{dz}{d\mu}$, then $g\perp R^t(K, \mu)$. $S_\mu$ is pure as $|g| > 0, ~\mu-a.a.$.

By Lemma \ref{lemmaBasic0} (5) (b), there exists $\mathbb Q \subset K$ with $\gamma(\mathbb Q) = 0$, if 
 \[
 \ D :=  K \setminus \left ( \mathbb T \cup \bigcup_{n=1}^\infty \partial B(\lambda_n,\delta_n) \right ) \setminus \mathbb Q,
 \]
then $\Theta_\mu(\lambda) = 0$ for $\lambda\in D$. Therefore, for $\lambda\in D$ and $\epsilon  < \frac{1 - |\lambda|}{2}$,
 \[
 \ \begin{aligned}
 \ \left | \dfrac{1}{2\pi i} \mathcal C_\epsilon (dz)(\lambda) - 1\right | = & \left | \dfrac{1}{2\pi i} \sum_{n = 1}^\infty \int_{B(\lambda_n, \delta_n)\cap \partial B(\lambda, \epsilon)} \dfrac{1}{z-\lambda}dz\right | \\ 
\ \le &\dfrac {\mu (B(\lambda, 2\epsilon))}{\epsilon}.
 \ \end{aligned}
 \]
Hence, the principle value $\mathcal C(dz)(\lambda) = 2\pi i,~\gamma|_D-a.a.$. This implies, by Definition \ref{FRDefinition1},  that $D \subset \mathcal R_0$. It is clear that $\mathbb T \cup \cup_{n=1}^\infty \partial B(\lambda_n,\delta_n) \subset \mathcal F_-$ since $\mathcal C (g\mu) (z) = 0$ for $g\perp R^t(K,\mu)$ and $z\in K^c$ (applying Definition \ref{FRDefinition1} for $v^-(g\mu, \Gamma_n, \lambda) = 0$). Therefore,
 \[
 \ \mathcal F_- \approx \mathbb T \cup \bigcup_{n=1}^\infty \partial B(\lambda_n,\delta_n),~ \mathcal R_0 \approx K\setminus \mathcal F_-,~ \gamma-a.a..
 \] 
\end{proof}
\bigskip

\section{Uniqueness of $\mathcal F$ and $\mathcal R$}
\bigskip

Define
 \begin{eqnarray}\label{ENDefinition}
\ \begin{aligned}\label{EDefinition}
 \ & \mathcal E(\nu_j, 1\le j \le N) \\
\  = &\left \{\lambda : ~\lim_{\epsilon \rightarrow 0} \mathcal C_\epsilon(\nu_j)(\lambda )\text{ exists, } \max_{1\le j\le N} |\mathcal C (\nu _j)(\lambda ) | \le \frac{1}{N} \right \}.
 \ \end{aligned}
 \end{eqnarray}
\smallskip

The theorem below, which will be used to prove Theorem \ref{ABPETheorem},  is a useful characterization of $\mathcal F$ and $\mathcal R$.

\begin{theorem}\label{FCharacterization}
There is a subset $\mathbb Q \subset \mathbb C$ with $\gamma(\mathbb Q) = 0$ such that if $\lambda \in\mathbb C \setminus \mathbb Q$, then $\lambda \in \mathcal F$ if and only if   
 \begin{eqnarray}\label{FCEq4}
 \ \underset{\delta\rightarrow 0}{\overline{\lim}}\dfrac{\gamma(B(\lambda, \delta)\cap\mathcal E(\nu_j,1 \le j \le N))}{\delta} > 0 
 \end{eqnarray}
for all $N \ge 1$. Consequently, $\mathcal F$ and $\mathcal R$ do not depend on choices of $\{\Gamma_n\}$ (as in Lemma \ref{GammaExist}) up to a set of zero analytic capacity. 
\end{theorem}

\begin{proof}
We first prove that there exists $\mathbb Q_1$ with $\gamma(\mathbb Q_1) = 0$ such that if $\lambda_0\in \mathcal {ZD}(\mu) \setminus \mathbb Q_1$, then $\lambda_0\in \mathcal F$ if and only if $\lambda_0$ satisfies \eqref{FCEq4}.

From Definition \ref{FRDefinition1} and Theorem \ref{FRProperties} \eqref{FCEq7}, we get $\mathcal {ZD}(\mu) \approx \mathcal R_0 \cup \mathcal F_0, ~\gamma-a.a.$. There exists $\mathbb Q_1$ with $\gamma(\mathbb Q_1) = 0$ such that for $\lambda_0\in \mathcal {ZD}(\mu) \setminus \mathbb Q_1$, $\Theta_{ \nu _j}(\lambda _0) = 0$,
$\mathcal C(\nu _j) (\lambda_0) = \lim_{\epsilon\rightarrow 0} \mathcal C_\epsilon (\nu _j) (\lambda_0)$ exists, and $\mathcal C(\nu _j) (z)$ is $\gamma$-continuous at $\lambda_0$ (Lemma \ref{CauchyTLemma})  for all $j\ge 1$.

If $\lambda_0\in \mathcal R_0$, then there exists $j_0$ such that $\mathcal C(\nu _{j_0}) (\lambda_0) \ne 0$. Let $\epsilon_0 = \frac 12 |\mathcal C(\nu _{j_0}) (\lambda_0)|$, we obtain that, for $N > N_0 := \max(j_0,\frac{1}{\epsilon_0} + 1)$,
 \[
 \ \mathcal E(\nu_j, 1\le j \le N)) \subset \{|\mathcal C(\nu _{j_0}) (z) - \mathcal C(\nu _{j_0}) (\lambda_0)| > \epsilon_0 \}.
 \]
Therefore, by Lemma \ref{CauchyTLemma}, for $N > N_0$,  
 \[
 \ \lim_{\delta\rightarrow 0} \dfrac{\gamma (B(\lambda_0 , \delta ) \cap \mathcal E(\nu_j, 1\le j \le N))}{\delta} = 0.
 \]
Thus, $\lambda_0$ does not satisfy \eqref{FCEq4}. 

Now for $\lambda_0\in \mathcal F_0$, $\mathcal C(\nu _j) (\lambda_0) = 0$ for all $j \ge 1$.
Using Lemma \ref{CauchyTLemma} and Theorem \ref{TTolsa} (2), we get
 \[
 \ \begin{aligned}
 \ &\lim_{\delta\rightarrow 0} \dfrac{\gamma (B(\lambda _0, \delta) \setminus \mathcal E(\nu_j, 1\le j \le N))}{\delta } \\
 \ \le & A_T \sum_{j=1}^N \lim_{\delta\rightarrow 0} \dfrac{\gamma (B(\lambda _0, \delta) \cap \{|\mathcal C(\nu _j)(z) - \mathcal C(\nu _j) (\lambda_0)| \ge \frac{1}{N}\})}{\delta } \\
 \ = & 0. 
 \ \end{aligned}
 \]
Hence, $\lambda_0$ satisfies \eqref{FCEq4}.

We now prove that there exists $\mathbb Q_2$ with $\gamma(\mathbb Q_2) = 0$ such that if $\lambda_0\in \mathcal {ND}(\mu) \setminus \mathbb Q_2$, then $\lambda_0\in \mathcal F$ if and only if $\lambda_0$ satisfies \eqref{FCEq4}.

From Definition \ref{FRDefinition1}, we get $\mathcal {ND}(\mu) \approx \mathcal R_1 \cup (\mathcal F_+ \cup \mathcal F_-), ~\gamma-a.a.$. There exists $\mathbb Q_2$ with $\gamma(\mathbb Q_2) = 0$ such that for $\lambda_0\in \mathcal {ND}(\mu) \cap \Gamma_n\setminus \mathbb Q_2$, 
$v^0(\nu _j, \Gamma_n, \lambda_0) = \mathcal C(\nu _j) (\lambda_0) = \lim_{\epsilon\rightarrow 0} \mathcal C_\epsilon (\nu _j) (\lambda_0)$, $v^+(\nu _j, \Gamma_n, \lambda_0)$, and $v^-(\nu _j, \Gamma_n, \lambda_0)$ exist for all $j \ge 1$ and Theorem \ref{GPTheorem1} (b), (c), and (d) hold. Fix $n = 1$.

If $\lambda_0\in \mathcal R_1$, then $\lambda_0\notin \mathcal F_0$ and there exist integers $j_0$, $j_1$, and $j_2$ such that $v^0(\nu _{j_0}, \Gamma_1, \lambda_0) \ne 0$, $v^+(\nu _{j_1}, \Gamma_1, \lambda_0) \ne 0$, and $v^-(\nu _{j_2}, \Gamma_1, \lambda_0) \ne 0$. Set
 \[
 \ \epsilon_0 = \dfrac 12 \min (|v^0(\nu _{j_0}, \Gamma_1, \lambda_0)|, |v^+(\nu _{j_1}, \Gamma_1, \lambda_0)|, |v^-(\nu _{j_2}, \Gamma_1, \lambda_0)|), 
 \]
then for $N > N_0 := \max (j_0, j_1, j_2, \frac{1}{\epsilon_0} + 1)$,
\[
\ \Gamma_1\cap \mathcal E (\nu _j, 1 \le j \le N ) \subset D := \Gamma_1 \cap \{|\mathcal{C}(\nu _{j_0} (z ) - v^0(\nu _{j_0}, \Gamma_1, \lambda_0)| \ge \epsilon_0 \},
 \] 
\[
\ U_{\Gamma_1}\cap \mathcal E (\nu _j, 1 \le j \le N ) \subset
 E := U_{\Gamma_1}\cap \{|\mathcal{C}(\nu _{j_1})(z ) - v^+(\nu _{j_1}, \Gamma_1, \lambda_0)| \ge \epsilon_0 \},
 \]
and
 \[
\ L_{\Gamma_1}\cap \mathcal E (\nu _j, 1 \le j \le N ) \subset
 F :=  L_{\Gamma_1}\cap \{|\mathcal{C}(\nu _{j_2})(z ) - v^-(\nu _{j_2}, \Gamma_1, \lambda_0)| \ge \epsilon_0 \}.
 \]
Therefore, using Theorem \ref{TTolsa} (2) and Theorem \ref{GPTheorem1}, we get for $N > N_0$,
 \[
 \begin{aligned}
 \ &\lim_{\delta\rightarrow 0}\dfrac{\gamma(B(\lambda_0, \delta) \cap \mathcal E (\nu_j, 1 \le j \le N ))}{\delta} \\
 \ \le & A_T\left ( \lim_{\delta\rightarrow 0}\dfrac{\gamma(B(\lambda_0, \delta) \cap D)}{\delta} + \lim_{\delta\rightarrow 0}\dfrac{\gamma(B(\lambda_0, \delta) \cap E)}{\delta} + \lim_{\delta\rightarrow 0}\dfrac{\gamma(B(\lambda_0, \delta) \cap F)}{\delta}\right ) \\
 \ = &0.
 \end{aligned}
 \]
Hence, $\lambda_0$ does not satisfy \eqref{FCEq4}.   
 
For $\lambda_0\in (\mathcal F_+ \cup \mathcal F_-) \cap \Gamma_1$, we may assume that $\lambda_0\in \mathcal Z(\nu_j,1 \le j < \infty, \Gamma_1, +)\subset \Gamma_1$. Using Theorem \ref{TTolsa} (2) and Theorem \ref{GPTheorem1}, we get ($v^+(\nu _j, \Gamma_1, \lambda_0) = 0$)
 \[
 \ \begin{aligned}
 \ &\lim_{\delta\rightarrow 0}\dfrac{\gamma(B(\lambda_0, \delta) \cap U_{\Gamma_1}\setminus \mathcal E (\nu_j, 1 \le j \le N ))}{\delta} \\
 \ \le & A_T \sum_{j=1}^N\lim_{\delta\rightarrow 0}\dfrac{\gamma(B(\lambda_0, \delta) \cap U_{\Gamma_1}\cap \{|\mathcal{C}(\nu_j)(z ) - v^+(\nu _j, \Gamma_1, \lambda_0)| \ge \frac{1}{N} \})}{\delta} \\
 \ \ = & 0.
 \ \end{aligned}
 \]
This implies
 \[
 \ \begin{aligned}
 \ & \underset{\delta\rightarrow 0}{\overline \lim}\dfrac{\gamma(B(\lambda_0, \delta) \cap \mathcal E (\nu_j, 1 \le j \le N ))}{\delta} \\
\ \ge &\underset{\delta\rightarrow 0}{\overline \lim}\dfrac{\gamma(B(\lambda_0, \delta) \cap U_{\Gamma_1}\cap \mathcal E (\nu_j, 1 \le j \le N ))}{\delta} \\
 \ > &0.
\ \end{aligned} 
\] 
Hence, $\lambda_0$ satisfies \eqref{FCEq4}.

Finally, \eqref{FCEq4} does not depend on choices of $\{\Gamma_n\}$, therefore, $\mathcal F$ and $\mathcal R$ are independent of choices of $\{\Gamma_n\}$ up to a set of zero analytic capacity.  
\end{proof}
\smallskip

We now show the properties of $g_n\mu$ on $\mathcal F$ are preserved for all annihilating measures $g\mu$.  

\begin{theorem}\label{FCForR}
Let $\Gamma_n$, $\Gamma$, and $d\mu = hd\mathcal H^1 |_{\Gamma} + d\mu_s$ be as in Lemma \ref{GammaExist}. Then
for $g\perp R^t(K, \mu)$,
\begin{eqnarray}\label{FCForREq1}
\ \mathcal C(g\mu)(\lambda) = 0,~\text{for}~ \lambda\in\mathcal F_0, ~ \gamma-a.a.,  
\end{eqnarray}
\begin{eqnarray}\label{FCForREq2}
\ v^+(g\mu, \Gamma_n, \lambda) = 0,~\text{for}~ \lambda\in\mathcal F_+\cap \Gamma_n, ~ \gamma-a.a.,  
\end{eqnarray}
and
\begin{eqnarray}\label{FCForREq3}
\ v^-(g\mu, \Gamma_n, \lambda) = 0,~\text{for}~ \lambda\in\mathcal F_-\cap \Gamma_n, ~ \gamma-a.a.  
\end{eqnarray}
for $n \ge 1$.
\end{theorem}
\smallskip

\begin{proof} (Theorem \ref{FCForR} Equation \eqref{FCForREq1}):
Assume that there exists a compact subset $F_1\subset \mathcal F_0$ such that $\gamma (F_1) > 0$ and  
 \begin{eqnarray}\label{FCForRAssump}
 \ Re(\mathcal C(g\mu)) > 0,~ \lambda\in F_1.
 \end{eqnarray}
  Using Lemma \ref{lemmaBasic0} (10), we assume that $\mathcal C_*(g\mu) (\lambda) < M_1$ for $\lambda\in F_1$. We now apply Lemma \ref{lemmaBasic6} and conclude that there exists a finite positive measure $\eta$ with $\text{spt}(\eta) \subset F_1$ satisfying (1)-(3) of Lemma \ref{lemmaBasic6}. From Lemma \ref{lemmaBasic0} (1), we can find a function $w$ supported on $F_1$ with $0\le w \le 1$ such that
 \[
 \ \eta(F_1) \le 2 \int w(z) d\eta(z), ~ \|\mathcal C_\epsilon(w\eta)\|_{L^\infty(\mathbb C)} \le C_6, 
 \]
where $\epsilon > 0$. From (3) of Lemma \ref{lemmaBasic6}, we have a function $f\in L^\infty(\mu)$ 
 with $f(\lambda) = \mathcal C (w\eta)(\lambda)$ for $\lambda\in F_1^c$
 so that 
\begin{eqnarray}\label{acTheoremForGeq1}
 \ \int f(z) g_n(z)d\mu (z) = - \int \mathcal C(g_n\mu) (z) w(z) d\eta (z) = 0,
 \end{eqnarray}
which implies $f\in R^t(K,\mu)$. Hence,
\[
 \ \int f(z) g(z)d\mu (z) = 0.
 \]
Let $\epsilon_k$ be as in (3) of Lemma \ref{lemmaBasic6}, we get $\mathcal C_{\epsilon_k}(w\eta)$ coverages to $f$ in $L^\infty (\mu)$ weak$^*$ topology and 
 \[
 \  \int \mathcal C_{\epsilon_k}(w\eta)(z) gd\mu  = - \int \mathcal C_{\epsilon_k}(g\mu)(z) w(z)d\eta(z).
 \]
Using Lebesgue dominated convergence theorem for the right hand side of above equation, we see that
 \[
 \   \int Re(\mathcal C(g\mu)(z)) w(z)d\eta(z) = - Re\left(\int f(z)g(z)d\mu(z)\right) = 0.
 \] 
Hence, $Re(\mathcal C(g\mu)(z)) w(z) = 0, ~ \eta-a.a.$. This is a contradiction to \eqref{FCForRAssump}.
\end{proof}
\smallskip

\begin{proof} (Theorem \ref{FCForR} Equation \eqref{FCForREq2}):
Assume that the rotation angle of $\Gamma_1$ is zero and there exists a compact subset $F_1\subset \mathcal F_+\cap \Gamma_1$ such that $\gamma (F_1) > 0$ and 
 \begin{eqnarray}\label{FCForRAssump2}
 \ Re(v^+(g\mu, \Gamma_1, \lambda )) > 0, ~\lambda\in F_1.
 \end{eqnarray}
 Using Lemma \ref{lemmaBasic0} (10), we assume that $\mathcal C_*(g\mu) (\lambda) < M_1$ for $\lambda\in F_1$.

Let $\eta$ be a finite positive measure with $\text{spt}(\eta) \subset F_1$ satisfying (1)-(3) of Lemma \ref{lemmaBasic6}.
Since $F_1 \subset \mathcal {ND} (\mu)$ (by Definition \ref{FRDefinition1}), we see that $\eta$ is absolutely continuous with respect to $\mathcal H^1 |_{F_1}$ as $\eta$ is linear growth. Let $\eta = h_\eta \mathcal H^1 |_{F_1}$, we may assume $h_\eta \le M_0$ for some constant $M_0$.

Clearly $h > 0$ on $F_1$. We may choose a smaller set (still use $F_1$) so that $h\ge c_0$ on $F_1$ for some constant  $c_0$. Hence, $\eta = h_\eta h^{-1}\mu |_{F_1}$ and $|h_\eta h^{-1}|\le \frac{M_0}{c_0}$.

With above $\eta$ and $F_1$, we choose $g_{n_k}$ coverages to $g$ in $L^{s}(\mu)$ weak$^*$ topology.
Since $v^+(g_{n_k}\mu, \Gamma_1, \lambda ) = \mathcal C(g_{n_k}\mu)(\lambda ) + \frac{1}{2}(g_{n_k}hL^{-1})(\lambda ) = 0 $, using the similar arguments as in the proof of \eqref{FCForREq1}, we see that the equation \eqref{acTheoremForGeq1} becomes:
\[
 \ \begin{aligned}
 \ \int f(z) g_{n_k}(z)d\mu (z) = & - \int \mathcal C(g_{n_k}\mu) (z) w(z) d\eta (z)\\
 \  = & \frac{1}{2}\int g_{n_k}(z) (hL^{-1})(z)w(z) h_\eta h^{-1}\chi_{F_1} d\mu (z).
\ \end{aligned} 
\]
By taking $k\rightarrow \infty$, we get
\[
\ \begin{aligned}
\ & \frac{1}{2}\int g(z) (hL^{-1})(z)w(z) d\eta (z) \\
\  = &\int f(z) g(z)d\mu (z) \\
\  = &\lim_{k \rightarrow\infty}\int \mathcal C_{\epsilon_k}(w\eta)(z) gd\mu  \\
\  = &- \lim_{k \rightarrow\infty}\int \mathcal C_{\epsilon_k}(g\mu)(z) w(z)d\eta(z)\\
\  = & - \int \mathcal C(g\mu) (z) w(z) d\eta (z)
\ \end{aligned}
\]
where $\mathcal C_{\epsilon_k}(w\eta)$ coverages to $f$ in weak$^*$ topology and Lebesgue dominated convergence theorem is applied for $\mathcal C_{\epsilon_k}(g\mu)$.
Therefore, 
\[
 \ \int Re(v^+(g\mu, \Gamma_1, z)) w(z) d\eta (z) = Re \left (\int v^+(g\mu, \Gamma_1, z) w(z) d\eta (z) \right ) = 0.
 \]
Thus, $Re(v^+(g\mu, \Gamma_1, z)) w(z) = 0,~\eta-a.a.$. This arrives a contradiction to \eqref{FCForRAssump2}.
\end{proof}
\smallskip

The proof of Theorem \ref{FCForR} Equation \eqref{FCForREq3} is the same as \eqref{FCForREq2}.
\smallskip

As a simple application of Theorem \ref{FCForR} \eqref{FCForREq1} and the fact that $\mathcal L^2(\mathcal F_+\cup \mathcal F_-) = 0$, we have the corollary below.

\begin{corollary} \label{acZero} 
For $g\perp R^t(K, \mu)$,
$\mathcal C(g\mu)(\lambda) = 0, ~ \mathcal L^2|_{\mathcal F}-a.a.$.
\end{corollary}
\smallskip

The following corollary, which proves that $\mathcal F$ and $\mathcal R$ are independent of choices of $\{g_n\}$ up to a set of zero analytic capacity, is a direct outcome of Theorem \ref{FCForR}. 
\smallskip

\begin{corollary}\label{NRBUnique} 
Let $\{f_n\}_{n=1}^\infty \subset R^t(K,\mu)^\perp$ be a dense subset. 
Let $\mathcal F'_0$, $\mathcal F'_+$, $\mathcal F'_-$, and $\mathcal F'$ be defined as in Definition \ref{FRDefinition1} with $\nu_j = f_j\mu$. 
Then
 \[
 \ \mathcal F_0 \approx \mathcal F'_0, ~ \mathcal F_+ \approx \mathcal F'_+, ~ \mathcal F_- \approx \mathcal F'_-, ~ \gamma-a.a..
 \]
Hence, $\mathcal F$ and $\mathcal R$ are well defined up to a zero analytic capacity set.
\end{corollary}

\bigskip

\section{Analytic capacity for subsets of non-removable boundary}
\bigskip

The aim of this section is to prove the following theorem. 
\smallskip

\begin{theorem}\label{acTheorem}
Let $S$ be an open square and let $E\subset S\cap \mathcal F$ be a compact subset with $\gamma (E) > 0$. Then there exists an absolute constant $c_5 > 0$ and a sequence of compact subsets $\{E_N\}_{N=1}^\infty$ with $E_N \subset S \cap \mathcal E(\nu_j, 1\le j \le N)$ and $\sup_{x\in E_N}\text{dist}(x, E) < \epsilon$ for a given $\epsilon > 0$ ($c_5$ is independent of $\epsilon$) such that $\gamma (E_N) \ge c_5 \gamma (E)$.  
\end{theorem}

We start with an outline of processing our proof of Theorem \ref{acTheorem}. Let $E$ be as in Theorem \ref{acTheorem}. 
From Theorem \ref{TTolsa} (1) equation \eqref{GammaEq2}, we find a finite, positive Borel measure $\eta$ on $\mathbb{C}$ with $spt(\eta) \subset E$ and  $\eta\in \Sigma(E)$ such that the Cauchy transform $\mathcal C\eta$ is bounded in $L^2(\eta )$ with norm $1$ and $c_6\gamma (E) \le \|\eta\|$. By Definition \ref{FRDefinition1},
 \[
 \ \eta (\text{spt}(\eta)) \le \eta (\text{spt}(\eta) \cap \mathcal F_0) + \eta (\text{spt}(\eta) \cap \mathcal F_+) + \eta (\text{spt}(\eta) \cap \mathcal F_-).
 \]
Hence, one of the following three cases holds
 \begin{eqnarray}\label{singularPart}
\ \eta (\text{spt}(\eta) \cap \mathcal F_0) \ge \dfrac{1}{3}\eta (spt(\eta)) \ge \dfrac{c_6}{3}\gamma (E),
 \end{eqnarray}
 \begin{eqnarray}\label{absolutePart}
 \ \eta (\text{spt}(\eta) \cap \mathcal F_+) \ge \dfrac{1}{3}\eta (spt(\eta)) \ge \dfrac{c_6}{3}\gamma (E),
 \end{eqnarray}
or
 \begin{eqnarray}\label{absolutePart2}
 \ \eta (\text{spt}(\eta) \cap \mathcal F_-) \ge \dfrac{1}{3}\eta (spt(\eta)) \ge \dfrac{c_6}{3}\gamma (E).
 \end{eqnarray}
\smallskip

\begin{proof} (Theorem \ref{acTheorem} if \eqref{singularPart} holds): 
Let $E_N \subset \text{spt}(\eta) \cap \mathcal F_0$ be a compact subset with 
 \[
 \ \eta (E_N) \ge \dfrac{1}{2}\eta (\text{spt}(\eta) \cap \mathcal F_0) \ge \dfrac{c_6}{6}\gamma (E).
 \]  
Certainly,  $ E_N \subset S \cap \mathcal E(\nu_j, 1\le j \le N)$ and $\sup_{x\in E_N}\text{dist}(x, E) = 0$.
$\eta |_{E_N}\in \Sigma(E_N)$ and the Cauchy transform $\mathcal C\eta |_{E_N}$ is bounded in $L^2(\eta |_{E_N})$ with norm $\le 1$. Now $\gamma (E_N) \ge c_5 \gamma (E)$ follows from Theorem \ref{TTolsa} (1) \eqref{GammaEq2} (applied to $\eta |_{E_N}$).
\end{proof}
\smallskip  

It remains to prove the case in 
which \eqref{absolutePart} is true (same for \eqref{absolutePart2}).

Now set $F = \text{spt}(\eta) \cap \mathcal F_+$. Then
$F \subset \cup_{n=1}^\infty \Gamma _n$
where $\{\Gamma _n\}$ are as in Lemma \ref{GammaExist}. Let $\beta_n$ be the angle of the rotation for $\Gamma _n$. Then
\[
\ \bigcup_{n=1}^\infty \Gamma _n = \bigcup_{k=1}^{36}\bigcup_{(k-1)\frac{\pi}{18} \le \beta_n < k\frac{\pi}{18}} \Gamma_n.
\] 
Hence,
 \[
\ \eta(F) \le \sum_{k=1}^{36} \eta \left (\bigcup_{(k-1)\frac{\pi}{18} \le \beta_n < k\frac{\pi}{18}} \Gamma_n \cap F \right).
\]
We can find a $k$, without loss of generality, assuming $k = 0$, such that 
 \[
 \ G:= \bigcup_{0 \le \beta_n < \frac{\pi}{18}} \Gamma_n \cap F
 \]
 satisfying 
 \[
 \ \eta(G) \ge \dfrac{1}{36} \eta(F) \ge \dfrac{c_6}{108}\gamma(E).
 \]
Further more,
\[
 \ \eta(G) = \lim _{m\rightarrow\infty} \eta\left(\bigcup_{0 \le \beta_n < \frac{\pi}{18}, n\le m} \Gamma_n \cap F \right).
 \]
Therefore, we may assume the number of $\Gamma_n$ is  $M$. Without loss of generality, we assume $\Gamma_1,\Gamma_2,...,\Gamma_M$, $H:=\cup_{n =1}^M \Gamma_n\cap F$, satisfying the following for $1\le n\le M$:

(A) Corresponding Lipschitz function $A_n$ of $\Gamma_n$ satisfying $\|A_n'\|_\infty \le \frac{1}{4}$;

(B) The rotation angles $\beta_n$ of $\Gamma_n$ are between $0$ and $\frac{\pi}{18}$; and

(C) $\eta(H) \ge \dfrac{c_6}{144}\gamma(E)$.

We need a couple of lemmas to finish the proof.
\smallskip  
 
\begin{lemma}\label{lemmaBasic3}
Let $A$ be a Lipschitz function with $\|A'\|_\infty \le \frac{1}{4}$. Let $\Gamma$ be its graph. Suppose that $B\subset\Gamma$ is a compact subset with $\mathcal H^1 |_{\Gamma}(B) > 0$. Then there exists a constant $\delta_0 > 0$ and $\{B(\lambda _n, \delta)\}_{n=1}^N$ for $\delta < \delta_0$ such that $B\subset \cup_{n=1}^N B(\lambda _n, \delta)$ and
 \[
 \ N\le \dfrac{\sqrt{5}\mathcal H^1 |_{\Gamma}(B)}{\delta}.
 \]   
\end{lemma}

\begin{proof}
There exists a family of sets $A_i$ with $diam(A_i) \le 1$, $B\subset\cup_i A_i$, and $\sum_i diam(A_i) \le \frac{3}{2} \mathcal H^1 |_{\Gamma}(B)$. Replacing $A_i$ by slightly bigger ball $B(\lambda_i, \delta_i)$, we get $B\subset\cup_i B(\lambda_i, \delta_i)$, and $\sum_i 2\delta_i \le 2 \mathcal H^1 |_{\Gamma}(B)$. We may assume that the number of balls is $M_0 < \infty$ since $B$ is compact. Let $\delta_0 = \min_{1\le i\le M_0}\delta_i$. Because 
\[
 \ \mathcal H^1 |_{\Gamma}(B(\lambda_i, \delta_i))\le 2\delta_i \sqrt{1+ \|A'\|_\infty^2} \le \sqrt{5}\delta_i,
 \] 
we can replace $B(\lambda_i, \delta_i)$ by $M_i$ equal radius $\delta$ balls, where $ M_i \le \frac{\sqrt{5}\delta_i}{2\delta} + 1 \le \frac{\sqrt{5}\delta_i}{\delta}$. Hence, there exist equal radius balls $\{B(\lambda _n, \delta)\}_{n=1}^N$ whose union covers $B$ and
 \[
 \ N = \sum_{i=1}^{M_0} M_i \le \dfrac{\sqrt{5}\mathcal H^1 |_{\Gamma}(B)}{\delta}.
 \]
The lemma is proved. 
\end{proof}
\smallskip

Recall $UC(\lambda, \alpha, \delta)$ and $LC(\lambda, \alpha, \delta)$ are vertical. We use $UC(\lambda, \alpha, \delta, \beta)$ and $LC(\lambda, \alpha, \delta, \beta)$ to define the rotated regions (at $\lambda$) of $UC(\lambda, \alpha, \delta)$ and $LC(\lambda, \alpha, \delta)$ by the angle $\beta$, respectively.
\smallskip

\begin{lemma}\label{lemmaBasic4}
Let $M$, $\{\Gamma_l\}_{l=1}^M$, $\{A_l\}_{l=1}^M$, $H$, $\beta_n$, and $\eta$ be as in (A)-(C). Let $\sqrt{3} \le \alpha \le 4$ ($\le \frac{1}{\|A_n'\|_\infty}$). 
Then for $\epsilon_1 > 0 $, there exists $\epsilon_2 < \epsilon_1$ and $H_0 \subset H$ ($H_0$ depends on $\epsilon_2$) with $\eta (H_0) \ge \frac{3}{4} \eta (H)$ such that ($i = \sqrt{-1}$) 
 \begin{eqnarray}\label{lemmaBasic4Result}
 \ H_0 + \epsilon_2 i \subset \mathcal E(\nu_j, 1\le j \le N).
 \end{eqnarray}
\end{lemma}

\begin{proof}
By construction ($H\subset \mathcal F_+$),  for $\lambda \in H\cap \Gamma_l$ ($1 \le l \le M$) and $1 \le j \le N$,
\[
 \ \lim_{\delta\rightarrow 0}\dfrac{\gamma(UC(\lambda, \alpha, \delta, \beta_l)\cap \{|\mathcal{C}(\nu _j)(z )| > \epsilon \})}{\delta}=0 
 \]
for $\epsilon > 0$.
Clearly, $UC(\lambda, 1, \delta) \subset UC(\lambda, \alpha, \delta, \beta_l)$ as  $0\le \beta_l < \frac{\pi}{18}$.
Using Theorem \ref{TTolsa} (2), we get
 \begin{eqnarray}\label{lemmaBasic4Eq0}
 \ \begin{aligned}
 \ & \lim_{\delta\rightarrow 0}\dfrac{\gamma(UC(\lambda, 1, \delta)\cap\mathcal E (\nu_j, 1\le j \le N)^c) }{\delta} \\
 \ \le &A_T\sum_{j=1}^N \lim_{\delta\rightarrow 0} \dfrac{\gamma(UC(\lambda, \alpha, \delta, \beta_l)\cap\{|\mathcal{C}(\nu _j)(z )| > \frac{1}{N}\})}{\delta} \\
 \ = & 0
\ \end{aligned} 
\end{eqnarray}
for $\lambda \in H\cap \Gamma_l$ and $1\le l \le M$.

Let $B_n$ be a subset consisting of $\lambda\in H$ satisfying
 \begin{eqnarray}\label{lemmaBasic4Eq1}
 \ \gamma(UC(\lambda, 1, \delta)\cap\mathcal E (\nu_j, 1\le j \le N)^c) \le \dfrac{\eta (H) \delta}{16\sqrt{10}A_T^2\sum_{l=1}^M\mathcal H^1(H\cap \Gamma_l)}
 \end{eqnarray}
for $\delta \le \frac{1}{n}$, where $A_T$ is as in Theorem \ref{TTolsa}. 

The sets $B_n \subset B_{n+1}$ and
 by \eqref{lemmaBasic4Eq0}, we see that   
 \[
 \ \eta \left (H\setminus \bigcup_{n=1}^\infty B_n \right ) = 0.
 \]
Choose $n$ large enough such that there is a compact subset $F_n\subset B_n$ with $\eta (F_n) \ge \frac{7}{8}\eta(H)$. 
From Lemma \ref{lemmaBasic3}, there exists $\delta_l > 0$ so that $H\cap\Gamma_l$ can be covered by $N(l,\delta)$ equal radius $\delta$ balls with
 \[
 \ N(l,\delta) \le \dfrac{\sqrt{5} \mathcal H^1 (H\cap\Gamma_l)}{\delta}
 \]
for $\delta < \delta_l$. Let $\delta_0 = \frac{1}{2}\min (\min_{1\le l \le M} \delta_l, \frac{1}{n})$ and $N(\delta) = \sum_{l=1}^M N(l,\delta)$. Then $H$ can be covered by $\{B(\lambda_k, \delta)\}_{k=1}^{N(\delta)}$ for $\delta \le \delta_0$. Hence,
$F_n \subset \bigcup_{k=1}^{N(\delta )} B(\lambda_k, \delta)$. We assume that $\lambda_k\in F_n$, otherwise, let $u_k \in B(\lambda_k, \delta) \cap F_n$ and we replace $B(\lambda_k, \delta)$ by $B(u_k, 2\delta)$.  
Let $\epsilon_2 < \min(\epsilon_1, \sqrt{2}\delta_0)$. Set
$UL_n = \left (F_n + \epsilon_2 i\right )$.
Therefore,
 \[
 \ UL_n \subset \bigcup_{k=1}^{N(\frac{\epsilon_2}{\sqrt{2}})} UC(\lambda_{k}, 1, 2\epsilon_2).
 \]
This implies 
 \[
 \ UL_n \cap \mathcal E (\nu_j, 1\le j \le N)^c \subset \bigcup_{k=1}^{N(\frac{\epsilon_2}{\sqrt{2}})} UC(\lambda_{k}, 1, 2\epsilon_2) \cap \mathcal E (\nu_j, 1\le j \le N)^c.
 \]
Let 
 \[
 \ H_0 = UL_n \cap \mathcal E (\nu_j, 1\le j \le N)  - \epsilon_2 i.
 \]
If $H_1 = UL_n \cap \mathcal E (\nu_j, 1\le j \le N)^c  - \epsilon_2 i$, then
 \[
\ \begin{aligned}
 \  \eta ( H_0 ) = &  \eta \left ( (UL_n  - \epsilon_2 i) \setminus H_1 \right ) \\
 \ \ge & \eta ( F_n) - \eta \left ( H_1 \right ).
 \ \end{aligned}
\]
$\eta |_{H_1}\in \Sigma(H_1)$ and the Cauchy transform $\mathcal C\eta |_{H_1}$ is bounded in $L^2(\eta |_{H_1})$ with norm $\le 1$.
From Theorem \ref{TTolsa} (1) \eqref{GammaEq2} and (2), we get
\[
 \ \begin{aligned}
 \  \eta \left (  H_1 \right ) \le & A_T \gamma \left (  H_1 \right )  \\
 \ = & A_T\gamma \left (  UL_n \cap \mathcal E (\nu_j, 1\le j \le N)^c \right )  \\
 \ \le & A_T^2\sum_{k=1}^{N(\frac{\epsilon_2}{\sqrt{2}})} \gamma(UC(\lambda_{k}, 1, 2\epsilon_2) \cap \mathcal E (\nu_j, 1\le j \le N)^c) \\
 \ \le & \dfrac{1}{8} \eta (H),
 \ \end{aligned}
\]
where \eqref{lemmaBasic4Eq1} is used in last step.
Combining above two inequalities, we get $\eta (H_0) \ge \frac{3}{4} \eta (H)$ and \eqref{lemmaBasic4Result} holds. This completes the proof of the lemma.
\end{proof}
\smallskip

Now the remaining proof of Theorem \ref{acTheorem} is a direct conclusion of Lemma \ref{lemmaBasic4}. 
\smallskip

\begin{proof} (Remaining of Theorem \ref{acTheorem}) Let $E_N = H_0 + i\epsilon _2$ for some $\epsilon_2 < \epsilon_1$, where $H_0$ and $\epsilon_2$ are as in Lemma \ref{lemmaBasic4}. From Lemma \ref{lemmaBasic4}, we choose $\epsilon_1 < \epsilon$ small enough so that $E_N \subset S\cap\mathcal E(\nu_j, 1\le j \le N)$, $\eta(E_N) \ge \frac{3}{4}\eta(H) \ge \frac{c_6}{192}\gamma(E)$, and $\sup_{x\in E_N}\text{dist}(x, E) < \epsilon$. $\eta |_{E_N}\in \Sigma(E_N)$ and the Cauchy transform $\mathcal C\eta |_{E_N}$ is bounded in $L^2(\eta |_{E_N})$ with norm $\le 1$. Applying Theorem \ref{TTolsa} (1) \eqref{GammaEq2} to $\eta |_{E_N}$, we prove  $\gamma (E_N) \ge c_5 \gamma (E)$. 
\end{proof}
\smallskip

Applying Theorem \ref{acTheorem}, we prove the following lemma which is critical for our approximation scheme in later sections.

\begin{lemma} \label{hFunction} 
Let $S$ be a square and let $E$ be a compact subset of $S\cap\mathcal F$ with $\gamma(E) > 0$. Then 
there exists $f\in R^t(K,\mu)\cap L^\infty(\mu)$ and a finite positive measure $\eta$ with $\text{spt}(\eta)\subset E$ and $\|\mathcal C_\epsilon (\eta) \| \le C_6$ such that $f(z) = \mathcal C(\eta)(z)$ for $z\in E^c$, 
 \[
 \ \|f\|_{L^\infty(\mu)} \le C_6,~ f(\infty) = 0,~ f'(\infty) = \gamma(E),
 \]
and
 \[
 \ \mathcal C(\eta)(z) \mathcal C(g_j\mu) (z) = \mathcal C(fg_j\mu) (z), ~\gamma|_{E^c}-a.a.
 \]
for $j \ge 1$. 
\end{lemma}

\begin{proof}
From Theorem \ref{acTheorem}, there exists $E_N \subset S\cap \mathcal E (g_j\mu, 1\le j \le N)$ such that $\gamma (E_N) \ge c_5 \gamma(E)$ and $\sup_{x\in E_N}\text{dist}(x, E) \rightarrow 0$. From Lemma \ref{lemmaBasic6}, there exists a positive measure $\eta_N$ supported in $E_N$ satisfying (1) - (3) of Lemma \ref{lemmaBasic6} and $c_4\gamma(E_N) \le \|\eta_N\| \le C_4\gamma(E_N)$.
Using Lemma \ref{lemmaBasic0} (1), we can find a function $w_N(z)$ satisfying $0\le w_N(z) \le 1$, 
 \[
 \ \|\eta_N\| \le 2 \|w_N\eta_N\|,~\text{and }\|\mathcal C_\epsilon (w_N\eta_N)\|_{L^\infty (\mathbb C)} \le C_7 
 \]
for all $\epsilon > 0$. 
From Lemma \ref{lemmaBasic6} (3), there exists $f_N\in L^\infty(\mu )$ with $f_N(\infty) = 0$, $f'_N(\infty) = \|w_N\eta_N\|$, $\|f_N\|_{L^\infty(\mu )} \le C_7$, and $f_N(z) = \mathcal C(w_N\eta_N)(z)$ for $z\in E_N^c$ such that    
 \[
 \ \left | \int f_N g_jd\mu \right |  = \left | \int \mathcal C(g_j\mu)(z) w_Nd\eta_N \right | \le \dfrac{\|\eta_N\|}{N} \le \dfrac{C_4\gamma(E)}{N}  
 \]
for $j=1,2,...,N$.
Now by passing to a subsequence, we may assume that $w_N\eta_N$ converges to $\eta_0$ in $C(\overline{S})^*$ weak$^*$ topology and $f_N$ converges to $f_0$ in $L^\infty(\mu)$ weak$^*$ topology. Clearly, $\text{spt}(\eta_0)\subset E$ since $\sup_{x\in E_N}\text{dist}(x, E)\rightarrow 0$ and $f_0(\lambda ) = \mathcal C(\eta_0)(\lambda )$ for $\lambda \in E^c$ with $f_0(\infty) = 0$, $f'_0(\infty) = \|\eta_0\|\ge \frac{1}{2}c_4\gamma (E)$, $\|f_0\|_{L^\infty(\mu)} \le C_7$, and 
 \[
 \ \int f_0 g_jd\mu = 0, ~ j=1,2,...,
 \]
which implies $f_0\in R^t(K, \mu)\cap L^\infty(\mu )$. For $\lambda \notin E$, using Lemma \ref{lemmaBasic6} (3), we conclude that $\dfrac{f_0(z) - f_0(\lambda)}{z-\lambda}\in R^t(K, \mu)\cap L^\infty(\mu )$. Set 
 \[
 \ f = \dfrac{f_0}{f'_0(\infty)} \gamma (E) \text{ and } \eta  = \dfrac{\eta_0}{f'_0(\infty)} \gamma (E). 
 \]
It is easy to verify that $f$ and $\eta$ satisfy the properties of the lemma.   
\end{proof} 
\bigskip

\section{Bounded point evaluations and non-removable boundary}
\bigskip

Let $\partial_o K$ denote the union of boundaries of $\mathbb C\setminus K$ connected components (outer boundary of $K$). Define
 \begin{eqnarray}\label{BOne}
 \ \partial_1 K = \left \{\lambda \in K:~ \underset{\delta\rightarrow 0}{\overline\lim} \dfrac{\gamma(B(\lambda,\delta)\setminus K)}{\delta} > 0 \right \}.
 \end{eqnarray}
Then
$\partial_o K\subset \partial_1 K \subset \partial K$.
\smallskip 

\begin{proposition} \label{NFSetIsBig} 
The following statements are true.

(1) $\text{spt}(\mu) \subset \overline{\mathcal R}$;

(2) $\overline{\text{int}(K) \setminus \mathcal F} = \overline{\text{int}(K)}$;

(3) $\partial_1 K \subset \mathcal F, ~\gamma-a.a.$.
\end{proposition}

\begin{proof}
(1) This is a direct consequence of Corollary \ref{acZero}. In fact, if $B(\lambda_0,\delta) \subset \mathcal F$, from Corollary \ref{acZero}, we see that $\mathcal C(g_j\mu)(z) = 0$ with respect to $\mathcal L^2|_{B(\lambda_0,\delta)}-a.a.$ This implies that $\mu(B(\lambda_0,\delta)) = 0$ by \eqref{CTDistributionEq}.

(2) If $B(\lambda_0, \delta) \subset \text{int}(K) \setminus  \overline{\text{int}(K) \setminus \mathcal F}$, then $B(\lambda_0,\delta) \subset \mathcal F$. Same as in  (1), we derive that $\mu(B(\lambda_0,\delta)) = 0$. Since $B(\lambda_0, \delta) \subset  \sigma(S_\mu) (=K)$, there exists $j_1$ such that the analytic function $\mathcal C(g_{j_1}\mu)$ on $B(\lambda_0, \delta)$ is not identically zero, which implies that $B(\lambda_0, \delta)\subset \mathcal R, ~\gamma-a.a.$. This is a contradiction. 

(3) follows from Theorem \ref{FCharacterization} and the fact that for $\lambda \in \partial_1 K$, 
 \[
 \ B(\lambda, \delta) \setminus K \subset B(\lambda, \delta) \cap\mathcal E(g_j\mu, 1\le j \le N).
 \]
\end{proof}
\smallskip

\begin{lemma} \label{GLemma}
Let $O$ be an open subset of $\mathbb C$, $g\perp R^t(K, \mu)$, and $a\ne 0$. If for $0 < \epsilon < \frac{|a|}{2}$ and $\lambda_0\in \mathbb C$,
 \begin{eqnarray}\label{GLemmaAssump}
 \ \lim_{\delta\rightarrow 0 }\dfrac{\gamma(B(\lambda_0, \delta)\cap O \cap \{|\mathcal C(g\mu)(z) - a| > \epsilon\})}{\delta} = 0,
 \end{eqnarray} 
then
 \[
 \ \lim_{\delta\rightarrow 0 }\dfrac{\gamma(B(\lambda_0, \delta)\cap O \cap \mathcal F)}{\delta} = 0.
 \] 
\end{lemma}

\begin{proof}
Suppose that there exists $\epsilon_0 > 0$  and $\delta_n \rightarrow 0$ such that 
	\[
 \ \gamma(B(\lambda_0, \delta_n)\cap O\cap \mathcal F)\ge 2\epsilon_0 \delta_n.
 \]
Let $F_n \subset B(\lambda_0, \delta_n)\cap O\cap \mathcal F$ be a compact subset such that $\gamma(F_n)\ge \epsilon_0 \delta_n$.
From Corollary \ref{NRBUnique}, we may assume $g_1 = g$. 
Let $d_n$ be the smallest distance between compact sets $(B(\lambda_0, \delta_n)\cap O)^c$ and $F_n$. From
Theorem \ref{acTheorem}, we find a compact subset $E_N^n \subset \mathcal E(g_j \mu, 1\le j \le N)$ such that $\gamma(E_N^n) \ge c_5\gamma(F_n)$ and 
 $\sup_{x\in E_N^n}\text{dist} (x, F_n) < d_n$.
Hence, $E_N^n \subset B(\lambda_0, \delta_n)\cap O \cap \mathcal E(g_j \mu, 1\le j \le N)$. 
If $\frac{1}{N} < \frac{|a|}{2}$, then 
 \[
 \ \mathcal E(g_j \mu, 1\le j \le N) \subset \{|\mathcal C(g\mu)(z) - a| > \epsilon\}.
 \]
Hence,
 \[
 \ \dfrac{\gamma(B(\lambda_0, \delta_n)\cap O\cap \{|\mathcal C(g\mu)(z) - a| > \epsilon\})}{\delta_n} \ge \dfrac{\gamma(B(\lambda_0, \delta_n)\cap O\cap E_N^n)}{\delta_n} \ge \epsilon _0 c_5,
 \]
which contradicts to the assumption \eqref{GLemmaAssump}. The lemma is proved.  
\end{proof}
\smallskip

The Theorem below is an important property that is needed in later sections.

\begin{theorem}\label{DensityCorollary}
For almost all $\lambda_0 \in \mathcal R$ with respect to $\gamma$, 
\begin{eqnarray}\label{DensityCorollaryEq1}
 \ \lim_{\delta\rightarrow 0 }\dfrac{\gamma(B(\lambda_0, \delta)\cap \mathcal F)}{\delta} = \lim_{\delta\rightarrow 0 }\dfrac{\gamma(B(\lambda_0, \delta)\setminus \mathcal R)}{\delta} = 0.
 \end{eqnarray} 
\end{theorem}

\begin{proof}
For almost all $\lambda_0\in \mathcal R_0$ with respect to $\gamma$, there exists $j_0$ such that $\mathcal C(g_{j_0}\mu)(\lambda_0) \ne 0$ exists and $\lambda_0 \in \mathcal {ZD}(g_{j_0}\mu)$. \eqref{DensityCorollaryEq1} follows from Lemma \ref{CauchyTLemma} and Lemma \ref{GLemma}.

For almost all $\lambda_0\in \mathcal R_1\cap \Gamma_1$ ($\Gamma_n$ as in Lemma \ref{GammaExist}) with respect to $\gamma$, there are integers $j_1$ and $j_2$ such that $v^+(g_{j_1}\mu, \Gamma_1, \lambda_0) \ne 0$ and $v^-(g_{j_2}\mu, \Gamma_1, \lambda_0) \ne 0$. Applying Theorem \ref{GPTheorem1} and Lemma \ref{GLemma} for $O= U_{\Gamma_1}$, we have 
 \begin{eqnarray}\label{DensityCorollaryEq2}
 \ \lim_{\delta\rightarrow 0 }\dfrac{\gamma(B(\lambda_0, \delta)\cap U_{\Gamma_1} \cap \mathcal F)}{\delta} = 0.
\end{eqnarray}
 Similarly, 
 \begin{eqnarray}\label{DensityCorollaryEq3}
 \ \lim_{\delta\rightarrow 0 }\dfrac{\gamma(B(\lambda_0, \delta)\cap L_{\Gamma_1} \cap \mathcal F)}{\delta} = 0.
 \end{eqnarray}
We claim that 
 \begin{eqnarray}\label{DensityCorollaryEq4} 
 \ \lim_{\delta\rightarrow 0 }\dfrac{\gamma(B(\lambda_0, \delta)\cap \Gamma_1 \cap \mathcal F)}{\delta} = 0.
 \end{eqnarray}
In fact, if it is not, we can assume that $\lambda_0$ is a Lebesgue point for $\mathcal H^1 |_{\Gamma_1\cap \mathcal N(h)}$. Hence, we must have 
 \[
 \ \underset{\delta\rightarrow 0 }{\overline \lim}\dfrac{\gamma(B(\lambda_0, \delta)\cap \Gamma_1 \cap \mathcal N(h)\cap \mathcal F)}{\delta} > 0.
 \]
By Definition \ref{FRDefinition1}, $B(\lambda_0, \delta)\cap \Gamma_1 \cap \mathcal N(h)\cap \mathcal F \subset \mathcal F_+ \cup \mathcal F_-$. Using Lemma \ref{lemmaBasic4} for $\Gamma_1$, we conclude that either
 \[ 
 \ \underset{\delta\rightarrow 0 }{\overline \lim}\dfrac{\gamma(B(\lambda_0, \delta)\cap U_{\Gamma_1} \cap \mathcal E(g_j \mu, 1\le j \le N))}{\delta}> 0
 \]
or
 \[ 
 \ \underset{\delta\rightarrow 0 }{\overline \lim}\dfrac{\gamma(B(\lambda_0, \delta)\cap L_{\Gamma_1} \cap \mathcal E(g_j \mu, 1\le j \le N))}{\delta}> 0.
 \]
However, for $0 < \epsilon < \frac 12 \min (|v^+(g_{j_1}\mu, \Gamma_1, \lambda_0)|, |v^-(g_{j_2}\mu, \Gamma_1, \lambda_0)|)$ and 
\newline
$N > \max(j_1,j_2, \frac{2}{\epsilon}+1)$ large enough,
 \[
 \ \mathcal E(g_j \mu, 1\le j \le N) \subset \{|\mathcal C(g_{j_1}\mu)(z) - v^+(g_{j_1}\mu, \Gamma_1, \lambda_0)| > \epsilon\}
 \]
and
 \[
 \ \mathcal E(g_j \mu, 1\le j \le N) \subset \{|\mathcal C(g_{j_2}\mu)(z) - v^-(g_{j_2}\mu, \Gamma_1, \lambda_0)| > \epsilon\}.
 \]
 This contradicts to Theorem \ref{GPTheorem1} (b) or (c). Applying Theorem \ref{TTolsa} (2) and combining \eqref{DensityCorollaryEq2}, \eqref{DensityCorollaryEq3}, and \eqref{DensityCorollaryEq4}, we prove \eqref{DensityCorollaryEq1}.   
\end{proof}
\smallskip

\begin{lemma}\label{BPELemma}
If $\lambda_0 \in \text{bpe}(R^t(K,\mu))$, then there exists a function $k\in L^{s}(\mu)$ with $k(\lambda_0) = 0$ so that $r(\lambda_0) = \int r(z) \bar k(z) d \mu(z)$ for all $r\in Rat(K)$.
\end{lemma}

\begin{proof}
Suppose that $\lambda_0\in \text{bpe}(R^t(K,\mu))$. If we assume that $\mu(\{\lambda_0\}) > 0$, then $\lambda_0\in bpe(R^t(K,\mu |_{\{\lambda_0\}^c}))$. In fact, if $\lambda_0\notin bpe(R^t(K,\mu |_{\{\lambda_0\}^c}))$, then there exist rational functions $r_n \in Rat(K)$ with $r_n(\lambda_0) = 1$ and $\|r_n\| _{L^t(\mu |_{\{\lambda_0\}^c})} \rightarrow 0$. This implies that
$\|r_n - \chi_{\{\lambda_0\}}\| _{L^t(\mu)} \rightarrow 0$.    
Hence, $\chi_{\{\lambda_0\}}\in R^t(K,\mu)$. This contradicts to the hypothesis that $S_\mu$ is pure.
\end{proof}
\smallskip

The following result is an interesting property for $\text{bpe}(R^t(K,\mu))$.

\begin{corollary}\label{BPETheorem} If $\lambda_0 \in \text{bpe}(R^t(K,\mu))$, then 
 \begin{eqnarray}\label{BPETheoremEq1}
\ \lim_{\delta\rightarrow 0} \dfrac{\gamma(B(\lambda_0,\delta)\cap \mathcal F)}{\delta} = 0.
 \end{eqnarray}
\end{corollary}

\begin{proof}
Let $k\in L^{s}(\mu)$ be as in Lemma \ref{BPELemma}.
 If $g(z) = (z-\lambda_0) \bar k(z)$, then $g\perp R^t(K,\mu)$, $\nu = g\mu$ satisfies the assumptions of Lemma \ref{CauchyTLemma}, and $\mathcal C(g\mu)(\lambda_0) = 1$. Thus, by Lemma \ref{CauchyTLemma}, we get
\[
\ \lim_{\delta\rightarrow 0} \dfrac{\gamma(B(\lambda_0,\delta)\cap \{|\mathcal C(g\mu) - 1| > \epsilon\}))}{\delta} = 0
 \] 
where $\epsilon< \frac{1}{2}$. Now \eqref{BPETheoremEq1} follows from Lemma \ref{GLemma}. 
\end{proof}
\smallskip

The following Lemma is from Lemma B in \cite{ARS09}.
\smallskip

\begin{lemma} \label{lemmaARS}
There are absolute constants $\epsilon _1 > 0$ and $C_9 < \infty$ with the
following property. If $\delta > 0$ and $E \subset  B(0, \delta)$ with 
$\gamma(E) < \delta\epsilon_1$, then
\[
\ |p(\lambda)| \le \dfrac{C_9}{\pi \delta^2} \int _{B(0, \delta)\setminus E} |p|\, d\mathcal L^2
\]
for all $\lambda$ in $B(0, \frac{\delta}{2})$ and all analytic polynomials $p$.
\end{lemma}
\smallskip

The theorem below provides  an important relation between $\text{abpe}(R^t(K,\mu))$ and $\mathcal R$. 

\begin{theorem}\label{ABPETheorem} The following property holds:
 \begin{eqnarray}\label{ABPETheoremEq1}
 \ \text{abpe}(R^t(K,\mu)) \approx \text{int}(K) \cap \mathcal R,~ \gamma-a.a..
 \end{eqnarray}
More precisely, the following statements are true:

(1) If $\lambda_0 \in \text{int}(K)$ and there exists $N\ge 1$ such that 
 \begin{eqnarray}\label{ABPETheoremEq2}
 \ \lim_{\delta \rightarrow 0} \dfrac{ \gamma (\mathcal E(g_j\mu, 1 \le j \le N) \cap B(\lambda_0, \delta))}{\delta} = 0,
 \end{eqnarray}
then $\lambda_0\in \text{abpe}(R^t(K,\mu))$.

(2)
\[
 \  \text{abpe}(R^t(K,\mu)) \subset \text{int}(K) \cap \mathcal R,~ \gamma-a.a..
 \]
\end{theorem}

\begin{proof}
(1): If $\lambda_0 \in \text{int}(K)$  satisfies \eqref{ABPETheoremEq2}, then we  choose $\delta > 0$ small enough such that $B(\lambda_0, \delta) \subset int(K)$ and 
$\gamma (E:= (\mathcal E(g_j\mu, 1 \le j \le N) \cap B(\lambda_0, \delta)))\le \epsilon_1 \delta$, where $\epsilon_1$ is from Lemma \ref{lemmaARS}. Hence, using Lemma \ref{lemmaARS}, we conclude
 \[
 \ \begin{aligned}
 \ |r(\lambda )| \le & \dfrac{C_9}{\pi \delta^2} \int _{B(\lambda_0, \delta) \setminus E} |r(z)| d\mathcal L^2(z) \\
 \ \le & \dfrac{NC_9}{\pi \delta^2} \int _{B(\lambda_0, \delta)} |r(z)| \max_{1\le j \le N} |\mathcal C(g_j\mu)(z)| d\mathcal L^2(z) \\
\ \le & \dfrac{NC_9}{\pi \delta^2} \int _{B(\lambda_0, \delta)} \max_{1\le j \le N} |\mathcal C(rg_j\mu)(z)| d\mathcal L^2(z) \\
\ \le & \dfrac{NC_9}{\pi \delta^2} \sum_{j = 1}^N\int _{B(\lambda_0, \delta)}  |\mathcal C(rg_j\mu)(z)| d\mathcal L^2(z) \\
\ \le & \dfrac{NC_9}{\pi \delta^2} \sum_{j = 1}^N\int \int _{B(\lambda_0, \delta)} \left |\dfrac{1}{z-w} \right | d\mathcal L^2(z) |r(w)||g_j (w)| d\mu (w) \\
\ \le & \dfrac{NC_{10}}{\delta} \sum_{j = 1}^N \|g_j\|_{L^{s}(\mu)} \|r\|_{L^{t}(\mu)}
 \ \end{aligned}
 \]
for all $\lambda$ in $B(\lambda_0, \frac{\delta}{2})$ and all $r\in Rat(K)$. This implies that $\lambda_0\in abpe(R^t(K,\mu))$.

(2): Let $E \subset G\cap \mathcal F$ be a compact subset with $\gamma(E) > 0$, where $G$ is a connected component of $abpe(R^t(K,\mu))$. By Lemma \ref{hFunction}, there exists $f\in R^t(K,\mu)\cap L^\infty(\mu)$ that is bounded analytic on $E^c$ such that  
 \[
 \ \|f\|_{L^\infty(\mu)} \le C_6,~ f(\infty) = 0,~ f'(\infty) = \gamma(E),
 \]
and $f(z) \mathcal C(g_j\mu) (z) = \mathcal C(fg_j\mu) (z), ~\gamma|_{E^c}-a.a.$ for $j \ge 1$. Let $r_n\in Rat(K)$ such that $\|r_n-f\|_{L^t(\mu)} \rightarrow 0$. Hence, $r_n$ uniformly tends to an analytic function $f_0$ on compact subsets of $G$ and $\frac{f-f_0(\lambda)}{z-\lambda} \in R^t(K,\mu)$ for $\lambda\in G$. Therefore, 
 \[
 \ f(z) \mathcal C (g_j\mu)(z) = \mathcal C (fg_j\mu)(z) = f_0(z) \mathcal C (g_j\mu)(z),~z\in G\setminus E,~ \gamma-a.a.
 \]
 For $B(\lambda, \delta) \subset G\setminus E$, from Corollary \ref{BPETheorem}, we see that $\gamma(B(\lambda, \delta)\cap \mathcal R) > 0$. Hence $f(z) = f_0(z)$ for $z\in G\setminus E$ since $\mathcal R \subset \mathcal N(g_j\mu, 1\le j < \infty)$.  Thus, the function $f_0(z)$ can be analytically extended to $\mathbb C$ and $f_0(\infty) = 0$. So $f_0$ is a zero function. This contradicts to $f'(\infty) \ne 0$. 

\eqref{ABPETheoremEq1} now follows from Theorem \ref{FCharacterization}.
\end{proof}
\smallskip

\begin{remark}\label{ABPERemark}
For Theorem \ref{ABPETheorem} (1), we do not need to assume that $K = \sigma(S_\mu)$. That is, if $K \subset K_0$, $R^t(K,\mu) = R^t(K_0,\mu)$, and $\lambda_0\in \text{int}(K_0)$ satisfies \eqref{ABPETheoremEq2}, then $\lambda_0\in \text{abpe}(R^t(K_0,\mu))$.
\end{remark}

\bigskip

\section{Further properties for $\mathcal F$}
\bigskip

In this section, we prove the following property of the non-removable boundary $\mathcal F$ that will be used to prove Theorem \ref{IrreducibilityTheorem}.

\begin{theorem}\label{Lemma3} If $S_\mu$ on $R^t(K,\mu)$ is irreducible, then for $\lambda \in K$,
 \[
 \ \underset{\delta \rightarrow 0}{\overline{\lim}}\dfrac{\gamma(B(\lambda, \delta)\setminus \mathcal F)}{\delta} = \underset{\delta \rightarrow 0}{\overline{\lim}}\dfrac{\gamma(B(\lambda, \delta)\cap \mathcal R)}{\delta} > 0.
 \] 
\end{theorem}
\smallskip

We first prove the following lemma.

\begin{lemma}\label{characterizationF}
Suppose that $S_\mu$ on $R^t(K,\mu)$ is pure and $G$ is a bounded open connected subset. If $f\in R^t(K, \mu)\cap L^\infty(\mu)$ satisfies
 \[
 \ \mathcal C(fg\mu)(z) = \chi _G(z) \mathcal C(g\mu)(z), ~ \gamma |_{(\partial G)^c}-a.a.
 \]
for $g\perp R^t(K, \mu)$,
then there is a Borel set $\Delta$ such that $G \subset \Delta \subset \overline{G}$ and $f= \chi_{\Delta}$.
\end{lemma}

\begin{proof}
For $g\perp R^t(K,\mu)$, since $fg\perp R^t(K,\mu)$, we get
\[
\ \begin{aligned}
 \ \mathcal C(f^2g\mu)(z) = &\chi _G(z) \mathcal C(fg\mu)(z) \\
 \ = &\chi _G(z) \mathcal C(g\mu)(z) \\
 \ = &\mathcal C(fg\mu)(z) , ~ \gamma |_{(\partial G)^c}-a.a..
 \ \end{aligned} 
\]
 If $\gamma (\partial G\cap \mathcal R_0) > 0$, then except a set of zero analytic capacity, for $\lambda_0\in \partial G \cap \mathcal R_0$, $\mathcal C(fg\mu)$ and $\mathcal C(f^2g\mu)$ are $\gamma$-continuous at $\lambda_0$ (see Lemma \ref{CauchyTLemma}). Let $P_\delta$ be a path starting at a point in $G\cap B(\lambda_0, \frac{\delta}{2})$ and ending at a point in $G\setminus B(\lambda_0, \delta)$  such that  $P_\delta\subset G$. Hence, if 
\[
 \ A_1 = \{|\mathcal C(f^2g\mu)(z)-\mathcal C(f^2g\mu)(\lambda_0)|>\epsilon\}
 \]
 and 
 \[
 \ A_2= \{|\mathcal C(fg\mu)(z)-\mathcal C(fg\mu)(\lambda_0)|>\epsilon\},
 \]
 then
 \[
 \ \lim_{\delta\rightarrow 0}\dfrac{\gamma(B(\lambda_0,\delta)\cap A_1)}{\delta} = \lim_{\delta\rightarrow 0}\dfrac{\gamma(B(\lambda_0,\delta)\cap A_2)}{\delta} = 0.
 \]
Thus, using Theorem \ref{TTolsa} (2), we get 
 \[
 \ \begin{aligned}
 \ &\underset{\delta\rightarrow 0|}{\overline{\lim}}\dfrac{\gamma(B(\lambda_0,\delta)\cap A_1^c\cap A_2^c\cap P_\delta)}{\delta} \\
 \ \ge & \dfrac{1}{A_T}\underset{\delta\rightarrow 0|}{\overline{\lim}}\dfrac{\gamma(B(\lambda_0,\delta)\cap P_\delta)}{\delta} - \lim_{\delta\rightarrow 0|}\dfrac{\gamma(B(\lambda_0,\delta)\cap A_1)}{\delta} - \lim_{\delta\rightarrow 0|}\dfrac{\gamma(B(\lambda_0,\delta)\cap A_2)}{\delta}\\
 \ > & 0. 
 \ \end{aligned}
 \]
There exists a sequence of $\lambda_n\in P_{\delta_n}$ such that $\lambda_n\rightarrow \lambda_0$, $\mathcal C(f^2g\mu)(\lambda_n) = \mathcal C(fg\mu)(\lambda_n)$, $\mathcal C(f^2g\mu)(\lambda_n)\rightarrow \mathcal C(f^2g\mu)(\lambda_0)$, and 
 $\mathcal C(fg\mu)(\lambda_n)\rightarrow \mathcal C(fg\mu)(\lambda_0)$. 
On the other hand, $\mathcal L^2(\mathcal R_1) = 0$ and by Corollary \ref{acZero},
\[
 \ \mathcal C(f^2g\mu)(z) = \mathcal C(fg\mu)(z) = 0, ~ \mathcal L^2|_{\mathcal F}-a.a..
 \]
 Therefore, using Lemma \ref{lemmaBasic0} (2) (a $\gamma$ zero set is an $\mathcal L^2$ zero set), we get 
 \[
 \ \mathcal C(f^2g\mu)(z) = \mathcal C(fg\mu)(z), ~ \mathcal L^2-a.a..
 \]
Hence, $f^2 = f$ because $S_\mu$ is pure. The lemma is proved.
\end{proof} 
\smallskip

\begin{remark}\label{characterizationFRemark}
Clearly, Lemma \ref{characterizationF} holds if $G$ is a finite union of disjoint bounded open connected subsets.
\end{remark}
\smallskip

We use Thomson's coloring scheme that is described at the beginning of section 2 of \cite{y19}. 

We fix a Borel set $E$. For a square $S$, whose edges are parallel to x-axis and y-axis, let $c_S$ denote the center and $d_S$ denote the side length. For $a>0,$ $aS$ is a square with the same center of $S$ ($c_{aS} = c_S$) and the side length $d_{aS} = ad_S.$ For a given $\epsilon > 0,$ a closed square $S$ is defined to be light $\epsilon$ if
 \begin{eqnarray}\label{2-2}
 \ \gamma (S \setminus E)) > \epsilon  d_S.
 \end{eqnarray}
A square is called heavy $\epsilon$ if it is not light $\epsilon$.
Let $R = \{ z: -1/2 < Re(z),Im(z) < 1/2 \}$ and $Q = \overline{\mathbb D}\setminus R.$ 

We now sketch our version of
Thomson's coloring scheme for $Q$ with a given $\epsilon$ and a positive integer $m.$ We refer the reader to \cite{thomson} section 2 for details.
 
For each integer $k > 3$ let $\{S_{kj}\}$ be an enumeration of the closed squares contained in $\mathbb C$ with edges of length $2^{-k}$
parallel to the coordinate axes, and corners at the points whose coordinates
are both integral multiples of $2^{-k}$ (except the starting square $S_{m1}$, see (3) below). 
In fact, Thomson's coloring scheme is just needed to be modified slightly as the following:

(1) Use our definition of a light $\epsilon$ square \eqref{2-2}.

(2) A path to $\infty$ means a path to any point that is outside of $Q$ (replacing the polynomially convex hull of $\Phi$ by $Q$).

(3) The starting yellow square $S_{m1}$ in the $m$-th generation is $R.$ Notice that the length of $S_{m1}$ in $m$-th generation is $1$ (not $2^{-m}$).

We will borrow notations that are used in Thomson's coloring scheme such as $\{\gamma_n\}_{n\ge m}$ and $\{\Gamma_n\}_{n\ge m},$ etc.  For the benefit of the reader, let us briefly describe the constructions of $\gamma_n$ and $\Gamma_n$. Let $\Gamma_m$ be the boundary of $R$. Color $R$ yellow. This completes  the scheme for the $m$-th generation of squares. We proceed the scheme recursively. First, color green every light $\epsilon$ square in the $n$-th generation that lies outside $\Gamma_{n-1}$ and has a side lying on $\Gamma_{n-1}$. Second, color green every light $\epsilon$ 
square in the $n$-th generation that can be joined to a green square from the first step by a path of light $\epsilon$ squares in the $n$-th generation that lie outside $\Gamma_{n-1}$. This completes the $n$-th generation of green squares. If there is a green square reaching outside of $Q$, we say that the process ends with the $n$-th generation of squares. Otherwise, let $\gamma_n$ denote the boundary of the polynomially convex hull of the union of $\Gamma_{n-1}$ and the green squares in the $n$-th generation. If a square in the $n$-th generation lies outside $\gamma_n$ but has a side lying on $\gamma_n$,
color the square red. Observe that each such red square must be heavy.
A square $S$ in the $n$-th generation is colored yellow if $S$ is outside $\gamma_n$, $S$ has no side lying on $\gamma_n$, and the distance from $S$ to some red square in the $n$-th generation is less than or equal to $n^22^{-n}$. Let $\Gamma_n$ denote the polynomially convex hull of the union of colored squares of the $n$-th generation.

Termination of the scheme means that the process ends for some $n$. That is, $\{\gamma_n\}_{n\ge m}$ and $\{\Gamma_n\}_{n\ge m}$ are finite. If $\{\gamma_n\}_{n\ge m}$ and $\{\Gamma_n\}_{n\ge m}$ are infinite, we say that there is a sequence of heavy $\epsilon$ barriers around the square $R$. The following is Lemma 2 of \cite{y19}.

\begin{lemma} \label{LightRoute}
For $\epsilon > 0$, there exists $\epsilon_1 > 0$ (depending on $\epsilon$) such that the following property holds: If 
$\gamma (\mathbb D \setminus E) < \epsilon_1$,
then there is a sequence of heavy $\epsilon$ barriers around the square $R$, that is, $\{\gamma_n\}_{n\ge m}$ and $\{\Gamma_n\}_{n\ge m}$ are infinite.
\end{lemma}

Let $\varphi$ be a smooth function with compact support. The localization operator
$T_\varphi$ is defined by
 \[
 \ (T_\varphi f)(\lambda) = \dfrac{1}{\pi}\int \dfrac{f(z) - f(\lambda)}{z - \lambda} \bar\partial \varphi (z) d\mathcal L^2(z),
 \]
where $f\in L^\infty (\mathbb C)$. One can easily prove the following
norm estimation for $T_\varphi:$
 \[
 \ \| T_\varphi f\| \le  4\|f\| diameter(supp \varphi)\|\bar\partial \varphi\|.
 \]

Let $g$ be an analytic function outside the disc $B(a, \delta)$ satisfying the condition
$g(\infty ) = 0.$ We consider the Laurent expansion of $g$ centered at $a$,
 \[
 \ g(z) = \sum_{n=1}^\infty \dfrac{c_n(g,a)}{(z-a)^n}.
 \]
We define 
 \[
 \ \alpha (g) = c_1(g, a), ~\beta (g,a) = c_2(g, a).
 \]
$\alpha (g)$ does not depend on the choice of $a$, while $\beta (g,a)$ depends on $a$. However, if $\alpha (g) = 0$, then $\beta (g,a)$ does not depend on $a$, in this case, we denote $\beta (g) = \beta (g,a)$.
\smallskip

\begin{proof} (Theorem \ref{Lemma3}) Fix $\epsilon < \frac{1}{36A_T}$. If there exists $\lambda \in K$ such that 
 \[
 \ \lim_{\delta \rightarrow 0}\dfrac{\gamma(B(\lambda, \delta)\setminus \mathcal F)}{\delta} = 0,
 \]
then there exists $\delta_0 > 0$ such that
 \begin{eqnarray}\label{lemma3Eq1}
 \ \gamma(B(\lambda, \delta)\setminus \mathcal F) < \epsilon_1 \delta,~ \delta \le \delta_0 
 \end{eqnarray}
where $\epsilon_1$ is as in Lemma \ref{LightRoute},
and $\mu (K \setminus \overline{B(\lambda, \delta_0)}) > 0$. Without loss of generality, we assume that $\lambda = 0$ and $\delta_0 = 1$. By Lemma \ref{LightRoute}, $\{\gamma_n\}_{n\ge m}$ and $\{\Gamma_n\}_{n\ge m}$ are infinite. 

Let $f_n = \chi_{G_n}$, where $G_n$ is the region bounded by $\gamma_n$. Clearly, we have $R \subset G_n$ and $f_n \in H^\infty (\mathbb C \setminus \gamma_n)$. We apply standard Vitushkin's scheme for $f_n$. Let $\{S_{ij}, \varphi_{ij}\}$ be the partition of unity with square size $l = \frac{1}{2^n}$ satisfying $\text{spt}(\varphi_{ij}) \subset 2S_{ij}$, $0 \le \varphi_{ij} \le 1$, and $\|\bar \partial \varphi_{ij} \| \le C_{10}/l$. Moreover, $\{S_{ij}\}$ contains all squares of $n$th generation of the scheme. Then
 \begin{eqnarray}\label{SVEq}
 \ f_n = \sum_{2S_{ij} \cap \gamma_n \ne \emptyset} (f_{ij}:= T_{\varphi_{ij}}f_n) = \sum_{2S_{ij} \cap \gamma_n \ne \emptyset} (f_{ij} - g_{ij} - h_{ij}) + f_l. 
 \end{eqnarray}
where 
 \[
 \ f_l = \sum_{2S_{ij} \cap \gamma_n \ne \emptyset} (g_{ij} + h_{ij}).
 \]
To apply standard Vitushkin's scheme, we need to construct functions $g_{ij}, h_{ij} \in R^t(K,\mu)\cap L^\infty(\mu)$ that are bounded analytic off $3S_{ij}$ and the following properties hold:
 \begin{eqnarray}\label{gProperty}
 \ \|g_{ij}\|_{L^\infty(\mu)} \le C_{11}, ~ g_{ij}(\infty) = 0, ~ \alpha(g_{ij}) = \alpha(f_{ij}),
\end{eqnarray}
\begin{eqnarray}\label{hProperty}
 \ \begin{aligned}
 \ & \|h_{ij}\|_{L^\infty(\mu)} \le C_{12}, ~ h_{ij}(\infty) = \alpha(h_{ij}) = 0, \\
 \ & \beta(h_{ij}, c_{ij}) = \beta(f_{ij} - g_{ij}, c_{ij}),
 \ \end{aligned}
\end{eqnarray}
and
 \begin{eqnarray}\label{ghProperty}
 \ \begin{aligned}
 \ &\mathcal C (g_{ij}g\mu)(z) = g_{ij}(z)\mathcal C (g\mu)(z),~\gamma |_{(3S_{ij})^c}-a.a., \\
 \ &\mathcal C (h_{ij}g\mu)(z) = h_{ij}(z)\mathcal C (g\mu)(z),~\gamma |_{(3S_{ij})^c}-a.a. 
 \ \end{aligned}
\end{eqnarray}
for $g\perp R^t(K,\mu)$. 

Assuming $g_{ij}, h_{ij}$ have been constructed and satisfy \eqref{gProperty}, \eqref{hProperty}, and \eqref{ghProperty}, we conclude that 
 \[
 \ |f_{ij}(z) - g_{ij}(z) - h_{ij}(z) | \le C_{13} \min \left (1, \dfrac{l^3}{|z - c_{ij}|^3} \right ).
 \]
Therefore, using standard Vitushkin scheme, we see that there exists $\hat f\in R^t (K, \mu) \cap L^\infty (\mu)$ and a subsequence $\{l_k\}$ such that $f_{l_k}$ converges to $\hat f$ in weak$^*$ topology in $L^\infty (\mu )$. Clearly, $\hat f(z) = \chi_G$ for $z\in \mathbb C \setminus \partial G$, where $G= \cup G_n$ is a simply connected region. For $\lambda\notin \partial G$, $\frac{f_{l_k}(z) - f_{l_k}(\lambda)}{z - \lambda}$ uniformly converges to $\frac{\hat f(z) - \hat f(\lambda)}{z - \lambda}$ on any compact subsets of $(\partial G)^c$. Therefore, $\frac{\hat f(z) - \hat f(\lambda)}{z - \lambda}\in R^t (K, \mu) \cap L^\infty (\mu)$ and   
\[
 \ \mathcal C (\hat fg\mu)(z) = \chi_G\mathcal C (g\mu)(z),~ \gamma |_{(\partial G)^c}-a.a.
\]
for $g\perp R^t(K,\mu)$. Thus, by Lemma \ref{characterizationF}, there exists $\Delta$ with $G \subset \Delta \subset \overline G$ such that $\hat f = \chi_{\Delta}$. If $\mu (\Delta) =0$, then  $G \subset \text{abpe} (R^t(K, \mu))$ as $\lambda\in K$ and $\partial K \subset \text{spt}(\mu)$. Thus, $G\subset \mathcal R, ~\gamma-a.a.$ by Theorem \ref{ABPETheorem}, which contradicts to \eqref{lemma3Eq1}. Therefore, $\hat f$ is a non-trivial characteristic function, which contradicts to the assumption that $S_\mu$ is irreducible.  

It remains to construct $g_{ij}, h_{ij}$ satisfying \eqref{gProperty}, \eqref{hProperty}, and \eqref{ghProperty}. For $l = \frac{1}{2^n}$ and $2S_{ij}\cap \gamma_n \ne \emptyset$,  there exists a heavy square $S$ with $S\subset 3S_{ij}$ such that
 \[
 \ \gamma(S \setminus \mathcal F) < \epsilon l .
 \]
We divide $S$ into 9 equal small squares. Let $S_d$ (with center $c_d$), $S_u$ (with center $c_u$) be the left bottom and upper squares, respectively. Then, using Theorem \ref{TTolsa} (2), we conclude that
 \[
 \ \gamma (S_d\cap \mathcal F) \ge \dfrac{1}{A_T} \gamma(S_d) - \gamma (S_d\setminus \mathcal F) \ge \dfrac{1-12A_T\epsilon}{12A_T}l \ge \dfrac{1}{18A_T}l. 
 \]
Similarly, $\gamma (S_u\cap \mathcal F) \ge \frac{1}{18A_T}l$.
Using Lemma \ref{hFunction}, we find a function $h_d\in R^t(K,\mu)\cap L^\infty(\mu)$ bounded analytic off $S_d$ and a finite positive measure $\eta_d$ supported in $S_d\cap \mathcal F$ with $\|\eta_d\|= \gamma (S_d\cap \mathcal F)$ satisfying the properties of Lemma \ref{hFunction}. Similarly, there is a function $h_u\in R^t(K,\mu)\cap L^\infty(\mu)$ bounded analytic off $S_u$ and a finite positive measure $\eta_u$ supported in $S_u\cap \mathcal F$ with $\|\eta_u\|= \gamma (S_u\cap \mathcal F)$ satisfying the properties of Lemma \ref{hFunction}. Let
 \[
 \ h_0 = \dfrac{\|\eta_u\|}{l}h_d - \dfrac{\|\eta_d\|}{l}h_u. 
 \]
Then $\alpha (h_0) = 0$ and
 \[
 \ \begin{aligned}
 \  |\beta(h_0, c_{ij}) | = & \dfrac{1}{l} \left | \int\int (w-z)d\eta_u(w)d\eta_d(z)\right | \\
 \ \ge & \dfrac{1}{l} \int \int Im(w-z)d\eta_u(w)d\eta_d(z) \\
 \ \ge & \dfrac{l^2}{972A_T^2}. 
 \end{aligned}
 \]
Set 
 \[
 \ g_{ij} = \dfrac{h_d}{\alpha(h_d)} \alpha(f_{ij}) \text{ and } h_{ij} = \dfrac{h_0}{\beta(h_0, c_{ij})} \beta(f_{ij} - g_{ij}, c_{ij}). 
 \]
It is easy to verify that $g_{ij}$ and $h_{ij}$ satisfy \eqref{gProperty}, \eqref{hProperty}, and \eqref{ghProperty}.       
\end{proof}
\smallskip

\begin{remark}
In section 5.2, we will introduce the modified Vitushkin scheme of P. Paramonov that estimates $\beta(h_0)$ for a group of squares. The function $h_0$ above is a special case of a key lemma (Lemma 2.7) in \cite{p95}. 
\end{remark}

\bigskip

\chapter{Nontangential limits and indices of invariant subspaces}
\bigskip

In this chapter, we assume that $1\le t <\infty$, $\mu$ is a finite positive Borel measure supported on a compact subset $K\subset \mathbb C$, $S_\mu$ on $R^t(K,\mu)$ is pure, and $K=\sigma(S_\mu)$. Let $\{g_n\}_{n=1}^\infty \subset R^t(K,\mu) ^\perp$ be a dense subset. Let $\Gamma$, $\{\Gamma_n\}$, and $h$ be as in Lemma \ref{GammaExist}.

We show that for $f\in R^t(K,\mu)$, there exists a unique $\rho(f)$ defined $\gamma-a.a.$ satisfying:
$\rho(f)(z) \mathcal C(g\mu) (z) = \mathcal C(fg\mu) (z),~ \gamma-a.a.$
for all $g\perp R^t(K,\mu)$. 
Theorem \ref{MTheorem2} proves $\rho(f)$ is $\gamma$-continuous at almost every $\lambda\in \mathcal R$ with respect to $\gamma$.    
It is also shown, in Theorem \ref{MTheorem3}, that $\rho(f)$ has nontangential limits either on $\mathcal F_+$ or $\mathcal F_-$ in full analytic capacitary density. In section 3, as applications, in case that $\mbox{abpe}(R^t(K, \mu))$ is adjacent to one of these Lipschitz graphs $\{\Gamma_n\}$, one can generalizes Theorem III to the space $R^t(K, \mu)$ (Theorem \ref{NonTL} and Theorem \ref{IndexForIS}).

\bigskip

\section{Definition of $\rho(f)$}
\bigskip   

It is well known that every function $f\in R^t(K, \mu)$ can be extended to an analytic function $\rho(f)$ on $\text{abpe}(R^t(K, \mu))$, that is, $\rho(f)(\lambda ) = (f, k_\lambda)$, where $k_\lambda$ is the reproducing kernel for $R^t(K, \mu)$. Furthermore, $\rho(f)(\lambda ) = f(\lambda ), ~ \mu |_{abpe(R^t(K, \mu))}-a.a.$. Clearly, $\frac{f(z) - \rho(f)(\lambda)}{z - \lambda}\in R^t(K, \mu)$. Hence,
 \begin{eqnarray}\label{BasicEq}
 \ \mathcal C(fg\mu)(\lambda) = \rho (f)(\lambda) \mathcal C(g\mu)(\lambda),~ \lambda \in \text{abpe}(R^t(K, \mu)), ~\gamma -a.a. 
 \end{eqnarray}
for every $g \perp R^t(K, \mu)$. As we know that $\text{abpe}(R^t(K, \mu))$ may be empty, it is meaningful to extend the identity \eqref{BasicEq} beyond the set $\text{abpe}(R^t(K, \mu))$. 

Now we define the function $\rho (f)$ for $f\in R^t(K, \mu)$ as the following. Let $\{r_n\}\subset Rat(K)$ satisfy: 
 \begin{eqnarray}\label{BasicEq2}
 \ \|r_n - f\|_{L^t(\mu )}\rightarrow 0, ~ r_n\rightarrow f,~ \mu-a.a.
 \end{eqnarray}
as $n\rightarrow \infty$. From Lemma \ref{lemmaBasic7}, we can choose a subsequence $\{r_{n_m}\}$ such that
 \begin{eqnarray}\label{BasicEq21}
 \ r_{n_m} (z) \mathcal C(g_j\mu)(z) \rightarrow \mathcal C(fg_j\mu)(z), ~ \gamma-a.a.
 \end{eqnarray}
as $m\rightarrow \infty$ and for all $j\ge 1$. Define $\rho  (f)$ as the limit of $r_{n_m}$:
 \[
 \ \text{Ran}(f) = \left \{ \lambda\in K:~ \lim_{m\rightarrow \infty} r_{n_m}(\lambda) \text{ exists, }\rho  (f) (\lambda) := \lim_{m\rightarrow \infty} r_{n_m}(\lambda) \right\}.
 \]
Then $\text{bpe}(R^t(K. \mu)) \subset \text{Ran}(f)$. From \eqref{BasicEq21},
 \[
 \ \mathcal N(g_j\mu, 1 \le j <\infty) :=  \bigcup_{j=1}^\infty \mathcal N(\mathcal C(g_j\mu)) \subset \text{Ran}(f),~ \gamma-a.a.. 
 \]
Using Theorem \ref{FCForR}, we get, for each $g\perp R^t(K,\mu)$, 
$\mathcal N(\mathcal C(g\mu)) \subset \text{Ran}(f),~ \gamma-a.a.$.
Therefore, from Lemma \ref{lemmaBasic7}, \eqref{BasicEq} can be extended as the following: 
 \begin{eqnarray}\label{BasicEq22}
 \ ~ \rho  (f)(z) \mathcal C(g\mu)(z) = \mathcal C(fg\mu)(z),~ \gamma - a.a..
 \end{eqnarray}
\smallskip

\begin{lemma}\label{lemmaBasic8}
Let $f\in R^t(K,\mu)$ and let $\{r_n\}\subset Rat(K)$ be as in \eqref{BasicEq2}. Then
 \begin{eqnarray}\label{BasicEq23}
 \ \begin{aligned}
 \ &\text{bpe}(R^t(K. \mu)) \subset \text{Ran}(f), \\
 \ & \mathcal R \subset \mathcal N(g_j\mu, 1 \le j <\infty) \subset \text{Ran}(f), ~\gamma-a.a.,
\ \end{aligned} 
 \end{eqnarray}
the definition of $\rho  (f)$ on $\text{bpe}(R^t(K. \mu))$ is unique, and $\rho  (f)$ on $\mathcal N(g_j\mu, 1 \le j <\infty) \setminus \text{bpe}(R^t(K. \mu))$ is independent of choices of $\{r_n\}$ up to a set of zero analytic capacity.
\end{lemma}

\begin{proof} 
\eqref{BasicEq23} follows from Definition \ref{FRDefinition1} and Theorem \ref{FRProperties}:
 \[
 \ \mathcal R \subset \mathcal F_0 ^c \approx \mathcal N(g_j\mu, 1 \le j <\infty),~\gamma-a.a.. 
 \]
 The definition of $\rho (f)$ is independent of choices of $\{r_n\}$ on $\text{bpe}(R^t(K. \mu))$ since 
 \[
 \ \rho(f)(\lambda) = (f, k_\lambda),~ \lambda \in \text{bpe}(R^t(K. \mu)).
 \]
Let $\{R_n\}\subset Rat(K)$ be another sequence satisfying \eqref{BasicEq2} and let $\{R_{m_k}\}$ satisfy \eqref{BasicEq21}. Set
$\rho_R(f) (z) = \lim_{k\rightarrow \infty}R_{m_k}(z)$ 
whenever the limit exists, in particular, for $z\in \mathcal N(g_j\mu, 1 \le j <\infty)$. Hence,
 \[
 \ \rho  (f)(z) \mathcal C(g_j\mu)(z) = \rho_R  (f)(z) \mathcal C(g_j\mu)(z) = \mathcal C(fg_j\mu)(z),~ \gamma - a.a., 
 \]
which implies $\rho  (f)(z) = \rho_R  (f)(z),~ \gamma - a.a.$ for $z\in \mathcal N(g_j\mu, 1 \le j <\infty)$.   
\end{proof}
\smallskip

The function $\rho(f)$ is also independent of $\{g_j\}$ up to a set of zero analytic capacity by Corollary \ref{NRBUnique}.
From above analysis, it is reasonable to define $\rho (f)(z) = 0$ for $z\in \mathcal F_0$. So except a set of zero analytic capacity $\mathbb Q$, the function $\rho(f)$ is well defined on $\mathbb C\setminus \mathbb Q$ and
\begin{eqnarray}\label{BasicEq3}
 \ \rho  (f)(z) = f (z), ~ \mu | _{\mathcal N(g_j\mu, 1 \le j <\infty) \setminus \mathbb Q} - a.a..
 \end{eqnarray}
 The following equality will be used later.
 \begin{eqnarray}\label{FEquality}
 \ \mathcal C (fg\mu)(z) = f(z)\mathcal C (g\mu)(z), ~\mathcal H^1 |_{\mathcal N(h)}-a.a.
 \end{eqnarray}
for all $g\perp R^t(K,\mu)$.
\smallskip

\begin{proposition}\label{Rhoprop} (1) For $f_1,f_2\in R^t(K,\mu)\cap L^\infty(\mu)$, 
 \[
 \ \rho(f_1f_2)(z) = \rho(f_1)(z)\rho(f_2)(z),~\gamma-a.a..
 \]

(2) For $f\in R^t(K,\mu)\cap L^\infty(\mu)$, $\|\rho(f)\|_{L^\infty(\mathcal L^2_{\mathcal R})} \le  \|f\|_{L^\infty(\mu)}$.
\end{proposition}

\begin{proof}
(1) For $f_1,f_2\in R^t(K,\mu)\cap L^\infty(\mu)$ and $g\perp R^t(K,\mu)$, we see that $f_2g\perp R^t(K,\mu)$. From \eqref{BasicEq22}, we get
 \[
 \ \begin{aligned}
 \ \rho  (f_1f_2)(z) \mathcal C(g\mu)(z) = &\rho  (f_1)(z)\mathcal C(f_2g\mu)(z) \\
 \ = &\rho  (f_1)(z)\rho  (f_2)(z)\mathcal C(g\mu)(z),~ \gamma - a.a.. 
 \ \end{aligned}
 \]
Hence, (1) follows.

(2) For $f\in R^t(K,\mu)\cap L^\infty(\mu)$ and $g\perp R^t(K,\mu)$, 
 \[
 \ \begin{aligned}
 \ \int_{\mathcal R} |\rho(f)(z)|^n |\mathcal C(g\mu)(z) | d \mathcal L^2(z) = & \int_{\mathcal R} |\mathcal C(f^ng\mu)(z) | d \mathcal L^2(z) \\
 \ \le & \int_{\mathcal R} \int \dfrac{|f(w)|^n |g(w)|}{|w-z|}d\mu(w) d \mathcal L^2(z) \\
\ \le &C_{14}\sqrt{\mathcal L^2(\mathcal R)}\|g\|_{L^1(\mu)}\|f\|^n_{L^\infty(\mu)},
 \ \end{aligned}
 \]
which implies 
 \[
 \ \left (\int_{\mathcal R} |\rho(f)(z)|^n |\mathcal C(g\mu)(z)| d \mathcal L^2(z) \right )^{\frac{1}{n}}\le \left (C_{14}\sqrt{\mathcal L^2(\mathcal R)}\|g\|_{L^1(\mu)} \right )^{\frac{1}{n}}\|f\|_{L^\infty(\mu)}. 
\]
Taking $n\rightarrow \infty$, we get $\|\rho(f)\|_{L^\infty(\mathcal L^2|_{\mathcal N(\mathcal C(g\mu))})} \le  \|f\|_{L^\infty(\mu)}$. Now (2) follows from the fact that $\mathcal R \subset \mathcal N(g_j\mu, 1 \le j <\infty),~\gamma-a.a.$.  
\end{proof}
\bigskip

\section{$\gamma$-continuity and nontangential limits}
\bigskip

We need several lemmas to prove that $\rho (f)$ is $\gamma$-continuous on $\mathcal R,~\gamma-a.a.$ and has nontangential limits in full analytic capacitary density on $\mathcal F_+\cup \mathcal F_-,~\gamma-a.a.$.
\smallskip

\begin{lemma}\label{lemmaBasic9}
Let $f\in R^t(K,\mu)$ and $g\perp R^t(K,\mu)$. Assume that for some $\lambda _0\in K$, 
\begin{eqnarray}\label{lemmaBasic9Eq1}
 \ \lim_{\delta\rightarrow 0} \dfrac{\int_{B(\lambda _0, \delta)} |g|d\mu }{\delta } = 0,
 \end{eqnarray}
 \begin{eqnarray}\label{lemmaBasic9Eq2}
 \ \lim_{\delta\rightarrow 0} \dfrac{\int_{B(\lambda _0, \delta)} |fg|d\mu }{\delta } = 0,
 \end{eqnarray}
and  
 \begin{eqnarray}\label{lemmaBasic9Eq3}
 \ \mathcal{C} (g\mu)(\lambda _0) = \lim_{\epsilon \rightarrow 0}\mathcal{C} _{\epsilon}(g\mu)(\lambda _0), ~ \mathcal{C} (fg\mu)(\lambda _0) = \lim_{\epsilon \rightarrow 0}\mathcal{C} _{\epsilon}(fg\mu)(\lambda _0)
\end{eqnarray}
exist. If $\mathcal{C} (g\mu)(\lambda _0) \ne 0$, then $a = \frac{\mathcal{C} (fg\mu)(\lambda _0)}{\mathcal{C} (g\mu)(\lambda _0)}$ is the $\gamma$-limit of $\rho (f)$ at $\lambda_0$.
 \end{lemma}

\begin{proof}
There exists $\mathbb Q\subset \mathbb C$ with $\gamma(\mathbb Q) = 0$ such that \eqref{BasicEq22} holds for $\lambda \in \mathbb Q^ c$. For $\lambda \in \mathbb Q^ c$ and $\mathcal{C} (g\mu)(\lambda) \ne 0$, we get
 \[
 \begin{aligned}
 \ & |\rho(f)(\lambda ) - a| \\
 \ = & \left |\dfrac{\mathcal{C} (fg\mu)(\lambda)}{\mathcal{C} (g\mu)(\lambda)} - \dfrac{\mathcal{C} (fg\mu)(\lambda_0)}{\mathcal{C} (g\mu)(\lambda_0)} \right| \\
 \ \le & \dfrac{|\mathcal{C} (fg\mu)(\lambda) - \mathcal{C} (fg\mu)(\lambda_0)|}{|\mathcal{C} (g\mu)(\lambda)|} + \dfrac{|\mathcal{C} (fg\mu)(\lambda_0)|  |\mathcal{C} (g\mu)(\lambda) - \mathcal{C} (g\mu)(\lambda_0)|}{|\mathcal{C} (g\mu)(\lambda)\mathcal{C} (g\mu)(\lambda_0)|}.
 \end{aligned}
 \]
If $\epsilon_0 < \frac{|\mathcal{C} (g\mu)(\lambda_0)|}{2}$, then $|\mathcal{C} (g\mu)(\lambda)| > \frac{|\mathcal{C} (g\mu)(\lambda_0)|}{2}$ for 
 \[
 \ \lambda \in A_0 := \{\lambda:~ |\mathcal{C} (g\mu)(\lambda) - \mathcal{C} (g\mu)(\lambda_0)| \le \epsilon_0\}.
 \]
 Hence,
\[
 \begin{aligned}
 \ & \{\lambda:~ |\rho (f)(\lambda) - a| >\epsilon\} \cap A_0 \\
 \ \subset & A_1:= \left \{\lambda:~ |\mathcal{C} (fg\mu)(\lambda_0)||\mathcal{C} (g\mu)(\lambda) - \mathcal{C} (g\mu)(\lambda_0)| > \dfrac{| \mathcal{C} (g\mu)(\lambda_0)|^2\epsilon}{4} \right \} \\
\ \bigcup & A_2:= \left  \{\lambda:~ |\mathcal{C} (fg\mu)(\lambda) - \mathcal{C} (fg\mu)(\lambda_0)| > \dfrac{| \mathcal{C} (g\mu)(\lambda_0)|\epsilon}{4}\right \}.
 \end{aligned}
 \]
From Lemma \ref{CauchyTLemma}, $\mathcal{C} (fg\mu)$ and $\mathcal{C} (g\mu)$ are $\gamma$-continuous at $\lambda_0$. Thus, using Theorem \ref{TTolsa} (2), we conclude that
\[
 \begin{aligned}
 \ & \lim_{\delta\rightarrow 0} \dfrac{\gamma(B(\lambda_0,\delta) \cap \{\lambda:~ |\rho(f)(\lambda) - a| >\epsilon\} \cap A_0)}{\delta} \\
 \ \le & A_T \lim_{\delta\rightarrow 0} \dfrac{\gamma(B(\lambda_0,\delta) \cap A_1)}{\delta} + A_T \lim_{\delta\rightarrow 0} \dfrac{\gamma(B(\lambda_0,\delta) \cap  A_2)}{\delta} \\
 \ = & 0.
 \end{aligned}
 \]
Using Theorem \ref{TTolsa} (2) again, since $\mathcal{C} (g\mu)$ is $\gamma$-continuous at $\lambda_0$, we get
 \[
 \begin{aligned}
 \ & \lim_{\delta\rightarrow 0} \dfrac{\gamma(B(\lambda_0,\delta) \cap \{\lambda:~ |\rho(f)(\lambda) - a| >\epsilon\})}{\delta} \\
 \ \le & A_T\lim_{\delta\rightarrow 0} \dfrac{\gamma(B(\lambda_0,\delta) \cap \{\lambda:~ |\rho(f)(\lambda) - \rho(f)(\lambda_0)| >\epsilon\} \cap A_0)}{\delta} \\
\ & + A_T \lim_{\delta\rightarrow 0} \dfrac{\gamma(B(\lambda_0,\delta) \cap A_0^c)}{\delta} \\
 \ = & 0.
 \end{aligned}
 \]    
The lemma is proved.
 \end{proof}
\smallskip

\begin{corollary}\label{BPEContinuity}
If $f \in R^t(K,\mu)$, then $\rho (f)$ is $\gamma$-continuous at $\lambda_0 \in \text{bpe}(R^t(K,\mu))$.
\end{corollary}

\begin{proof}
Using  Lemma \ref{BPELemma}, we can find a function $k\in L^{s}(\mu)$ with $k(\lambda_0) = 0$ so that $r(\lambda_0) = \int r(z) \bar k(z) d \mu(z)$ for all $r\in Rat(K)$. If $g(z) = (z-\lambda_0) \bar k(z)$, then $g\perp R^t(K,\mu)$. The function $g$ and $f\in R^t(K,\mu)$ satisfy the assumptions \eqref{lemmaBasic9Eq1}, \eqref{lemmaBasic9Eq2}, and \eqref{lemmaBasic9Eq3} of Lemma \ref{lemmaBasic9}. Moreover, $\mathcal C(g\mu)(\lambda_0) = 1$. The corollary follows from Lemma \ref{lemmaBasic9}.  
\end{proof}
\smallskip

\begin{lemma}\label{lemmaBasic10}
Let $f\in R^t(K,\mu)$ and $g\perp R^t(K,\mu)$. Let $\{\Gamma_n\}$ be as in Lemma \ref{GammaExist} and $\Gamma \in \{\Gamma_n\}$. Let $\alpha < \frac{1}{\|A'\|_\infty}$, where $A$ is the Lipschitz function for $\Gamma$. Suppose that $\mu = h\mathcal H^1 |_{\Gamma} + \sigma$ is the Radon-Nikodym 
decomposition with respect to $\mathcal H^1 |_{\Gamma}$, where $h\in L^1(\mathcal H^1 |_{\Gamma})$ and $\sigma\perp \mathcal H^1 |_{\Gamma}$.
Then there exists a subset $\mathbb Q\subset \mathbb C$ with $\gamma(\mathbb Q) = 0$, such that the following statements are true.

(a) $\mathcal C(g\mu ) (\lambda) = \lim_{\epsilon\rightarrow 0} \mathcal C_{\epsilon}(g\mu )(\lambda)$ and $\mathcal C(fg\mu ) (\lambda) = \lim_{\epsilon\rightarrow 0} \mathcal C_{\epsilon}(fg\mu )(\lambda)$ exist for $\lambda\in \mathbb C\setminus \mathbb Q$;

(b) for $\lambda_0 \in \Gamma \cap \mathcal N(h) \setminus \mathbb Q$, if $v^+(g\mu, \Gamma, \lambda_0) \ne 0$, then 
\[
 \ \lim_{\delta\rightarrow 0}\dfrac{\gamma(U_\Gamma \cap B(\lambda_0, \delta )\cap\{\lambda:~ |\rho (f)(\lambda) - f(\lambda_0)| > \epsilon\})}{\delta} = 0;  
 \]

(c) for $\lambda_0 \in \Gamma \cap \mathcal N(h)\setminus \mathbb Q$, if $v^-(g\mu, \Gamma, \lambda_0) \ne 0$, then 
\[
 \ \lim_{\delta\rightarrow 0}\dfrac{\gamma(L_\Gamma \cap B(\lambda_0, \delta )\cap\{\lambda:~ |\rho (f)(\lambda) - f(\lambda_0)| > \epsilon\})}{\delta} = 0;  
 \]

and

(d) for $\lambda_0 \in \Gamma \cap \mathcal N(h) \setminus \mathbb Q$, if $v^0(g\mu, \Gamma, \lambda_0) \ne 0$, then 
\[
 \ \lim_{\delta\rightarrow 0}\dfrac{\gamma(\Gamma \cap B(\lambda_0, \delta )\cap\{\lambda:~ |\rho (f)(\lambda) - f(\lambda_0)| > \epsilon\})}{\delta} = 0.  
 \]
\end{lemma}

\begin{proof}
The proofs of (b), (c), and (d) are similar to that of Lemma \ref{lemmaBasic9}. In fact, for (b), by \eqref{BasicEq22}, we may assume that 
$\mathcal C(fg\mu ) (\lambda_0) = \rho(f)(\lambda_0) \mathcal C(g\mu ) (\lambda_0).$
In this case, since  $\rho (f) (z) = f(z), ~h\mathcal H^1 |_{\Gamma}-a.a.$ 
(see \eqref{FEquality}),  
we assume that $\rho (f)(\lambda_0) = f(\lambda_0)$. Clearly, we have
 \[
 \ f(\lambda_0) = \dfrac{v^+(fg\mu, \Gamma, \lambda_0)}{v^+(g\mu, \Gamma, \lambda_0)}.
 \]
Hence, if we replace $\mathcal C(g\mu ) (\lambda_0)$ by $v^+(g\mu, \Gamma, \lambda_0)$ and $\mathcal C(fg\mu ) (\lambda_0)$ by $v^+(fg\mu, \Gamma, \lambda_0)$ in the proof of Lemma \ref{lemmaBasic9}, respectively, then (b) follows from
 Theorem \ref{GPTheorem1} (instead of Lemma \ref{CauchyTLemma} for $\gamma$-continuity of $\mathcal C(g\mu)$ and $\mathcal C(fg\mu)$ at $\lambda_0$ in Lemma \ref{lemmaBasic9}) and 
 the same proof of Lemma \ref{lemmaBasic9}.  
\end{proof}
\smallskip

\begin{theorem}\label{MTheorem2}
If $f\in R^t(K,\mu)$, then there exists a subset $\mathbb Q\subset \mathbb C$ with $\gamma(\mathbb Q) = 0$ such that $\rho (f)$ is $\gamma$-continuous at each $\lambda_0\in \mathcal R \setminus \mathbb Q$.
\end{theorem}

\begin{proof}
Since 
 \[
 \ \mathcal {ZD}(\mu) \subset \mathcal {ZD}(g_j\mu) \cap \mathcal {ZD}(fg_j\mu) , ~\gamma-a.a.
 \]
(see Lemma \ref{lemmaBasic0} (5)) and 
 \[
 \ \mathcal R_0 = \mathcal {ZD}(\mu) \cap \mathcal N(g_j\mu, 1 \le j< \infty)
 \]
(see Definition \ref{FRDefinition1}), using Lemma \ref{lemmaBasic9}, we conclude that for almost all $\lambda_0 \in \mathcal R_0$ with respect to $\gamma$, $\rho (f)$ is $\gamma$-continuous at $\lambda_0$. On the other hand,
 \[
 \ \begin{aligned}
 \ \mathcal R_1 = & \bigcup_{n = 1}^\infty \left (\mathcal N(g_j\mu,1 \le j < \infty, \Gamma_n, +) \cap \mathcal N(g_j\mu,1 \le j < \infty, \Gamma_n, -) \right ) \cap \mathcal {ND} (\mu)\\
\ \subset & \bigcup_{j=1}^\infty \mathcal N(\mathcal C(g_j\mu)), ~\gamma-a.a.
\end{aligned}
 \]
(see Definition \ref{FRDefinition1}). Therefore, for $\lambda_0 \in \mathcal R_1$, there exist integers $n_0$, $j_1$, $j_2$, and $j_3$ such that $\lambda_0 \in \Gamma_{n_0}$, $v^+(g_{j_1}\mu, \Gamma_{n_0}, \lambda_0) \ne 0$, $v^-(g_{j_2}\mu, \Gamma_{n_0}, \lambda_0) \ne 0$, and $v^0(g_{j_3}\mu, \Gamma_{n_0}, \lambda_0) \ne 0$. Using Lemma \ref{lemmaBasic10}, we see that (b), (c), and (d) hold for $\lambda_0$. Hence, Theorem \ref{TTolsa} (2) implies 
\[
\ \begin{aligned}
 \ &\lim_{\delta\rightarrow 0}\dfrac{\gamma(B(\lambda_0, \delta )\cap\{\lambda:~ |\rho (f)(\lambda) - f(\lambda_0)| > \epsilon\})}{\delta} \\
 \ \le & A_T\lim_{\delta\rightarrow 0}\dfrac{\gamma(B(\lambda_0, \delta )\cap U_{\Gamma_{n_0}}\cap\{\lambda:~ |\rho (f)(\lambda) - f(\lambda_0)| > \epsilon\})}{\delta} \\
 \ &+ A_T\lim_{\delta\rightarrow 0}\dfrac{\gamma(B(\lambda_0, \delta )\cap \Gamma_{n_0}\cap\{\lambda:~ |\rho (f)(\lambda) - f(\lambda_0)| > \epsilon\})}{\delta} \\
 \ & + A_T\lim_{\delta\rightarrow 0}\dfrac{\gamma(B(\lambda_0, \delta )\cap L_{\Gamma_{n_0}}\cap\{\lambda:~ |\rho (f)(\lambda) - f(\lambda_0)| > \epsilon\})}{\delta} \\
 \ = & 0.  
 \ \end{aligned}
 \]
Therefore,
$\rho (f)$ is $\gamma$-continuous at $\lambda_0$. The theorem follows from $\mathcal R \approx \mathcal R_0 \cup \mathcal R_1, ~\gamma-a.a.$
 (see Definition \ref{FRDefinition1}).
\end{proof}
\smallskip

The following theorem is a generalized version of nontangential limits (in full analytic capacitary density) of functions in $R^t(K,\mu)$. 
\smallskip

\begin{theorem}\label{MTheorem3}
If $f\in R^t(K,\mu)$, then the following properties hold:

(a) There exists a $\mathbb Q_1 \subset \mathbb C$ with $\gamma(\mathbb Q_1) = 0$ such that for all $\lambda_0\in \mathcal F_+\cap \Gamma_n\setminus \mathbb Q_1$,  
\begin{eqnarray}\label{MTheorem3Eq1}
 \ \lim_{\delta\rightarrow 0}\dfrac{\gamma(L_{\Gamma_n} \cap B(\lambda_0, \delta )\cap\{\lambda:~ |\rho (f)(\lambda) - f(\lambda_0)| > \epsilon\})}{\delta} = 0.  
 \end{eqnarray}

(b) There exists a $\mathbb Q_2 \subset \mathbb C$ with $\gamma(\mathbb Q_2) = 0$ such that for all $\lambda_0\in \mathcal F_-\cap \Gamma_n\setminus \mathbb Q_2$,  
\begin{eqnarray}\label{MTheorem3Eq2}
 \ \lim_{\delta\rightarrow 0}\dfrac{\gamma(U_{\Gamma_n} \cap B(\lambda_0, \delta )\cap\{\lambda:~ |\rho (f)(\lambda) - f(\lambda_0)| > \epsilon\})}{\delta} = 0.  
 \end{eqnarray}
\end{theorem}

\begin{proof} By Theorem \ref{FRProperties},
\[
\ \mathcal F_+ \cup \mathcal F_- \subset \mathcal F_0^c \approx \bigcup_{j=1}^\infty \mathcal N(\mathcal C(g_j\mu)), ~\gamma-a.a..
 \]
Therefore, for $\lambda_0 \in (\mathcal F_+ \cup \mathcal F_-)\cap \Gamma_n$, there exists $j_1$ such that $v^0(g_{j_1}\mu, \Gamma_n, \lambda_0) = \mathcal C (g_{j_1}\mu) (\lambda_0 ) \ne 0$. If $\lambda_0 \in \mathcal F_+$, then $v^+(g_{j_1}\mu, \Gamma_n, \lambda_0) = 0$ and $v^-(g_{j_1}\mu, \Gamma_n, \lambda_0) \ne 0$. 
(a) follows from Lemma \ref{lemmaBasic10} (c). Similarly, (b) follows from Lemma \ref{lemmaBasic10} (b).
\end{proof}
\smallskip

Certainly, one can replace $L_{\Gamma_n}$ in \eqref{MTheorem3Eq1} by the lower cone $LC(\lambda_0, \Gamma_n, \alpha)$ and $U_{\Gamma_n}$ in \eqref{MTheorem3Eq2} by the upper cone $UC(\lambda_0, \Gamma_n, \alpha)$, respectively.
\smallskip

\begin{example}\label{BVExample} 
Let $K$ be as in Example \ref{FCExample}. Let $\mu$ be a finite positive measure supported in $K$ such that $S_\mu$ is pure and $\sigma(S_\mu) = K$ (If $K$ has no interior, then $K$ is Swiss cheese). 
For $\lambda_0\in \mathbb T$, let $NR(\lambda_0)$ be the convex hull of $\{\lambda_0\}$ and $B(0, \frac{1}{2})$. If $f \in R^t(K, \mu)$, then for almost all $\lambda_0\in \mathbb T$, there exists a set $E_\delta (\lambda_0)$ such that 
$\lim_{\delta\rightarrow 0} \frac{\gamma(E_\delta (\lambda_0))}{\delta} = 0$
and
 \begin{eqnarray}\label{MTheorem3Ex}
 \ \underset{\lambda\in NR(\lambda_0)\cap B(\lambda_0, \delta)\setminus E_\delta(\lambda_0)}{\lim_{\delta\rightarrow 0}} \rho(f)(\lambda) = f(\lambda_0), ~ \mu|_{\mathbb T}-a.a..  
 \end{eqnarray}

\end{example}

\begin{proof}
It is well known that $\mu|_{\mathbb T} << \mathcal H^1 |_{\mathbb T}$ since $S_\mu$ is pure. 
\eqref{MTheorem3Ex} follows from Theorem \ref{MTheorem3}.  
\end{proof}
\bigskip

\section{Nontangential limits and indices of invariant subspaces}
\bigskip

In this section, we discuss some applications of Theorem \ref{MTheorem2} and Theorem \ref{MTheorem3}.
For an open connected set $\Omega$, in order to discuss the usual non-tangential limits, we need to require the region $\Omega$ to contain some types of non-tangential regions. 

We define the set $\mathcal {NT} _+(\Omega, \Gamma_n)$ to be the collection of $\lambda\in \partial \Omega \cap \mathcal F_- \cap \Gamma_n$ such that there exists $UC(\lambda, \alpha_\lambda, \delta_\lambda, \beta_\lambda) \subset \Omega \cap UC(\lambda, \alpha, \delta, \beta_n)$ for some $\alpha < \frac{1}{\|A_n'\|_\infty}$ and $\delta>0$.
Similarly, we define the set $\mathcal {NT} _-(\Omega, \Gamma_n)$ to be the collection of $\lambda\in \partial \Omega \cap \mathcal F_+ \cap \Gamma_n$ such that there exists $LC(\lambda, \alpha_\lambda, \delta_\lambda, \beta_\lambda) \subset \Omega \cap LC(\lambda, \alpha, \delta, \beta_n)$. 
\smallskip
 
\begin{lemma}\label{lemmaBasic11}
Let $f(z)$ be analytic on $UC(\lambda, \alpha, \delta, \beta)$. If for $\epsilon > 0$, 
\[
 \ \lim_{\delta\rightarrow 0}\dfrac{\gamma(UC(\lambda, \alpha, \delta, \beta)\cap \{|f(z) - a| > \epsilon\})}{\delta} = 0,
 \]
then for $0 < \theta < 1$,
 \[
 \ \underset{z\in UC(\lambda, \theta \alpha, \delta, \beta)}{\lim_{\delta\rightarrow 0}}f(z) = a.
 \]
\end{lemma}

\begin{proof}
There exists a constant $0 < \theta_0 < 1$ only depending on $\theta$ such that $B(w, \theta_0\delta)\subset UC(\lambda, \alpha, \delta, \beta)$ for $w\in UC(\lambda, \theta \alpha, \delta, \beta)$ and $|w - \lambda | = \frac{\delta}{2}$. Let $\delta_0 > 0$ such that for $\delta < \delta_0$, we have
 \[
 \ \gamma(UC(\lambda, \alpha, \delta, \beta)\cap \{|f(z) - a| > \epsilon\}) < \epsilon_1 \theta_0\delta
 \]
where $\epsilon_1$ is as in Lemma \ref{lemmaARS}. Applying Lemma \ref{lemmaARS}, we conclude that
 \[
 \ |f(w) - a| \le \dfrac{C_9}{\pi \delta ^2}\int _{B(w,\theta_0\delta)\setminus \{|f(z) - a| > \epsilon\}}|f(z) - a| d\mathcal L^2(z) \le C_{15} \epsilon 
 \]
for $w\in UC(\lambda, \theta \alpha, \delta, \beta)$ and $|w - \lambda | = \frac{\delta}{2}$.
The lemma follows.
\end{proof}
\smallskip

\begin{corollary} \label{BPENT}
Suppose that $UC(\lambda, \alpha, \delta, \beta) \subset \text{abpe}(R^t(K, \mu ))$. If for $f \in R^t(K, \mu )$, $\rho (f)$ is $\gamma$-continuous at $\lambda$, then
 \[
 \ \underset{z\in UC(\lambda, \theta \alpha, \delta, \beta)}{\lim_{\delta\rightarrow 0}}\rho (f)(z) = f(\lambda )
 \]
 for $0 < \theta < 1$.
\end{corollary}
\smallskip

\begin{corollary}
Let $0 < \theta < 1$.
If $\lambda\in \text{bpe}(R^t(K, \mu ))$ and $UC(\lambda, \alpha, \delta, \beta) \subset \text{int}(K)$, then there exists $\delta_0 > 0$ such that $UC(\lambda, \theta\alpha, \delta_0, \beta) \subset \text{abpe}(R^t(K, \mu ))$. Hence, for $f \in R^t(K, \mu )$, we have
 \[
 \ \underset{z\in UC(\lambda, \theta \alpha, \delta, \beta)}{\lim_{\delta\rightarrow 0}}\rho (f)(z) = f(\lambda ).
 \]
\end{corollary}

\begin{proof}
Using Lemma \ref{BPELemma}, we can find $k \in L^s(\mu)$ such that $k(\lambda) = 0$, $r(\lambda) = (r, k)$ for $r\in \text{Rat}(K)$, $g = (z - \lambda)\overline k \perp R^t(K,\mu)$, $\mathcal C(g\mu)(\lambda) = 1$, and $\mathcal C(g\mu)$ satisfies the assumptions of Lemma \ref{CauchyTLemma}. So by
  Lemma \ref{CauchyTLemma}, 	
$\mathcal C(g\mu)$ is $\gamma$-continuous at $\lambda$. Let $\theta_0$ be as in the proof of Lemma \ref{lemmaBasic11} for $w\in UC(\lambda, \theta \alpha, \delta, \beta)$ and $|w - \lambda | = \frac{\delta}{2}$. Let $\delta_0 > 0$ such that for $\delta < \delta_0$, we have
 \[
 \ \gamma(UC(\lambda, \alpha, \delta, \beta)\cap \{|\mathcal C(g\mu)(z) - 1| > \frac{1}{2}\}) < \epsilon_1 \theta_0\delta
 \]
where $\epsilon_1$ is as in Lemma \ref{lemmaARS}. Applying Lemma \ref{lemmaARS}, we have the following calculation
 \[
 \ \begin{aligned}
 \ |r(w)| \le & \dfrac{C_9}{\pi \delta ^2}\int _{B(w,\theta_0\delta)\setminus \{|\mathcal C(g\mu)(z) - 1| > \frac12\}} |r(z)| d\mathcal L^2(z) \\
\ \le & \dfrac{2C_9}{\pi \delta ^2}\int _{B(w,\theta_0\delta)\cap \{|\mathcal C(g\mu)(z)| \ge \frac12\}} |r(z)| |\mathcal C(g\mu)(z)| d\mathcal L^2(z) \\
\ \le & \dfrac{C_{16}}{\delta}\|r\|.
 \ \end{aligned} 
 \]
Hence, $UC(\lambda, \theta\alpha, \delta, \beta) \subset \text{abpe}(R^t(K, \mu ))$ for $\delta < \delta_0$. Now the result follows from Corollary \ref{BPENT}.    
\end{proof}
\smallskip

The following theorem, which extends Aleman-Richter-Sundberg's Theorem (Theorem III (a)), directly follows from Theorem \ref{MTheorem3} and Lemma \ref{lemmaBasic11}. In fact, for their case, $\text{abpe}(P^t(\mu)) = \Omega = \mathbb D$, 
$\bigcup_n \mathcal {NT} _+(\mathbb D, \Gamma_n) = \{\lambda\in \mathbb T:~ h(\lambda) > 0\}$, and $\mathcal {NT} _-(\mathbb D, \Gamma_n) = \emptyset$ for $n \ge 1$.
\smallskip

\begin{theorem}\label{NonTL}
Let $\Omega$ be a connected component of $\text{abpe}(R^t(K,\mu))$. Set $\mu _ a = h\mathcal H^1 |_\Gamma$. If $f \in R^t(K, \mu )$, then the nontangential limit $f^*(\lambda )$ of f at $\lambda\in \mathcal {NT} _+(\Omega, \Gamma_n)$ exists $\mu_a-a.a.$ and $f^* = f |_{\mathcal {NT} _+(\Omega, \Gamma_n)}$ as elements of $L^t(\mu_a |_{\mathcal {NT} _+(\Omega, \Gamma_n)}).$ The same result holds for $\mathcal {NT} _-(\Omega, \Gamma_n)$.
\end{theorem}
\smallskip

Notice that the nontangential limits of an analytic function does not depend on $\alpha$ up to a zero $\mathcal H^1$ measure set.

In Example \ref{BVExample}, if $\Omega$ is a connected component of $\text{abpe}(R^t(K, \mu))$,
\newline
$E = \bigcup_n \mathcal {NT} _+(\Omega, \Gamma_n)$, and $\mu (E\cap \mathbb T) > 0$, then \eqref{MTheorem3Ex} becomes
 \[
 \ \underset{\lambda\in NR(\lambda_0)}{\lim_{\lambda\rightarrow \lambda_0}} \rho(f)(\lambda) = f(\lambda_0), ~ \mu|_{E\cap \mathbb T}-a.a..
 \]
\smallskip

The following theorem extends Aleman-Richter-Sundberg's Theorem (Theorem III (b)).

\begin{theorem}\label{IndexForIS}
Let $\Omega$ be a connected component of $\text{abpe}(R^t(K, \mu))$. Set $\mu _ a = h\mathcal H^1 |_\Gamma$.  Suppose that there exists $n$ such that $\mu_a (\mathcal {NT} _+(\Omega, \Gamma_n) \cap \partial _1 K) > 0$ (or $\mu_a (\mathcal {NT} _-(\Omega, \Gamma_n) \cap \partial _1 K) > 0$), where $\partial _1 K$ is defined as in \eqref{BOne}. If $\mathcal M$ is a non-trivial $Rat(K)$ closed invariant subspace, then $\text{dim}(\mathcal M / ((z-\lambda)\mathcal M)) = 1$ for $\lambda \in \Omega$.   
\end{theorem}

\begin{proof}
Without loss of generality, we assume that $0\in \Omega$ and $\mu_a (\mathcal {NT} _+(\Omega, \Gamma_n) \cap \partial _1 K) > 0$. Let $\phi \perp \mathcal M$ and $f_1,f_2\in \mathcal M$, define
 \[
 \ H(\lambda ) = \int \dfrac{f_1(\lambda) f_2(z) - f_2(\lambda) f_1(z)}{z - \lambda} \phi(z) d\mu(z).
 \]
Clearly, $H(\lambda)$ is analytic on $\Omega$. It suffices to show that $H$ is identically zero. We assume that $\beta_n = 0$. For $\lambda \in \mathcal {NT} _+(\Omega, \Gamma_n) \cap \partial _1 K$, since $\partial _1 K \subset \mathcal F$ by Proposition \ref{NFSetIsBig} (3), we see that there exists $j_1$ such that $v^+ (g_{j_1}\mu, \Gamma_n, \lambda) \ne 0$, $v^0 (g_{j_1}\mu, \Gamma_n, \lambda) \ne 0$, and $v^- (g_{j_1}\mu, \Gamma_n, \lambda) = 0$. Therefore, for 
 \[
 \ \epsilon < \frac{1}{2} \min (|v^+ (g_{j_1}\mu, \Gamma_n, \lambda)|, |v^0 (g_{j_1}\mu, \Gamma_n, \lambda) |), 
 \]
we have
 \[
 \ B(\lambda, \delta)\cap U_{\Gamma_n} \setminus K \subset B(\lambda, \delta)\cap U_{\Gamma_n}\cap \{|\mathcal C(g_{j_1}\mu)(z)- v^+ (g_{j_1}\mu, \Gamma_n, \lambda) | >\epsilon \}
 \]
and
 \[
 \ B(\lambda, \delta)\cap \Gamma_n \setminus K \subset B(\lambda, \delta)\cap \Gamma_n \cap \{|\mathcal C(g_{j_1}\mu)(z)- v^0 (g_{j_1}\mu, \Gamma_n, \lambda) | >\epsilon \}.
 \]
Applying Theorem \ref{GPTheorem1}, we get
 \[
 \ \lim_{\delta\rightarrow 0} \dfrac{\gamma(B(\lambda, \delta)\cap U_{\Gamma_n} \setminus K)}{\delta} = \lim_{\delta\rightarrow 0} \dfrac{\gamma(B(\lambda, \delta)\cap \Gamma_n \setminus K)}{\delta} = 0.
 \]
Because $\underset{{\delta\rightarrow 0}}{\overline \lim} \frac{\gamma(B(\lambda, \delta) \setminus K)}{\delta} > 0$,
from Theorem \ref{TTolsa} (2), one sees that
 \[
 \ \underset{{\delta\rightarrow 0}}{\overline \lim} \dfrac{\gamma(B(\lambda, \delta)\cap L_{\Gamma_n} \setminus K)}{\delta} > 0.
 \]
Now by Theorem \ref{GPTheorem1} for $f_1\phi\mu$ and $f_2\phi\mu$, 
 \[
 \ \underset{{\delta\rightarrow 0}}{\overline \lim} \dfrac{\gamma(B(\lambda, \delta)\cap L_{\Gamma_n} \cap K^c \cap \{|\mathcal C(f_i\phi\mu) - v^- (f_i\phi\mu, \Gamma_n, \lambda)| < \epsilon \})}{\delta} > 0
 \]
for $i = 1,2$. Thus, we get 
 \[
 \ v^- (f_1\phi\mu, \Gamma_n, \lambda) = v^- (f_2\phi\mu, \Gamma_n, \lambda) = 0
 \]
as  $\mathcal C(f_1\phi\mu)(z) = \mathcal C(f_2\phi\mu)(z)  = 0$ for $z\in K^c$.
Hence,
 \[
 \ \mathcal C(f_i\phi\mu) (\lambda) = \frac{1}{2} f_i(\lambda)\phi(\lambda)h(\lambda) L^{-1}(\lambda)
 \] 
and
 \[
 \ v^+ (f_i\phi\mu, \Gamma_n, \lambda) = f_i(\lambda)\phi(\lambda)h(\lambda) L^{-1}(\lambda)
 \] 
for $\lambda \in \mathcal {NT} _+(\Omega, \Gamma_n) \cap \partial _1 K,~\mu_a-a.a.$ for $i = 1,2$. Let $E \subset \mathcal {NT} _+(\Omega, \Gamma_n) \cap \partial _1 K$ be a compact subset with $\mu_a(E) > 0$. By Theorem \ref{NonTL}, we assume that  the nontangential limits of $f_1$ and $f_2$ exist and  are bounded on $E$. Let
 \[
 \ \Omega_0 = \bigcup_{\lambda\in E}UC(\lambda, \alpha_\lambda, \delta_\lambda, \beta_\lambda).
 \]
By above arguments, we may assume that $|f_1(z)| \le M$, $|f_2(z)| \le M$ for $z\in \Omega_0$, $|h(\lambda) \phi (\lambda) L^{-1}(\lambda) | \le M$ for $\lambda\in E$, and $|f_1(z)f_2(\lambda) - f_2(z)f_1(\lambda)| < \frac{\epsilon}{2M}$ for $z\in UC(\lambda, \alpha_\lambda, \delta, \beta_\lambda)$ and $\delta < \delta_\lambda$ small enough.  
Therefore,
 \[
 \ \begin{aligned}
 \ & |H(z)| \\
 \ \le & |f_1(z)(\mathcal C(f_2\phi\mu) - (f_2\phi hL^{-1})(\lambda))| + |f_2(z)(\mathcal C(f_1\phi\mu) - (f_1\phi hL^{-1})(\lambda))| \\
 \ & + |f_1(z)f_2(\lambda) - f_2(z)f_1(\lambda)||h(\lambda)  \phi (\lambda) L^{-1}(\lambda) | \\
 \ \le & M|\mathcal C(f_2\phi\mu) - (f_2\phi hL^{-1})(\lambda)| + M|\mathcal C(f_1\phi\mu) - (f_1\phi hL^{-1})(\lambda)| + \frac{\epsilon}{2}
 \ \end{aligned} 
 \]
for $\lambda\in E$ and $z\in UC(\lambda, \alpha_\lambda, \delta, \beta_\lambda)$ for $\delta < \delta_\lambda$ small enough. Hence,
 \[
 \ \begin{aligned}
 \ & \lim _{\delta\rightarrow 0} \dfrac{\gamma(UC(\lambda, \alpha_\lambda, \delta, \beta_\lambda)\cap \{|H(z)| > \epsilon\})}{\delta} \\
 \ \le & A_T\lim _{\delta\rightarrow 0} \dfrac{\gamma(B(\lambda, \delta)\cap U_{\Gamma_n}\cap \{|\mathcal C(f_2\phi\mu) - v^+ (f_2\phi\mu, \Gamma_n, \lambda)|>\frac{\epsilon}{4M}\})}{\delta} \\
 \ & + A_T \lim _{\delta\rightarrow 0} \dfrac{\gamma(B(\lambda, \delta)\cap U_{\Gamma_n}\cap \{|\mathcal C(f_1\phi\mu) - v^+ (f_1\phi\mu, \Gamma_n, \lambda)|>\frac{\epsilon}{4M}\})}{\delta} \\
 \ = &0.
 \ \end{aligned} 
 \] 

Now applying Lemma \ref{lemmaBasic11}, we get
 \[
 \ \underset{z\in UC(\lambda, \frac{\alpha_\lambda}{2}, \delta, \beta_\lambda)}{\lim_{\delta\rightarrow 0}} H(z) = 0.
 \]
By Privalov's Theorem, we get $H(z) = 0$ as $\mu_a(E) > 0$. The theorem is proved.
\end{proof}

\bigskip

\chapter{An integral estimate and functions in $R^t(K,\mu)\cap L^\infty(\mu)$}
\bigskip

In this chapter, we prove two key lemmas (Theorem \ref{MLemma1} and Theorem \ref{MLemma2}) that will be used in proving our main theorem (Theorem \ref{DecompTheorem}).

We assume that $1\le t <\infty$, $\mu$ is a finite positive Borel measure supported on a compact subset $K\subset \mathbb C$, $S_\mu$ on $R^t(K,\mu)$ is pure, and $K=\sigma(S_\mu)$. Let $\{g_n\}_{n=1}^\infty \subset R^t(K,\mu) ^\perp$ be a dense subset. Denote $\mathcal E _N = \mathcal E (g_j\mu, 1\le j \le N)$.
\bigskip

\section{An integral estimate for $\rho(f)$}
\bigskip 

The objective of this section is to show the integral estimation \eqref{MLemma1Eq}.

\begin{theorem}\label{MLemma1}
If $f\in R^t(K, \mu)\cap L^\infty(\mu)$ and $\varphi$ is a smooth function with support in $B(\lambda, \delta)$, then
 \begin{eqnarray}\label{MLemma1Eq}
 \ \left |\int \rho(f)(z) \bar \partial \varphi (z) d\mathcal L^2(z) \right | \le C_{17} \|\rho(f)\| \delta \|\bar \partial \varphi\| \gamma(B(\lambda, 2\delta) \cap \mathcal F).
 \end{eqnarray}
\end{theorem}
\smallskip

The proof of Theorem \ref{MLemma1} relies on the lemma below.

\begin{lemma}\label{uniform}
There exists an absolute constant $C_{18} > 0$ such that
 \begin{eqnarray}\label{MLemma1Ineq}
 \ \lim_{N\rightarrow\infty} \gamma(B(\lambda, \delta)\cap \mathcal E _N) \le C_{18}\gamma(B(\lambda, 2\delta)\cap \mathcal F).
 \end{eqnarray}
\end{lemma}
\smallskip

We present our proof of Theorem \ref{MLemma1} assuming Lemma \ref{uniform} has been proved.

\begin{proof} (Theorem \ref{MLemma1} assuming \eqref{MLemma1Ineq} holds)
Let $\{r_{n}\}$ be as in Lemma \ref{lemmaBasic7}. From Lemma \ref{lemmaBasic7} (3), we find $A_\epsilon \subset \mathcal R$ such that $\gamma(A_\epsilon) < \epsilon$ and on $\mathcal R \setminus A_\epsilon$, $\{r_{n_k}\mathcal C(g_j\mu)\}$ converges to $\rho(f)\mathcal C(g_j\mu)$ uniformly for all $j \ge 1$.  Therefore, $\{r_{n_k}\}$ converges to $\rho(f)$ uniformly on $\mathcal R _{\epsilon, N} =  \mathcal R \setminus (A_\epsilon\cup \mathcal E_N)$. Set $K_{\epsilon, N} = \overline{\mathcal R _{\epsilon, N}} \subset K$. Hence, $\{r_{n_k}\} \subset R(K_{\epsilon, N})$ converges to $f_{\epsilon, N}\in R(K_{\epsilon, N})$ uniformly on $K _{\epsilon, N}$ and 
 \begin{eqnarray}\label{MLemma1RhoF}
 \ \rho(f) (z) = f_{\epsilon, N}(z), ~ z\in \mathcal R _{\epsilon, N}.
 \end{eqnarray}
 We extend $f_{\epsilon, N}$ as a continuous function on $\mathbb C$ with $\|f_{\epsilon, N}\|_{\mathbb C} \le 2\|\rho(f)\|_{L^\infty(\mathcal L^2|_{\mathcal R})}$. Let $\psi$ be a smooth function with $\text{spt}(\psi) \subset B(\lambda_0, \delta)$. We have the following calculation:
\[
 \ \begin{aligned}
 \  &|(T_\psi f_{\epsilon, N})'(\infty) | \\
 \ = &\dfrac{1}{\pi}\left | \int f_{\epsilon, N}(z) \bar \partial \psi (z) d \mathcal L^2(z) \right | \\
 \ \le & C_{19}\|\rho(f)\| \delta\|\bar\partial \psi\| \gamma(B(\lambda_0, \delta)\setminus K_{\epsilon, N}) \\
 \ \le &C_{19} \|\rho(f)\|\delta\|\bar\partial \psi\| \gamma(B(\lambda_0, \delta)\cap(\mathcal F\cup A_\epsilon\cup\mathcal E_N)) \\
 \ \le & C_{20} \|\rho(f)\| \delta\|\bar\partial \psi\| (\gamma(B(\lambda_0, \delta)\cap\mathcal F) + \gamma(A_\epsilon ) + \gamma(B(\lambda_0, \delta)\cap \mathcal E_N))
 \end{aligned}
 \]
where Theorem \ref{TTolsa} (2) is applied for the last step.  Notice that $\rho(f)(z) = 0, ~ \mathcal L^2|_{\mathcal R^c}-a.a.$, we get, by \eqref{MLemma1RhoF}, Lemma \ref{lemmaBasic0} (2), and Theorem \ref{TTolsa} (2),  
 \[
 \ \begin{aligned}
 \ &\left | \int \rho(f)(z) \bar \partial \psi (z) d \mathcal L^2(z) \right | \\
 \ \le & \left | \int (\rho(f)(z) - f _{\epsilon, N}(z))\bar \partial \psi (z) d \mathcal L^2(z) \right | + \left | \int f_{\epsilon, N}(z) \bar \partial \psi (z) d \mathcal L^2(z) \right | \\
 \ \le & \int _{\mathcal R _{\epsilon, N}^c}|\rho(f)(z) - f _{\epsilon, N}(z)||\bar \partial \psi (z)| d \mathcal L^2(z) + \pi |(T_\psi f_{\epsilon, N})'(\infty) | \\
 \ \le & C_{21} \|\rho(f)\| \| \bar \partial \psi\| \mathcal L^2(B(\lambda_0, \delta)\cap \mathcal R _{\epsilon, N}^c)  + \pi |(T_\psi f_{\epsilon, N})'(\infty) | \\
 \ \le & C_{22} \|\rho(f)\| \delta \| \bar \partial \psi\| \gamma (B(\lambda_0, \delta)\cap \mathcal R _{\epsilon, N}^c)  + \pi |(T_\psi f_{\epsilon, N})'(\infty) | \\
 \ \le & C_{23} \|\rho(f)\| \delta\|\bar\partial \psi\| (\gamma(B(\lambda_0, \delta)\cap\mathcal F) + \gamma(A_\epsilon ) + \gamma(B(\lambda_0, \delta)\cap \mathcal E_N)).
 \end{aligned}
 \]
Thus, using Lemma \ref{uniform} and taking $\epsilon\rightarrow 0$ and $N\rightarrow \infty$, we get
 \[
 \ \left | \int \rho(f)(z) \bar \partial \psi (z) d \mathcal L^2(z) \right | \le C_{23} \|\rho(f)\| \delta\|\bar\partial \psi\| \gamma(B(\lambda_0, 2\delta)\cap\mathcal F).
 \]
We complete the proof. 
\end{proof}
\smallskip

The remaining section is to prove Lemma \ref{uniform}.

\begin{lemma}\label{zeroR}
For $f\in R^t(K, \mu)\cap L^\infty(\mu)$, if $\eta$ is a finite positive measure with compact support, $\eta\in \Sigma(\text{spt}(\eta))$, and $\|\mathcal C_\epsilon (\eta)\| \le 1$ such that 
\[
 \ \mathcal C (\eta)(\lambda) = \rho(f)(\lambda), ~ \mathcal L^2_{\mathcal R}-a.a.,
 \]
then 
 \begin{eqnarray}\label{zeroREq1}
 \ \Theta_\eta(\lambda) = 0, ~ \gamma|_{\mathcal R}-a.a.
 \end{eqnarray}
and
 \begin{eqnarray}\label{zeroREq2}
 \ \mathcal C (\eta)(\lambda) = \rho(f)(\lambda), ~ \gamma |_{\mathcal R}-a.a..
 \end{eqnarray}
\end{lemma}

\begin{proof}
Suppose that 
 \begin{eqnarray}\label{zeroREq3}
 \ \gamma \{\lambda \in \mathcal R:~ \Theta_\eta^*(\lambda) > 0\} > 0. 
 \end{eqnarray}
Then using Lemma \ref{lemmaBasic0} (4) and (5), we can find a Lipschitz graph $\Gamma$ such that $\eta(\Gamma \cap \mathcal R) > 0$. Without loss of generality, we assume that the rotation angle of $\Gamma$ is zero. 
Set 
\[
 \ U_{\epsilon, \eta}^\lambda  = \{z:~ |\mathcal C(\eta)(z) - v^+(\eta,\Gamma, \lambda)| > \epsilon\} \cap U_\Gamma, 
 \]
 \[
 \ L_{\epsilon, \eta}^\lambda = \{z:~ |\mathcal C(\eta)(z) - v^-(\eta,\Gamma, \lambda)| > \epsilon\} \cap L_\Gamma, 
 \]
and 
 \[
 \ A_{\epsilon, f}^\lambda = \{z:~ |\rho(f)(z) - \rho(f)(\lambda)| > \epsilon\}.
 \]
Except for a zero $\mathcal H |_\Gamma$ set $\mathbb Q$, from Theorem \ref{GPTheorem1}, 
 \[
 \ \lim_{\delta\rightarrow 0}\dfrac{\gamma(B(\lambda,\delta)\cap U_{\epsilon, \eta}^\lambda)}{\delta} = \lim_{\delta\rightarrow 0}\dfrac{\gamma(B(\lambda,\delta)\cap L_{\epsilon, \eta}^\lambda)}{\delta} = 0,
 \]
and from Theorem \ref{DensityCorollary} and Theorem \ref{MTheorem2},
 \[
 \ \lim_{\delta\rightarrow 0}\dfrac{\gamma(B(\lambda,\delta)\cap A_{\epsilon, f}^\lambda)}{\delta} = \lim_{\delta\rightarrow 0}\dfrac{\gamma(B(\lambda,\delta)\cap \mathcal F)}{\delta} = 0
 \]
for $\lambda \in \Gamma \cap \mathcal R \setminus \mathbb Q$. Hence, from Lemma \ref{lemmaBasic0} (2),
 \[
 \ \lim_{\delta\rightarrow 0}\dfrac{\mathcal L^2 (B(\lambda,\delta)\cap U_\Gamma \cap (U_{\epsilon, \eta}^\lambda\cup A_{\epsilon, f}^\lambda \cup \mathcal F))}{\delta} = 0.
 \]
Therefore, for $\lambda \in \Gamma \cap \mathcal R\setminus \mathbb Q$, 
 \[
 \ \underset{\delta \rightarrow 0}{\overline{\lim}} \dfrac{\mathcal L^2(B(\lambda, \delta) \cap U_\Gamma \cap (U_{\epsilon, \eta}^\lambda)^c\cap (A_{\epsilon, f}^\lambda)^c \cap \mathcal R \cap\{\mathcal C (\eta)(z) = \rho(f)(z)\})}{\delta^2}> 0.
 \] 
There exists a sequence of $\{\lambda_n\}\subset \mathcal R$ tending to $\lambda$ such that $\mathcal C (\eta)(\lambda_n) = \rho(f)(\lambda_n)$,  $\mathcal C (\eta)(\lambda_n) \rightarrow v^+(\eta,\Gamma, \lambda)$, and  $\rho(f)(\lambda_n) \rightarrow \rho(f)(\lambda)$. So $\rho(f)(\lambda) = v^+(\eta,\Gamma, \lambda)$. Similarly,  $\rho(f)(\lambda) = v^-(\eta,\Gamma, \lambda)$.
Hence, $\eta (\Gamma \cap \mathcal R) = 0$, which contradicts to \eqref{zeroREq3}. Thus, \eqref{zeroREq1} is proved.

By Lemma \ref{CauchyTLemma}, $\mathcal C (\eta)(\lambda)$ is $\gamma$-continuous for $\lambda\in\mathcal R, ~ \gamma-a.a.$. Set 
 \[
 \ B_{\epsilon, f}^\lambda = \{|\mathcal C (\eta)(z) - \mathcal C (\eta)(\lambda)| < \epsilon\}.
 \]
Similarly, we can prove
\[
 \ \underset{\delta \rightarrow 0}{\overline{\lim}} \dfrac{\mathcal L^2(B(\lambda, \delta) \cap B_{\epsilon, f}^\lambda \cap (A_{\epsilon, f}^\lambda)^c\cap \mathcal R \cap\{\mathcal C (\eta)(z) = \rho(f)(z)\})}{\delta^2} > 0,
 \]
which implies \eqref{zeroREq2}.  
\end{proof}
\smallskip

Let $\phi$ be a bounded non-negative function on $\mathbb R$ supported on $[0,1]$ with $0 \le \phi(z) \le 2$ and $\int \phi(|z|) d\mathcal L^2(z) = 1$. Let $\phi_\epsilon(z) = \frac{1}{\epsilon^2} \phi (\frac{|z|}{\epsilon})$. Define the kernel function
$K_\epsilon = - \frac{1}{z} * \phi_\epsilon$.
For a finite complex-valued measure $\nu$ with compact support, define
$\tilde {\mathcal C}_\epsilon \nu = K_\epsilon * \nu$.
Clearly,
 \[
 \ \tilde {\mathcal C}_\epsilon \nu = \phi_\epsilon * \mathcal C\nu = \mathcal C(\phi_\epsilon*\nu).
 \]
It is easy to show that
 \[
 \ K_\epsilon (z) = - \dfrac{1}{z}, ~ |z| \ge \epsilon
 \]
and
$\|K_\epsilon \|_\infty \le \frac{C_{26}}{\epsilon}$.
Hence,
 \begin{eqnarray}\label{CTDiff}
 \ \begin{aligned}
 \ |\tilde {\mathcal C}_\epsilon \nu (\lambda) - \mathcal C_\epsilon \nu (\lambda)| = &\left |\int_{|z-\lambda | \le \epsilon} K_\epsilon(\lambda - z) 
d\nu
(z) \right | \\
 \ \le &C_{26} \dfrac{|\nu|(B(\lambda, \epsilon))}{\epsilon}.
 \ \end{aligned}
 \end{eqnarray}
\smallskip

\begin{lemma}\label{distributionLemma}
Suppose that $\eta$ is a finite positive measure with compact support, is linear growth, $\|\mathcal C_\epsilon (\eta)\| \le 1$, and $\Theta_\eta(\lambda) = 0, ~ \lambda \in \mathcal R, ~ \gamma-a.a.$. Set $\nu = g_j\mu$ for a given $j$. If $|\tilde {\mathcal C}_\epsilon (\nu) (\lambda)|, ~ \mathcal M_{\nu} (\lambda)\le M < \infty, ~ \eta-a.a.$, then there are two functions $F_1\in L^\infty (\mu)$ and $F_2\in L^\infty (\eta)$ with $F_1(z) = \mathcal C(\eta),~ \mu |_{\mathcal R\setminus \mathbb Q_1}-a.a.$, where $\gamma(\mathbb Q_1)= 0$, and $F_2(z) = \mathcal C(\nu),~ \eta |_{\mathcal R}-a.a.$ such that in the sense of distribution,
 \[
 \ \bar\partial (\mathcal C (\eta)\mathcal C (\nu)) = - \pi(F_1\nu + F_2 \eta).
 \]
\end{lemma}

\begin{proof}
We choose $\phi$ as a smooth non-negative function on $\mathbb R$ supported on $[0,1]$ with $0 \le \phi \le 2$ and $\int \phi(|z|) d\mathcal L^2(z) = 1.$ Then $K_\epsilon$ is a smooth function. There is a subset $\mathbb Q_1$ with $\gamma(\mathbb Q_1) = 0$ such that $\lim_{\epsilon\rightarrow }\mathcal C_\epsilon (\eta)(\lambda) = \mathcal C (\eta)(\lambda)$ exists  and $\Theta_\eta(\lambda) = 0$ for $\lambda \in \mathcal R \setminus \mathbb Q_1$. 
So, by \eqref{CTDiff}, $\tilde {\mathcal C}_\epsilon \eta (z)$ converges to $\mathcal C\eta (z)$ on $\mathcal R \setminus \mathbb Q_1$ and $\|\tilde {\mathcal C}_{\epsilon_k} \eta \| \le C_{27}$. $\mathcal C\nu(z) = 0, ~ \mathcal L^2|_{\mathcal R ^c}-a.a.$.  
For a smooth function $\varphi$ with compact support, by Lebesgue dominated convergence theorem, we have
 \[
 \ \begin{aligned}
 \ &\int \bar \partial \varphi (z)\mathcal C(\eta)(z)\mathcal C(\nu)(z) d\mathcal L^2(z) \\
 \ = & \lim_{\epsilon\rightarrow 0 }\int \bar \partial \varphi (z)\tilde {\mathcal C}_\epsilon \eta(z)\mathcal C(\nu)(z) d\mathcal L^2(z) \\
\ = & \lim_{\epsilon\rightarrow 0 }\int \bar \partial (\varphi (z)\tilde {\mathcal C}_\epsilon \eta(z))\mathcal C(\nu)(z) d\mathcal L^2(z) - \lim_{\epsilon\rightarrow 0 }\int \varphi (z) \bar \partial ( \tilde {\mathcal C}_\epsilon \eta(z))\mathcal C(\nu)(z) d\mathcal L^2(z) \\
\ = & I - \lim_{\epsilon\rightarrow 0 }II_\epsilon.
\ \end{aligned}
 \]
We can find a sequence $\{\tilde {\mathcal C}_{\epsilon_k} \eta(z)\}$ converging to $F_1$ in $L^\infty (\mu)$ weak$^*$ topology and   
 \[
 \ \tilde {\mathcal C}_{\epsilon_k} \eta(z) \rightarrow  \mathcal C \eta(z) = F_1(z), ~ \mu|_{\mathcal R \setminus \mathbb Q_1}-a.a.
\]
(see \eqref{CTDiff}).
  It is clear that
 \[
 \ I = - \lim_{\epsilon_k\rightarrow 0 }\int \varphi (z)\tilde {\mathcal C}_{\epsilon_k} \eta(z) \bar \partial\mathcal C(\nu)(z) d\mathcal L^2(z) = \pi \int \varphi (z)F_1(z) d\nu(z).
 \]
Now we estimate II:
 \[
 \ \begin{aligned}\label{CTDistributionEq1}
\ II_\epsilon = &\int \varphi (z) \bar \partial ( \mathcal C(\phi_\epsilon* \eta)(z))\mathcal C\nu(z) d\mathcal L^2(z) \\
\  = & - \pi \int \varphi (z) \mathcal C\nu(z) (\phi_\epsilon* \eta)(z)d\mathcal L^2(z). \\
\  = & - \pi \int (\varphi \mathcal C\nu)* \phi_\epsilon(z)d\eta(z). \\
\ \end{aligned}
 \]
Set
$\Phi_\epsilon (z) = (\varphi \mathcal C\nu)* \phi_\epsilon(z) - \varphi (z) \tilde {\mathcal C}_\epsilon\nu(z)$. We want to show that  
 \begin{eqnarray}\label{CTDistributionEq2}
 \ \int \Phi_\epsilon (z) d\eta(z) \rightarrow 0,\text{ as }\epsilon \rightarrow 0.
 \end{eqnarray}
In fact, 
let $\phi_\epsilon^1(z) = z\phi_\epsilon(z)$, $\phi_\epsilon^2(z) = \bar z\phi_\epsilon(z)$, and  
\[
 \ \varphi (w) - \varphi (z) = \partial \varphi (z) (w - z) + \bar  \partial \varphi (z) (\bar w - \bar z) + O(|w-z|^2),
 \]
where $|O(|w-z|^2)| \le C_{28}|w-z|^2$. Then
 \[
 \ \begin{aligned}
 \ &|\Phi_\epsilon (z) - \partial \varphi (z) \mathcal C\nu * \phi_\epsilon^1(z) - \bar  \partial \varphi (z) \mathcal C\nu * \phi_\epsilon^2(z)|  \\
\ \le & C_{28} \int |w-z|^2 |\mathcal C\nu(w)|\phi_\epsilon(z-w)d\mathcal L^2(w) \\
\ \le & C_{29}\epsilon.
\ \end{aligned}
 \]
 On the other hand,
 \[
 \ K_\epsilon^1 (z) := (-\dfrac{1}{z} * \phi_\epsilon^1)(z) = \begin{cases}0 ,& |z| \ge \epsilon,\\1 - \int_{|w|<|z|}\phi_\epsilon(w)d\mathcal L^2(w),  
&|z| < \epsilon
\end{cases}
 \]
and
 \[
 \ K_\epsilon^2 (z) := (-\dfrac{1}{z} * \phi_\epsilon^2)(z) = \begin{cases}-\dfrac{1}{z^2}\int |w|^2\phi_\epsilon(w)d\mathcal L^2(w) ,& |z| \ge \epsilon,\\-\dfrac{1}{z^2}\int_{|w|<|z|}|w|^2\phi_\epsilon(w)d\mathcal L^2(w), 
&|z| < \epsilon.
\end{cases}
 \]
Hence, $|K_\epsilon^1 (z)| \le 1$ and $|K_\epsilon^2 (z)| \le 1$. Therefore, $|\mathcal C\nu * \phi_\epsilon^1(z)| \le \|\nu\|$, $|\mathcal C\nu * \phi_\epsilon^2(z)| \le \|\nu\|$, and 
 \[
 \ |\Phi_\epsilon (z)| \le \|\partial \varphi \|_\infty\|\nu\| + \|\bar \partial \varphi \|_\infty\|\nu\| +  
C_{21}\epsilon.
 \] 
Moreover, for $\lambda\in \text{spt}(\eta)$,
$|\mathcal C\nu * \phi_\epsilon^1(\lambda)|  \le |\nu|(B(\lambda,\epsilon))$
and
 \[
 \ \begin{aligned}
\ &|\mathcal C\nu * \phi_\epsilon^2(\lambda)| \\
 \ \le & \epsilon^2 \int_{|z-\lambda| > \epsilon} \dfrac{1}{|z-\lambda|^2} d|\nu|(z) + |\nu|(B(\lambda,\epsilon)) \\
\ \le & \epsilon^2 \sum_{n=0}^\infty \int_{2^{n+1}\epsilon > |z-\lambda| \ge 2^n\epsilon}\dfrac{1}{|z-\lambda|^2} d|\nu|(z) + |\nu|(B(\lambda,\epsilon)) \\
\ \le & \epsilon \sum_{n=0}^\infty \dfrac{2}{2^n} \dfrac{|\nu|(B(\lambda,2^{n+1}\epsilon))}{2^{n+1}\epsilon} + |\nu|(B(\lambda,\epsilon)) \\
\ \le & 4\epsilon M + |\nu|(B(\lambda,\epsilon)),~\eta-a.a.. 
\ \end{aligned}
 \]
Hence,
$\Phi_\epsilon (\lambda) \rightarrow 0,~ \eta-a.a.,~\text{ as }\epsilon \rightarrow 0$.
\eqref{CTDistributionEq2} follows from Lebesgue dominated convergence theorem.
Consequently,
\[
 \ \lim_{\epsilon \rightarrow 0}II_\epsilon = -\pi \lim_{\epsilon \rightarrow 0} \int \varphi (z) \tilde {\mathcal C}_\epsilon\nu(z)d\eta(z).
 \]
Since $\|\tilde {\mathcal C}_\epsilon (\nu)\|_{L^\infty (\eta)}\le M < \infty$, we may find a subsequence $\{\tilde {\mathcal C}_{\epsilon_k}\nu\}$ converging to $F_2$ in $L^\infty(\eta)$ weak$^*$ topology.
Clearly, 
 \[
 \ \lim_{\epsilon\rightarrow 0 } \tilde {\mathcal C}_\epsilon\nu(z) = \mathcal C\nu(z) = F_2(z), ~ \eta |_{\mathcal R \cap \mathcal {ZD}(\nu)}-a.a. 
 \]
since a zero $\gamma$ set is also $\eta$ zero set by Lemma \ref{lemmaBasic0} (9). By Lemma \ref{lemmaBasic0} (5), we see that $\eta (\mathcal R \cap \mathcal {ND}(\nu)) = 0$ since $\Theta_\eta (z) = 0, ~\eta |_{\mathcal R}-a.a.$. The lemma is proved.
\end{proof}
\smallskip

\begin{lemma}\label{etaProperty}
Let $F$ be a Borel subset such that $\mathcal C_*(g_j\mu)(z) \le M_j < \infty$ and $\mathcal M_{g_j\mu}(z ) \le M_j < \infty$ for $z \in F$. Let $\eta_N$ be a finite positive measure such that $\text{spt}(\eta_N) \subset F\cap \mathcal E_N$, $\eta_N\in \Sigma(F)$, and $\|\mathcal C_\epsilon (\eta_N)\| \le 1$ for all $\epsilon > 0$. Suppose that $\eta_N$ tends to $\eta$ in $C(\overline{F})^*$ weak$^*$ topology. Then 
\newline
(1) $\text{spt}(\eta) \subset \overline{F}$, $\eta\in \Sigma(\text{spt}(\eta))$, $\|\mathcal C_\epsilon (\eta)\| \le 1$, and $\lim_{N\rightarrow \infty} \|\eta_N\| = \|\eta \|$;
\newline
(2) $\mathcal M_{g_j\mu}(z ) \le M_j$ and $| \mathcal C_\epsilon (g_j\mu)(z)|,~|\tilde {\mathcal C}_\epsilon (g_j\mu)(z)| \le C_{30}M_j$ for $z \in \overline{F}$;
\newline
(3) there exists $f\in R^t(K,\mu) \cap L^\infty(\mu)$ such that 
 \[
 \ \Theta_\eta (z) = 0, ~\mathcal  C(\eta )(z) = \rho (f)(z), ~ \gamma |_{\mathcal R}-a.a.; 
 \]
\newline
(4) $\eta (\mathcal R) = 0$; and
\newline
(5) $\|\eta\| \le C_{31} \gamma (\text{spt}(\eta) \cap \mathcal F)$. 
\end{lemma}

\begin{proof}
(1) is trivial since $\mathcal C_\epsilon (\eta_N)$ converges to $\mathcal C_\epsilon (\eta)$ in $L^\infty (\mathbb C)$ weak$^*$ topology as $N\rightarrow \infty$.

(2) For $z \in \overline{F}$, let $\lambda_n\in F\cap B(z, \frac 1n)$, we have 
 \[
 \ \dfrac{|g_j\mu|(B(z, \delta))}{\delta}\le \dfrac{\delta+\frac 1n}{\delta}\mathcal M_{g_j\mu}(\lambda_n) \le \dfrac{\delta+\frac 1n}{\delta}M_j, 
 \]
which implies $\mathcal M_{g_j\mu}(z) \le M_j$ for $z \in \overline{F}$. 
 For $z \in F$, we have, by \eqref{CTDiff},
 \[
 \ \begin{aligned}
 \ |\tilde {\mathcal C}_\epsilon (g_j\mu)(z)|\le & |\tilde {\mathcal C}_\epsilon (g_j\mu) (z) - \mathcal C_\epsilon (g_j\mu) (z)| + |\mathcal C_\epsilon (g_j\mu) (z)| \\ 
 \ \le  & C_{32} \mathcal M_{g_j\mu} (z) + \mathcal C_* (g_j\mu) (z) \\ 
 \ \le &C_{31}M_j.
 \ \end{aligned}
 \]
Because $\tilde {\mathcal C}_\epsilon(g_j\mu) (z)$ is continuous on $\overline{F}$, using \eqref{CTDiff} again, we prove (2).

(3) As in Lemma \ref{lemmaBasic6} (3), we conclude that there exists a sequence of $\{\epsilon_k\}$ such that $\mathcal C_{\epsilon_k}(\eta_N)$ converges to $f_N$ in weak$^*$ topology in $L^\infty(\mu)$, $\|f_N\|_{L^\infty(\mu)} \le 1$,
 \begin{eqnarray}\label{eqn(3)0}
 \ \int f_Ng_jd\mu = - \int \mathcal C(g_j\mu) d\eta_N,
 \end{eqnarray}
and
 \begin{eqnarray}\label{eqn(3)}
 \ \int \dfrac{f_N(z) - \mathcal C (\eta_N)(\lambda)}{z - \lambda}g_j(z)d\mu(z) = - \int \mathcal C(g_j\mu)(z) \dfrac{d\eta_N(z)}{z - \lambda}
 \end{eqnarray}
for $\lambda \in (\text{spt}(\eta_N))^c$. We may assume that $f_N$ converges to $f$ in $L^\infty(\mu)$ weak$^*$ topology. Hence, by \eqref{eqn(3)0},
 \[
 \ \left | \int fg_jd\mu \right |= \lim_{N\rightarrow \infty}\left | \int f_Ng_jd\mu \right |\le \lim_{N\rightarrow \infty} \int |\mathcal C(g_j\mu)| d\eta_N \le \lim_{N\rightarrow \infty} \dfrac{\|\eta_N\|}{N} = 0,
 \]
which implies $f\in R^t(K,\mu)\cap L^\infty(\mu)$.

Let $ E\supset F$ satisfy $\mathcal C_*(g_j\mu)(z) \le N_j < \infty$ ($M_j < N_j$) for $z\in E $ such that $\gamma (E^c)$ is small enough (see Lemma \ref{lemmaBasic0} (10)).
Let $\psi$ be a smooth function with compact support. Then, for $n\le N$, using \eqref{eqn(3)} and the fact that $\text{spt}(\eta_N)\subset \mathcal E_N \subset \mathcal E_n$, we have 
 \begin{eqnarray}\label{eqnnN}
 \ \begin{aligned}
 \ &\left | \int_{E\cap\mathcal R} \int \dfrac{f_N(z) - \mathcal C (\eta_N)(\lambda)}{z - \lambda}g_j(z) \psi(\lambda) d\mu(z) d \mathcal L^2(\lambda) \right | \\
 \ \le &\left | \int_{E\cap\mathcal R\cap \mathcal E_n^c} \int \dfrac{f_N(z) - \mathcal C (\eta_N)(\lambda)}{z - \lambda}g_j(z) \psi(\lambda) d\mu(z) d \mathcal L^2(\lambda) \right | \\ 
\ & + C_{33} \int_{E\cap\mathcal R\cap \mathcal E_n} \int \dfrac{|g_j(z)||\psi (\lambda)|}{|z - \lambda|} d\mu(z) d \mathcal L^2(\lambda) \\
 \ \le & \int_{E\cap\mathcal R\cap \mathcal E_n^c} \int \dfrac{|\mathcal C(g_j\mu)(z)||\psi(\lambda)|}{|z - \lambda|} d\eta_N(z) d \mathcal L^2(\lambda) + C_{34} \sqrt{\mathcal L^2(E\cap\mathcal R\cap \mathcal E_n)} \\
 \ \le & C_{35} (\dfrac{1}{N} + \sqrt{\mathcal L^2(E\cap\mathcal R\cap \mathcal E_n)})
 \ \end{aligned}
 \end{eqnarray}
for $j \le N$. By Definition \ref{FRDefinition1},
 \[
 \ \mathcal R_0\cap \bigcap_{n = 1}^\infty \mathcal E_n = \emptyset,~ \gamma-a.a.
 \]
and hence,
 \[
 \ \mathcal L^2(E\cap\mathcal R\cap \mathcal E_n) \le \mathcal L^2(\mathcal R\cap \mathcal E_n) = \mathcal L^2(\mathcal R_0\cap \mathcal E_n) \rightarrow 0
 \]
as $n\rightarrow \infty$.
On the other hand,
 \[
 \ \begin{aligned}
 \ & \lim_{N\rightarrow\infty}\int_{E\cap\mathcal R} \mathcal C (f_Ng_j\mu)(\lambda) \psi (\lambda) d \mathcal L^2(\lambda) \\
 \ = & - \lim_{N\rightarrow\infty}\int f_N\mathcal C (\psi \chi_{E\cap\mathcal R}d\mathcal L^2)g_j d \mu \\
 \ = & - \int f\mathcal C (\psi \chi_{E\cap\mathcal R}d\mathcal L^2)g_j d \mu \\
 \ = & \int_{E\cap\mathcal R} \mathcal C (fg_j\mu)\psi d\mathcal L^2,
 \ \end{aligned}
 \]
and since $\mathcal C (\mathcal C (g_j\mu)\psi \chi_{E\cap\mathcal R}d\mathcal L^2)$ ($|\mathcal C (g_j\mu)(z)| \le N_j,~ z \in E$) is a continuous function,
\[
 \ \begin{aligned}
 \ & \lim_{N\rightarrow\infty}\int_{E\cap\mathcal R} \mathcal C (\eta_N)\mathcal C (g_j\mu)(\lambda) \psi (\lambda) d \mathcal L^2(\lambda) \\
 \ = & - \lim_{N\rightarrow\infty}\int \mathcal C (\mathcal C (g_j\mu)\psi \chi_{E\cap\mathcal R}d\mathcal L^2) d \eta_N \\
 \ = & - \int \mathcal C (\mathcal C (g_j\mu)\psi \chi_{E\cap\mathcal R}d\mathcal L^2) d \eta \\
 \ = &\int_{E\cap\mathcal R} \mathcal C (\eta)\mathcal C (g_j\mu)(\lambda) \psi (\lambda) d \mathcal L^2(\lambda).
 \ \end{aligned}
 \]
Hence, by first taking $N\rightarrow \infty$ and then taking $n\rightarrow \infty$ for \eqref{eqnnN}, we get
 \[
 \ \mathcal C (\eta)\mathcal C (g_j\mu)(\lambda) = \mathcal C (fg_j\mu)(\lambda) ~ \mathcal L^2_{\mathcal R} -a.a.
 \]
as $\gamma (\mathcal R \setminus E)$ is small enough (Lemma \ref{lemmaBasic0} (10)) (so is $\mathcal L^2 (\mathcal R \setminus E)$ by Lemma \ref{lemmaBasic0} (2)). 
Therefore, by \eqref{BasicEq22},
 \[
 \ \mathcal C (\eta)(\lambda) = \rho(f)(\lambda) ~ \mathcal L^2_{\mathcal R}-a.a..
 \]
Now (3) follows from  Lemma \ref{zeroR}.

(4) Since a zero $\gamma$ set is also a zero $\eta$ set (see Lemma \ref{lemmaBasic0} (9)), we get
 \[
 \ \Theta_\eta (\lambda) = 0, ~\mathcal C (\eta)(\lambda) = \rho(f)(\lambda), ~ \eta |_{\mathcal R}-a.a..
 \]
 Let $\mu = h\eta + \mu_s$ be the Radon Nikodym decomposition with respect to $\eta$, where $\mu_s\perp\eta$. From Lemma \ref{distributionLemma}, we have
 \[
 \ F_1g_j\mu + F_2 \eta = fg_j\mu.
 \]
Together with \eqref{BasicEq3}, we get
 \[
 \ F_1(z)g_j(z)h(z) = f(z)g_j(z)h(z), ~\eta|_{\mathcal R}-a.a.. 
 \]
Therefore, 
 \[
 \ F_2(z) = \mathcal C(g_j\mu)(z) = 0, ~\eta|_{\mathcal R}-a.a. 
 \]
for $j \ge 1$. Thus, $\eta (\mathcal R) = 0$ since $\mathcal R \subset \cup_{j=1}^\infty \mathcal N(\mathcal C(g_j\mu))$ by Definition \ref{FRDefinition1}.

(5)
There is an open subset $O$ such that $\text{spt}(\eta)\cap \mathcal R \subset O$ and $\eta(O) \le \frac 12 \|\eta\|$. If we define $E:= \text{spt}(\eta)\setminus O$, then $E$ is compact, $E \subset \text{spt}(\eta)\cap\mathcal F$, and
 \[ 
 \ \|\eta\| \le 2\eta(E).
 \]
Hence, by the definition of $c^2(\eta)$ and Proposition 3.3 in \cite{Tol14},
 \[
 \ c^2(\eta|_E) \le c^2(\eta) \le C_{36} \|\eta\| \le 2C_{36}\|\eta|_E\|.  
 \]
It now follows from Theorem \ref{TTolsa} (1) \eqref{GammaEq3} that 
 \[
 \ \|\eta|_E\| \le C_{31} \gamma (E) \le C_{31}\gamma(\text{spt}(\eta)\cap\mathcal F).
 \]
This completes the proof.
\end{proof}
\smallskip

We are now ready to finish the proof of Lemma \ref{uniform}.

\begin{proof} (Lemma \ref{uniform})
Set
 \[
 \ \epsilon_0 = \lim_{N\rightarrow\infty} \gamma(B(\lambda, \delta)\cap \mathcal E_N) ( = \inf_{N\ge 1} \gamma(B(\lambda, \delta)\cap \mathcal E_N)).
 \]
We assume $\epsilon_0 > 0$. The case that $\epsilon_0 = 0$ is trivial.   

From Lemma \ref{lemmaBasic0} (10), there exists a Borel subset $F\subset B(\lambda, 2\delta)$ such that
\newline
(a) $\gamma (B(\lambda, 2\delta) \setminus F) < \frac{a_T}{2A_T}\epsilon_0$, where $a_T$ and $A_T$ are as in Theorem \ref{TTolsa} (1) \eqref{GammaEq1};
\newline
(b) $\mathcal C_*(g_j\mu)(z) \le M_j < \infty$ for $z \in F$;
\newline
(c) $\mathcal M_{g_j\mu}(z ) \le M_j < \infty$ for $z \in F$.

From Theorem \ref{TTolsa} (1) \eqref{GammaEq1}, there exists a positive measure $\eta_N$ with $\eta_N\in\Sigma(\text{spt}(\eta_N))$, $\text{spt}(\eta_N) \subset B(\lambda, \delta)\cap F \cap \mathcal E_N$, and $ \| \mathcal C_\epsilon (\eta_N) \| _\infty \le 1$ such that 
 \[
 \ \begin{aligned}
 \ \|\eta_N\| \ge & \dfrac{a_T}{2} \gamma(B(\lambda, \delta)\cap F \cap \mathcal E_N) \\ 
 \ \ge &\dfrac{1}{2} ( \dfrac{a_T}{A_T} \gamma(B(\lambda, \delta)\cap \mathcal E_N) - \gamma (B(\lambda, 2\delta) \setminus F)). 
 \ \end{aligned} 
 \]
Hence, 
 \[
 \  \gamma(B(\lambda, \delta)\cap \mathcal E_N) \le \dfrac{2A_T}{a_T}\|\eta_N\| + \dfrac{A_T}{a_T}\gamma (B(\lambda, 2\delta) \setminus F) \le \dfrac{2A_T}{a_T}\|\eta_N\| + \frac 12 \epsilon_0.
 \]
Thus, $\gamma(B(\lambda, \delta)\cap \mathcal E_N) \le \dfrac{4A_T}{a_T}\|\eta_N\|$.
We may assume that $\eta_N \in C(\overline{B(\lambda, \delta)\cap F})^*\rightarrow \eta$ in weak $^*$ topology. Now the lemma follows from Lemma \ref{etaProperty}. 
\end{proof}

\bigskip

\section{Functions in $R^t(K,\mu) \cap L^\infty (\mu)$}
\bigskip

Let $\mathcal D$ be a bounded Borel subset. Define 
 \[
 \ L^\infty(\mathcal L^2_{\mathcal D}) = \{f\in L^\infty(\mathbb C):~ f(z) = 0,~ z\in \mathcal D^c\}. 
 \]
Let $GC(\mathcal D)$ consist of functions $f\in L^\infty(\mathcal L^2_{\mathcal D})$ for which $f$ is $\gamma$-continuous on $\mathcal D \setminus \mathbb Q_f$, where $\mathcal L^2(\mathbb Q_f) = 0$. We assume that $\chi_{\mathcal D}\in GC(\mathcal D)$. In this case, there exists a subset $\mathbb Q$ with $\mathcal L^2(\mathbb Q) = 0$ such that for $\lambda \in \mathcal D \setminus \mathbb Q$,
 \begin{eqnarray}\label{densityAssumption}
 \ \lim_{\delta\rightarrow 0} \dfrac{\gamma(B(\lambda, \delta)\setminus \mathcal D)}{\delta} = 0.
 \end{eqnarray}
Set $\mathcal F_{\mathcal D} = \mathcal F \setminus \mathcal D$ and let $E_0\subset \mathcal F_{\mathcal D}$ be a Borel subset. 
Let $H^\infty_{\mathcal D, E_0}(\mathcal L^2_{\mathcal D })$ denote the weak$^*$ closure in $L^\infty(\mathcal L^2_{\mathcal D})$ of functions $f$, where $f$ is bounded analytic on $\mathbb C \setminus E_f$ and $E_f$ is a compact subset of $E_0$.
In this section, we prove the following theorem.

\begin{theorem}\label{MLemma2}
Let $\mathcal D$ be a bounded Borel subset satisfying the condition \eqref{densityAssumption} and let $E_0\subset \mathcal F_{\mathcal D}$ be a Borel subset. Let $f\in GC(\mathcal D)$ be given with $\|f\|_{L^\infty (\mathcal L^2_{\mathcal D})}\le 1$. Suppose that, for $\lambda _0\in \mathbb C$, $\delta > 0$, and a smooth function $\varphi$ with support in $B(\lambda _0, \delta)$, we have 
 \begin{eqnarray}\label{MLemma2Eq}
 \ \left | \int f(z) \dfrac{\partial \varphi (z)}{\partial  \bar z} d\mathcal L^2(z) \right | \le C_{38}\delta \left \|\dfrac{\partial \varphi (z)}{\partial  \bar z} \right \|_\infty \gamma (B(\lambda _0, 2\delta) \cap E_0).
 \end{eqnarray} 
Then $f\in H^\infty_{\mathcal D, E_0}(\mathcal L^2_{\mathcal D })$ and there exists $\hat f\in R^t(K,\mu) \cap L^\infty(\mu )$ such that $\mathcal C(\hat fg\mu)(z) = f(z )\mathcal C(g\mu)(z), ~ \mathcal L^2_{\mathcal D}-a.a.$ for each $g\perp R^t(K,\mu)$ and $\|\hat f\|_{L^\infty (\mu )} \le C_{39}$. 
\end{theorem}
\smallskip

For a positive integer $N$, $\lambda \in \mathcal D$, and $f\in GC(\mathcal D)$ that is $\gamma$-continuous at $\lambda$, define
 \[
 \ F(f, N, \lambda) = \{z: ~ |f(z) - f(\lambda )| \ge \frac{1}{N^3}\}. 
 \]
Let $\mathcal D(f, m, N)$ be the set of $\lambda \in \mathcal D$ such that
 \begin{eqnarray}\label{DFMNDef1}
 \ \gamma (F(f, N, \lambda)\cap B(\lambda, \delta )) < \dfrac{1}{N^4} \delta
 \end{eqnarray}
and
 \begin{eqnarray}\label{DFMNDef2}
 \ \gamma ( B(\lambda, \delta ) \setminus \mathcal D) < \dfrac{1}{N^2} \delta
 \end{eqnarray}
for $\delta \le \frac{1}{m}$. From the definition of $GC(\mathcal D)$ and \eqref{densityAssumption}, we conclude that
 \begin{eqnarray}\label{denseSet}
 \ \mathcal L^2\left (\mathcal D \setminus \bigcup_{m=1}^\infty \mathcal D(f, m, N) \right ) = 0.
 \end{eqnarray}

Let $\delta > 0$ and $\delta_N = \frac{1}{2N+1} \delta$. 
The proof of Theorem \ref{MLemma2} relies on Lemma \ref{hFunction} and modified Vitushkin approximation scheme by P. V. Paramonov in \cite{p95}. We provide a short description of P. V. Paramonov's ideas to process our proof.

Let $\{\varphi_{ij},S_{ij}\}$ be a smooth partition of unity, where the length of the square $S_{ij}$ is $\delta_N$, the support of $\varphi_{ij}$ is in $2S_{ij}$, $0 \le \varphi_{ij} \le 1$, 
 \[
 \ \|\bar\partial \varphi_{ij} \| \le \frac{C_{40}}{\delta_N},~ \sum \varphi_{ij} = 1,
 \]
 and 
 \[
 \ \bigcup_{i,j=1}^\infty S_{ij} = \mathbb C,~ Int(S_{ij})\cap Int(S_{i_1j_1}) = \emptyset
 \]
for $(i,j) \ne (i_1,j_1)$. 
If $f_{ij} = T_{\varphi_{ij}}f$, then
\begin{eqnarray}\label{FSum1}
 \ f = \sum_{ij} f_{ij} = \sum_{2S_{ij} \cap \mathcal D \ne \emptyset} f_{ij}
 \end{eqnarray}
since $f(z) = 0, ~ \mathcal L^2|_{\mathcal D^c}-a.a.$.

For $2S_{ij} \cap \mathcal D \ne \emptyset$, \eqref{MLemma2Eq} becomes ((2.7) in \cite{p95}) 
 \begin{eqnarray}\label{Estimate1}
 \ \begin{aligned}
 \ & |\alpha(f_{ij})| \\
 \ = & \dfrac{1}{\pi}\left | \int f(z) \dfrac{\partial \varphi _{ij}(z)}{\partial  \bar z} d\mathcal L^2(z) \right | \\
 \ \le &C_{41} \gamma (B(c_{ij}, 2\sqrt{2}\delta _N) \cap E_0)
 \ \end{aligned}
 \end{eqnarray}
and replacing $\varphi _{ij}$ by $(z - c_{ij})\varphi _{ij}$ in \eqref{MLemma2Eq}, we get ((2.8) in \cite{p95})
 \begin{eqnarray}\label{Estimate12}
 \ \begin{aligned}
 \ & |\beta(f_{ij}, c_{ij})| \\
 \ \le &C_{43} \delta _N \|(z - c_{ij})\bar \partial\varphi _{ij}\|\gamma (B(c_{ij}, 2\sqrt{2}\delta _N) \cap E_0) \\
 \ \le &C_{44} \delta _N \gamma (B(c_{ij}, 2\sqrt{2}\delta _N) \cap E_0)
 \ \end{aligned}
 \end{eqnarray}
where $c_{ij}$ denotes the center of $S_{ij}$ and $\alpha (f_{ij})$ and $\beta(f_{ij}, c_{ij})$ are defined in section 3.5.
The standard Vitushkin approximation scheme requires to construct $a_{ij}$ such that $f_{ij} - a_{ij}$ has triple zeros at $\infty$, which requires to estimate both $\alpha (a_{ij})$ and $\beta(a_{ij}, c_{ij})$ (e.g. see section 3.5). 
The main idea of P. V. Paramonov is that one does not actually need to estimate
each coefficient $\beta(a_{ij}, c_{ij})$. It suffices to do away (with appropriate estimates) with the sum of coefficients $\sum_{j\in I_{is}} \beta(a_{ij}, c_{ij})$ for a special partition $\{I_{is}\}$ into non-intersecting groups $I_{is}$.

Let $\alpha _{ij}  = \gamma (B(c_{ij}, 2\sqrt{2}\delta _N) \cap E_0)$.  Set $min_i = \min\{j:~ 2S_{ij} \cap \mathcal D \ne \emptyset \}$ and $max_i = \max\{j:~ 2S_{ij} \cap \mathcal D \ne \emptyset \}$. Let $I_i = \{min_i, min_i+1,..., max_i\}$. 
We call a subset $I$ of $I_i$ a complete group of indices if the following
conditions are satisfied: 
 \[
 \ \begin{aligned}
 \ I = &\{j_s + 1,j_s + 2,...,j_s + s_1,j_s + s_1 + 1,...,j_s + s_1 + s_2, \\
 \ &j_s + s_1 + s_2+1,...,j_s + s_1 + s_2+s_3 \} \subset I_i,
 \ \end{aligned}
 \]
where $s_2$ ($k_2$ in \cite{p95}) is an absolute constant chosen as in the proof of Lemma 2.7 of \cite{p95},
 \[
 \ \delta_N \le \sum_{j=j_s+1}^{j_s+s_1}\alpha_{ij} < \delta_N + k_1\delta_N,
 \]
and 
 \[
 \ \delta_N \le \sum_{j=j_s+s_1+s_2+1}^{j_s+s_1+s_2+s_3}\alpha_{ij} < \delta_N + k_1\delta_N,
 \]
where $k_1$ ($\ge 3$) is a fixed integer and used in \cite{p95}. For our purpose, we can actually set $k_1=3$.

We now present a detailed description of the procedure of partitioning $I_i$ into
groups. we split each $I_i$ into (finitely many) non-intersecting groups $I_{il}$, $l = 1,...,l_i$, as follows.
Starting from the lowest index $min_i$ in $I_i$ we include in $I_{i1}$ (going upwards and without jumps in $I_i$) all indices until we have collected a minimal (with respect to the number of elements) complete group $I_{i1}$. Then we repeat this
procedure for $I_i\setminus I_{i1}$, and so on. After we have constructed several complete
groups $I_{i1},...,I_{il_i-1}$ (this family can even be empty) there can remain the last
portion $I_{il_i} = I_i\setminus(I_{i1}\cup...\cup I_{il_i-1})$ of indices in $I_i$, which includes no complete
groups. We call this portion $I_{il_i}$ an incomplete group (clearly, there is at most one incomplete group for each $i$).

Let $g_{ij}^1$ be as in Lemma \ref{hFunction} satisfying $g_{ij}^1(z) = \mathcal C(\eta_{ij}^1)(z)$ for $z\in (\text{spt}(\eta_{ij}^1))^c$, where $\eta_{ij}^1$ is a finite positive measure with $\text{spt}(\eta_{ij}^1) \subset 6S_{ij} \cap E_0$, $\eta_{ij}^1\in\Sigma (6S_{ij})$, $\|g_{ij}^1\| \le C_6$, $\alpha(g_{ij}^1) = \alpha_{ij}$, and $g_{ij}^1 \in R^t(K,\mu)\cap L^\infty (\mu)$. Let $g_{ij}^0 = \frac{\alpha (f_{ij})}{\alpha_{ij}}g_{ij}^1$, $\eta_{ij}^0 = \frac{\alpha (f_{ij})}{\alpha_{ij}} \eta_{ij}^1$, and $g_{ij} = f_{ij} - g_{ij}^0$. Clearly, by \eqref{Estimate1},
 \[
\ \|g_{ij}^0\| \le C_{45}, ~|\alpha(g_{ij}^0)| \le C_{45}\alpha_{ij}, ~| \beta(g_{ij}^0, c_{ij})| \le C_{45}\delta _N \alpha_{ij}.
 \]
Hence, by \eqref{Estimate12},  
\begin{eqnarray}\label{GConditions}
 \ \|g_{ij}\| \le C_{46}, ~\alpha(g_{ij}) = 0, ~| \beta(g_{ij}, c_{ij})| \le C_{46}\delta _N \alpha_{ij}.  
 \end{eqnarray}
Let $I=I_{il}$ be a group, define
 \[
 \ g_I = \sum_{j\in I} g_{ij}, ~ \alpha(g_I) = \sum_{j\in I} \alpha(g_{ij}),~\beta (g_I) = \sum_{j\in I} \beta (g_{ij},c_{ij}),
 \]
and let $I'(z) = \{j\in I:~|z-c_{ij}| >3k_1\delta_N  \}$, 
 \[
 \ L_I' (z) = \sum_{j\in I'(z)} \left ( \dfrac{\delta_N\alpha_{ij}}{|z - c_{ij}|^2} + \dfrac{\delta_N^3}{|z - c_{ij}|^3} \right ).
 \]
Define $L_I(z) = L_I' (z)$ if $I = I'(z)$, otherwise, $L_I(z) = 1+L_I' (z)$. 
It is proved in (2.22) of \cite{p95}, by \eqref{GConditions},  that 
 \begin{eqnarray}\label{Eq2.22}
 \ |g_I(z)| \le C_{47} L_I(z), ~ |g_I(z)| \le C_{47}, ~ \alpha(g_I) = 0,~|\beta (g_I)| \le C_{47}\delta_N^2.
 \end{eqnarray}

The follow lemma is due to Lemma 2.7 in \cite{p95}.  

\begin{lemma} For each complete group $I_{il}$, there exists a function
$h_{I_{il}}$ that has the following form
 \begin{eqnarray}\label{HFunction}
 \ \begin{aligned}
 \ & h_{I_{il}}^0 = \sum_{j\in I_{il}^1}\sum_{k\in I_{il}^2} \left(H^i_{jk} : = \dfrac{\delta_N}{|c_{ik} - c_{ij}|}(\lambda_{ik}^jh_{ik} - \lambda_{ij}^kh_{ij})\right ), \\
 \ &h_{I_{il}} = \dfrac{\beta (g_{I_{il}})}{\beta (h_{I_{il}}^0)}h_{I_{il}}^0, 
\ \end{aligned}
 \end{eqnarray}
where $I_{il}^1 = (j_s+1,...,j_s+s_1)$ and $I_{il}^2 = (j_s+s_1+s_2+1,...,j_s+s_1+s_2+s_3)$,
satisfying:

(H1) $h_{ij}$ is bounded analytic off $6S_{ij}$ from Lemma \ref{hFunction} and $h_{ij}(z) = \mathcal C(\eta_{ij})(z)$ for $z\in (\text{spt}(\eta_{ij}))^c$, where $\eta_{ij}$ is a finite positive measure with $\eta_{ij} \in \Sigma(6S_{ij})$ and  $\text{spt}(\eta_{ij}) \subset 6S_{ij} \cap E_0$, satisfying 
 \begin{eqnarray}\label{HConditions}
 \ \|h_{ij}\|_\infty \le C_{48}, ~ \alpha(h_{ij}) = \alpha_{ij};
 \end{eqnarray}

(H2) $\lambda_{ij}^k, \lambda_{ik}^j \ge 0$ and 
 \[
 \ \sum_{j\in I_{il}^1}\lambda_{ik}^j \le 1, ~ \sum_{k\in I_{il}^2} \lambda_{ij}^k \le 1;
 \]  

(H3)
 \[
 \ \sum_{j\in I_{il}^1}\sum_{k\in I_{il}^2}\lambda_{ij}^k\alpha_{ij} = \delta _N, ~ \sum_{j\in I_{il}^1}\sum_{k\in I_{il}^2}\lambda_{ik}^j\alpha_{ik} = \delta _N; 
 \]

(H4) $\alpha (H^i_{jk}) = 0$, that is, $\lambda_{ik}^j\alpha_{ik} = \lambda_{ij}^k\alpha_{ij}$; and

(H5) if $|z - c_{ij}| > 3k_1\delta_N$ and $|z - c_{ik}| > 3k_1\delta_N$, then
 \[
 \ |H^i_{jk} (z)| \le C_{49}\left (\dfrac{\lambda_{ij}^k\alpha_{ij}\delta_N}{|z - c_{ij}|^2} + \dfrac{\lambda_{ik}^j\alpha_{ik}\delta_N}{|z - c_{ik}|^2}\right)
 \]
and for all $z \in \mathbb C,$
 \begin{eqnarray}\label{Eq2.25}
 \ |h_{I_{il}}(z)| \le C_{50} L_{I_{il}}(z),  ~ \alpha (h_{I_{il}}) = 0, ~ \beta (h_{I_{il}}) = \beta (g_{I_{il}}).
 \end{eqnarray}
\end{lemma}

We rewrite \eqref{FSum1} as the following
 \[
 \ f = \sum_{i} \sum_{l = 1}^{l_i-1} (g_{I_{il}} - h_{I_{il}}) + \sum_{i} \sum_{j\in I_{il_i}} g_{ij} + f_{\delta_N}
 \]
where
 \begin{eqnarray}\label{FDelta}
 \ f_{\delta_N} = \sum_{i} \sum_{j\in I_{il}, 1\le l \le l_i} g_{ij}^0 + \sum_{i} \sum_{l = 1}^{l_i-1} h_{I_{il}} \in R^t(K,\mu)\cap L^\infty(\mu).
 \end{eqnarray}

Let $L_1$ be a subset of complete groups $I_{il}$. Let $L_2$ be a subset of index groups $I\subset I_{il}$ for some $l$ (at most three $l$) such that for each $i$, there are at most three $I \subset I_{il}$ in $L_2$, denoted by $I_i^1,I_i^2,I_i^3$. We can view $L_2$ as a collection of incomplete groups. Define $\Psi_{iI_{il}}(z) = g_{I_{il}}(z) - h_{I_{il}}(z)$ for $1 \le l \le l_i - 1$. 
 The following inequality is basically proved in \cite{p95} though some modifications are needed.
\begin{eqnarray}\label{FBounded1}
\ \sum_{I_{il}\in L_1} |\Psi_{iI_{il}}(z)| + \underset{1\le k \le 3}{\sum_{I_i^k\in L_2}} |g_{I_i^k}(z)| \le C_{51} \text{min} \left (1, \left (\dfrac{\delta _N}{\text{dist}(z,\partial R)} \right )^\frac14 \right) 
\end{eqnarray}
where $R$ is a square (union of some $S_{ij}$, for the proof of \eqref{keyEstimate1}, $R = (2N+1)S$ for some $S\in \{S_{ij}\}$ and we need  $R\cap S_{ij} = \emptyset$ for $j\in I$, $I \in L_1\cup L_2$). Since the proof of \eqref{FBounded1} involves many technical estimations, we will provide a short description at the end of our theorem proof.

Now assuming \eqref{FBounded1} holds, we claim that for $N > 3k_1$ (recall $\delta = \frac{1}{m}$ and $\delta_N = \frac{1}{2N+1}\delta$),
 \begin{eqnarray}\label{keyEstimate}
 \ \int _{\mathcal D(f, m, N)} |f(z) - f_{\delta_N}(z)| d\mathcal L^2(z) \le \dfrac{C_{52}}{N^\frac14}.
 \end{eqnarray}

If \eqref{keyEstimate} is proved, we can prove Theorem \ref{MLemma2} easily as the following.

\begin{proof} (Theorem \ref{MLemma2} assuming \eqref{keyEstimate} holds) 
Clearly, by \eqref{FBounded1}, $\|f_{\delta_N}\|_{L^\infty (\mathcal L^2_{\mathcal D})} \le C_{53}$ and $f_{\delta_N}$ is bounded analytic on $\mathbb C\setminus F_N$, where $F_N \subset E_0$ is a compact subset. Moreover, $\|f_{\delta_N}\|_{L^\infty (\mu)} \le C_{53}$ and $f_{\delta_N}\in R^t(K, \mu)\cap L^\infty(\mu)$. Hence, there exists $m_N$ such that, from \eqref{denseSet}, 
 \[
 \ \mathcal L^2(\mathcal D \setminus \mathcal D(f, m_N, N) ) < \dfrac{1}{N}.
 \]
Set $f_N = f_{\delta_N}$. Clearly, by \eqref{keyEstimate},
 \[
 \ \|f_N - f\|_{L^1 (\mathcal L^2_{\mathcal D})}\rightarrow 0.
 \]
Therefore, $f_N\rightarrow f,~ \mathcal L^2_{\mathcal D}-a.a.$ (by passing to a subsequence). Thus, $f_N\rightarrow f$ in $L^\infty(\mathcal L^2_{\mathcal D})$ weak$^*$ topology. This implies $f\in H^\infty_{\mathcal D, E_0}(\mathcal L^2_{\mathcal D })$ as $f_N\in H^\infty_{\mathcal D, E_0}(\mathcal L^2_{\mathcal D })$. We may also assume that $f_N\rightarrow \hat f$ in $L^\infty(\mu)$ weak$^*$ topology.  For $g\perp R^t(K,\mu)$, by Lemma \ref{hFunction}, we have 
 \[
 \ f_N(\lambda) \mathcal C (g\mu) (\lambda) = \mathcal C (f_Ng\mu) (\lambda),~ \mathcal L^2_{\mathcal D}-a.a..
 \]
For $\lambda\in \mathcal D$ with $\int \frac{1}{|z-\lambda|}|g(z)|d\mu (z) < \infty$, we conclude $\mathcal C (f_Ng\mu) (\lambda)\rightarrow \mathcal C (\hat fg\mu) (\lambda),~ \mathcal L^2_{\mathcal D}-a.a.$. Therefore, 
 \[
 \ \mathcal C (\hat fg\mu) (\lambda) = f(\lambda)\mathcal C (g\mu) (\lambda), ~ \mathcal L^2_{\mathcal D}-a.a. 
 \]
Certainly, $\|\hat f\|_{L^\infty(\mu)} \le C_{53}$.
 This proves the theorem.
\end{proof}
\smallskip

To prove \eqref{keyEstimate}, we have
\[
 \ \int _{\mathcal D(f, m, N)} |f(z) - f_{\delta_N}(z)| d\mathcal L^2(z) = \sum_{ij}  \int _{S_{ij} \cap \mathcal D(f, m, N)}|f(z) - f_{\delta_N}(z)| d\mathcal L^2(z).
 \] 
We now fix a square $S\in \{S_{ij}\}$ with $S \cap \mathcal D(f, m, N) \ne \emptyset$. It is suffice to prove 
\begin{eqnarray}\label{keyEstimate1}
 \ \int _{S \cap \mathcal D(f, m, N)} |f(z) - f_{\delta_N}(z)| d\mathcal L^2(z) \le \dfrac{C_{54}}{N^\frac14} \mathcal L^2(S).
 \end{eqnarray}
\smallskip

\begin{proof} (of \eqref{keyEstimate1}) Fix $S = S_{i_0j_0}$.
Let $J$ be the set of indices $(i,l)$ for $1 \le l \le l_i - 1$ such that there is a square $S_{ij}$ in the complete group $I_{il}$ satisfying $6S_{ij} \subset (2N+1)S$. 
Let $J_0$ be the set of indices $i$ such that there is a square $S_{ij}$ in the incomplete group $I_{il_i}$ satisfying $6S_{ij} \subset (2N+1)S$. Let $J_1$ be the subset of index $i$ such that there exists $l$ with $(i,l)\in J$. 
From \eqref{FBounded1}, for $z \in S$, we get
 \[
 \ \sum_{(i,l)\notin J} |g_{I_{il}}(z) - h_{I_{il}}(z)| + \sum_{i\notin J_0} |g_{I_{il_i}}(z)| \le \dfrac{C_{55}}{N^\frac{1}{4}}.
 \]
Therefore, for $z\in S$,  
\begin{eqnarray}\label{FBounded2}
  \ \begin{aligned}
 \ &  |f(z) - f_{\delta_N}(z)|\\
 \ \le & \dfrac{C_{56}}{N^\frac14} + \sum_{i\in J_1}\sum_{(i,l)\in J} |g_{I_{il}}(z)| + \sum_{i\in J_1}\sum_{(i,l)\in J} | h_{I_{il}}(z)| + \sum_{i\in J_0} |g_{I_{il_i}}(z)|.
 \ \end{aligned} 
 \end{eqnarray}
Set $I_{il}^u = \{j\in I_{il}:~j > j_0,~ 6S_{ij} \setminus (2N+1)S \ne \emptyset\}$ and $I_{il}^d = \{j\in I_{il}:~j < j_0,~ 6S_{ij} \setminus (2N+1)S \ne \emptyset\}$. For each $i$ with $(i,l)\in J$ or $i\in J_0$, there is at most one $I_{il}^u$ with $I_{il}^u \ne \emptyset$ and one $I_{il}^d$ with $I_{il}^d \ne \emptyset$. Then
 \[
 \ |g_{I_{il}}(z)| \le |g_{I_{il}^u}(z)| + |g_{I_{il}^d}(z)| + |g_{I_{il} \setminus (I_{il}^u\cup I_{il}^d)}(z)|.
 \]
Hence, from \eqref{FBounded1} and for $z\in S$, 
 \[
 \ \sum _{I_{il}^u \ne \emptyset}|g_{I_{il}^u}(z)| + \sum _{I_{il}^d \ne \emptyset} |g_{I_{il}^d}(z)| \le \dfrac{C_{57}}{N^\frac{1}{4}}. 
 \]
Therefore,  
\begin{eqnarray}\label{FBounded3}
  \ \begin{aligned}
 \ I(z) := & \sum_{i\in J_1}\sum_{(i,l)\in J} |g_{I_{il}}(z)| + \sum_{i\in J_0} |g_{I_{il_i}}(z)| \\
 \ \le & \dfrac{C_{58}}{N^\frac14} + \sum_{6S_{ij}\subset (2N+1)S} (|f_{ij}(z)| + |g^0_{ij}(z)|).
 \ \end{aligned} 
 \end{eqnarray} 
  
There exists $\lambda_0\in S \cap \mathcal D(f, m, N)$ such that $\gamma (B(\lambda_0, \delta) \cap F(f, N, \lambda_0)) < \frac{1}{N^4} \delta$ (see \eqref{DFMNDef1}) and $\gamma (B(\lambda_0, \delta) \cap E_0) \le \gamma ( B(\lambda_0, \delta ) \setminus \mathcal D)< \frac{1}{N^2} \delta < \frac{3}{N} \delta _N$ (see \eqref{DFMNDef2}). Using Lemma \ref{lemmaBasic0} (3), we get
 \begin{eqnarray}\label{acSum}
 \ \sum_{6S_{ij}\subset (2N+1)S} \gamma (6S_{ij} \cap E_0) \le C_{59}\gamma(B(\lambda_0, \delta) \cap E_0) < \frac{C_{60}}{N} \delta _N. 
 \end{eqnarray}
Thus, for $z, \lambda \in F(f, N, \lambda_0)^c$, we have $|f(z) - f(\lambda)| < \frac{2}{N^3}$ and 
 \begin{eqnarray}\label{FIJEstimate}
 \ \begin{aligned}
 \ &|f_{ij}(\lambda )| \\
 \ \le &\dfrac{1}{\pi} \int_{F(f, N, \lambda_0)}\dfrac{2\|f\|}{|z-\lambda|}|\bar\partial \varphi _{ij}|d\mathcal L^2(z) \\
 \ &+ \dfrac{1}{N^3\pi} \int_{F(f, N, \lambda_0)^c}\dfrac{2}{|z-\lambda|}|\bar\partial \varphi _{ij}|d\mathcal L^2(z) \\
 \ \le & C_{61}\sqrt{\mathcal L^2(B(\lambda_0, \delta) \cap F(f, N, \lambda_0))} \|\bar\partial \varphi _{ij}\| + C_{62}\dfrac{1}{N^3} \\
 \ \le &\dfrac{C_{63}}{N^3}.
 \end{aligned}
 \end{eqnarray}
From the definition of $g^0_{ij}$ (before \eqref{GConditions}) and \eqref{MLemma2Eq}, we get
 \begin{eqnarray}\label{GIJEstimate}
 \ \begin{aligned}
 \ &\int_{S \cap \mathcal D(f, m, N)} |g^0_{ij}(\lambda )| d\mathcal L^2(\lambda) \\
 \ \le & \int_{S \cap \mathcal D(f, m, N)} \int \dfrac{1}{|z-\lambda | }d|\eta^0_{ij}|d\mathcal L^2 \\
 \ \le & C_{62} \sqrt{\mathcal L^2(S)} |\alpha (f_{ij}) | \\ 
\ \le & C_{63} \sqrt{\mathcal L^2(S)} \gamma (6S_{ij}\cap E_0).
 \ \end{aligned}
 \end{eqnarray}
Combining \eqref{FIJEstimate} with \eqref{GIJEstimate} and applying \eqref{acSum}, we get the following estimate for \eqref{FBounded3},
 \begin{eqnarray}\label{IEstimate}
 \ \begin{aligned}
 \ & \int_{S \cap \mathcal D(f, m, N)} I(z) d\mathcal L^2(z) \\
\ \le & C_{64}(2N+1)^2 \left ( \dfrac{\mathcal L^2(S)}{N^3} + \int_{B(\lambda_0, \delta) \cap F(f, N, \lambda_0)} |f_{ij}(\lambda )| d\mathcal L^2(\lambda ) \right ) \\
 \ & + C_{65}\sum_{6S_{ij}\subset (2N+1)S}\gamma (6S_{ij}\cap E_0)\sqrt{\mathcal L^2(S)} \\
\ \le &C_{66}(2N+1)^2 \left ( \dfrac{\mathcal L^2(S)}{N^3} + \mathcal L^2(B(\lambda_0, \delta) \cap F(f, N, \lambda_0)) \right ) \\
 \ &+ C_{67}\gamma ((2N+1)S\cap E_0)\sqrt{\mathcal L^2(S)} \\
 \ \le &\dfrac{C_{68}}{N}\mathcal L^2(S).
 \ \end{aligned}
 \end{eqnarray}

Define 
 \[
 \ II_i(z) := \sum_{(i,l)\in J} | h_{I_{il}}(z)|,~ II(z) := \sum_{i\in J_1} II_i(z).
 \]
For $ z\in S$ and $i\in J_1$ (notice that $N > 3k_1$), if $j\in I_{il}^1\cap (I_{il}^d \cup I_{il}^u)$ and $k\in I_{il}^2\cap (I_{il}^d \cup I_{il}^u)$, then, by (H5),
 \[
 \ |H^i_{jk} (z)| \le C_{69}\dfrac{\lambda_{ij}^k\alpha_{ij} + \lambda_{ik}^j\alpha_{ik}}{N^2\delta_N};
 \]
if $j\in I_{il}^1\setminus (I_{il}^d \cup I_{il}^u)$ and $k\in I_{il}^2 \cap (I_{il}^d \cup I_{il}^u)$, then, by (H1) and (H4),
 \[
 \ |H^i_{jk} (z)| \le C_{70} \left (\dfrac{\lambda_{ik}^j \alpha_{ik}}{N\delta_N} + \lambda_{ij}^k | \mathcal C(\eta_{ij})(z)| \right ) = C_{70}\lambda_{ij}^k \left (\dfrac{\alpha_{ij}}{N\delta_N} + |\mathcal C(\eta_{ij})(z)|\right ) ;
 \]
if $j\in I_{il}^1\cap (I_{il}^d \cup I_{il}^u)$ and $k\in I_{il}^2 \setminus (I_{il}^d \cup I_{il}^u)$, then, by (H1) and (H4),
 \[
 \ |H^i_{jk} (z)| \le C_{71}\lambda_{ik}^j \left (\dfrac{\alpha_{ik}}{N\delta_N} + |\mathcal C(\eta_{ik})(z)|\right ) ;
 \]
and if $j\in I_{il}^1\setminus (I_{il}^d \cup I_{il}^u)$ and $k\in I_{il}^2 \setminus (I_{il}^d \cup I_{il}^u)$, then, by (H1),
 \[
 \ |H^i_{jk} (z)| \le C_{72} (\lambda_{ik}^j |\mathcal C(\eta_{ik})(z)| + \lambda_{ij}^k |\mathcal C(\eta_{ij})(z)|).
 \]
Therefore,
 \[
 \ \begin{aligned}
 \ & II_i(z) \\
 \ \le & \dfrac{C_{73}}{N^2} +  C_{74} \sum_{(i,l)\in J}\sum_{j\in I_{il}^1\setminus (I_{il}^d \cup I_{il}^u)} \sum_{k \in I_{il}^2\setminus (I_{il}^d \cup I_{il}^u)} (\lambda_{ik}^j |\mathcal C(\eta_{ik})(z)| + \lambda_{ij}^k |\mathcal C(\eta_{ij})(z)|)\\
 \ &  + C_{75} \sum_{(i,l)\in J}\sum_{j\in I_{il}^1\cap (I_{il}^d \cup I_{il}^u)} \sum_{k \in I_{il}^2\setminus (I_{il}^d \cup I_{il}^u)}  \lambda_{ik}^j \left (\dfrac{\alpha_{ik}}{N\delta_N} + |\mathcal C(\eta_{ik})(z)| \right ) \\
\ & + C_{76} \sum_{(i,l)\in J}\sum_{j\in I_{il}^1\setminus (I_{il}^d \cup I_{il}^u)} \sum_{k \in I_{il}^2\cap (I_{il}^d \cup I_{il}^u)}  \lambda_{ij}^k \left (\dfrac{\alpha_{ij}}{N\delta_N} + |\mathcal C(\eta_{ij})(z)|\right ). 
 \end{aligned}
 \]
Using (H2), we get
 \[
 \ II_i(z) \le \dfrac{C_{73}}{N^2} +  C_{77} \sum_{(i,l)\in J}\sum_{j\in I_{il}\setminus (I_{il}^d \cup I_{il}^u)}\left (\dfrac{\alpha_{ij}}{N\delta_N} + |\mathcal C(\eta_{ij})(z)|\right ).
 \]
Hence,
 \begin{eqnarray}\label{IIEstimate}
 \ \begin{aligned}
 \ & \int_{S \cap \mathcal D(f, m, N)} II(z) d\mathcal L^2(z) \\
 \ \le &\dfrac{3C_{73}}{N} \mathcal L^2(S) + C_{78} \sum_{i\in J_1}\sum_{6S_{ij}\subset (2N+1)S} \left (\dfrac{\alpha_{ij}}{N\delta_N}\mathcal L^2(S) + \|\eta_{ij}\|(\mathcal L^2(S))^\frac12\right )\\
 \ \le &\dfrac{3C_{73} }{N}\mathcal L^2(S) + C_{79} \gamma((2N+1)S\cap E_0)(\mathcal L^2(S))^\frac12 \\
	 \ \le &\dfrac{C_{80} }{N}\mathcal L^2(S)
 \end{aligned}
 \end{eqnarray}
where \eqref{acSum} is used again. Combining \eqref{IEstimate} with \eqref{IIEstimate}, we prove \eqref{keyEstimate}.
\end{proof}

\begin{proof} (key steps for proving \eqref{FBounded1}) 
For $\text{dist}(z, \partial R ) \le k_1 \delta_N$, the proof in \cite{p95} shows that \eqref{FBounded1} is bounded by an absolute constant. Hence, we assume that $\text{dist}(z, \partial R ) > k_1 \delta_N$ below. Let $ l(z)\delta_N + k_1\delta_N\le \text{dist}(z, \partial R) < l(z)\delta_N + (k_1+1)\delta_N$ for $z \in R$ and $\delta_N$ small enough, where $l(z)$ is a positive integer. We need to modify the estimates in \cite{p95} under the condition $|z - c_{ij}| > l(z)\delta_N$. 

Without loss of generality, we assume that $z=0$. In this case, 
 \begin{eqnarray}\label{ZConditionEq1}
 \ |z - c_{ij}| = |0 - c_{ij}| > l(0)\delta_N.
 \end{eqnarray}  

The modification of (2.34) in \cite{p95} is 
 \begin{eqnarray}\label{Eq2.34}
 \ \begin{aligned}
 \ \sum_{j\in I} \dfrac{\delta_N^3}{|z - c_{ij}|^3} |_{z = 0} \le &C_{81} \sum_{l = \sqrt{\max(0, l(0)^2 - i^2)}}^\infty \dfrac{1}{(i^2 + l^2)^{\frac{3}{2}}} \\
 \ \le &C_{81} \int_{\sqrt{\max(0, l(0)^2 - i^2)}}^\infty \dfrac{dt}{(i^2 + t^2)^{\frac{3}{2}}} \\
 \ \le & \dfrac{C_{81}}{\max(l(0)^2,i^2)}. 
 \ \end{aligned}
 \end{eqnarray} 
Also
 \[
 \ \sum_{j\in I} \dfrac{\delta_N \alpha_{ij}}{|z - c_{ij}|^2} |_{z = 0} \le \dfrac{1}{\delta_N\max(l(0)^2, i^2)}\sum_{j\in I}  \alpha_{ij}\le \dfrac{C_{82}}{\max(l(0)^2, i^2)}.  
 \]
From \eqref{Eq2.22}, we get, for $I_i^k\in L_2$, $k=1,2,3$,
 \[
 \ |g_{I_i^k}(0)| \le C_{83} \min\left \{1, \dfrac{1}{\max(l(0)^2, i^2)} \right \}
 \] 
which implies that
 \begin{eqnarray}\label{Eq2.35}
 \ \underset{1\le k \le 3}{\sum_{I_i^k\in L_2}} |g_{I_i^k}(0)| \le \dfrac{C_{84}}{l(0)}.
  \end{eqnarray}
		
Now fix a complete group $I = I_{il}$ for $1 \le l \le l_i - 1$. For all $i \le 4k_2$, we estimate from \eqref{Eq2.22} that
 \begin{eqnarray}\label{Eq2.370}
 \ \sum_{l = 1}^{l_i-1} |\Psi_{iI}(0)| \le C_{85}\min\left \{1, \sum_k \dfrac{1}{(l(0) + k)^2} \right \} \le \dfrac{C_{85}}{l(0)}.
 \end{eqnarray}
We now fix $i$ and denote
$d_{iI} = \max_{j_1, j_2\in I}|c_{ij_1} - c_{ij_2}| + 2k_1\delta_N$.
Let $c_{iI} = c_{ij_1}$, where $j_1 = \min\{j\in I\}$. We denote by $S' = \{l: ~ d_{iI_l} \le i^{\frac{1}{4}}\delta_N,~ 1  \le l \le l_i -1 \}$ and $S'' = \{l\notin S': ~ 1  \le  l \le  l_i -1 \}$. Using maximum modulus principle for $(z-c_{iI_l})^3\Psi_{iI_l}(z)$ on $|z - c_{iI_l}| \ge d_{iI_l}$ and $l\in S'$, we have,
 \[
 \ |\Psi_{iI_l}(z)| \le \dfrac{C_{86}(\delta_N |i|^\frac{1}{4})^3}{|z - c_{iI_l}|^3}.
 \]
Since $|0 - c_{iI_l}| > |i| \delta_N > 3 \delta_N |i|^\frac{1}{4}$, using \eqref{Eq2.34}, we derive that
 \begin{eqnarray}\label{Eq2.37}
\ \begin{aligned}
 \ & \sum_{|i| > 4k_2}\sum_{l\in S'}|\Psi_{iI_l}(0)|\\
 \ \le & C_{87} \sum_{|i| \ge 4k_2} |i|^\frac{3}{4}\min\left \{1, \dfrac{1}{\max(l(0)^2, i^2)} \right \} \\
 \ \le & C_{87} \sum_{l(0) \ge |i| \ge 4k_2} \dfrac{|i|^\frac{3}{4}}{l(0)^2} + C_{59} \sum_{|i| > l(0)} \dfrac{1}{|i|^\frac{5}{4}} \\
 \ \le & \dfrac{C_{87}}{l(0)^\frac{1}{4}}.
 \end{aligned}
 \end{eqnarray}

Now for $|i| > 4k_2$ and $l\in S''$, let $b_{iI}$ be the point in $\{c_{ij}, ~ j\in I\}$ with the smallest distant to zero. by \eqref{Eq2.22} and \eqref{Eq2.25}, we see that
 \begin{eqnarray}\label{Eq2.380}
 \ \begin{aligned}
 \ & \sum_{l\in S''}|\Psi_{iI_l}(0)| \\
 \ \le &\sum_{l\in S''}  \dfrac{C_{88}\delta_N^2}{|b_{iI_l}|^2} + \sum _j \dfrac{C_{88}\delta_N^3}{|c_{ij}|^3} \\
\ \le &\sum_{q=0}^\infty \dfrac{C_{88}}{\max(l(0)^2, i^2) + q^2 |i|^\frac{1}{2}} + \sum _j \dfrac{C_{88}\delta_N^3}{|c_{ij}|^3} \\
 \ \le &\dfrac{C_{88}}{\max(l(0), |i|)|i|^\frac{1}{4}}\int_0^\infty \dfrac{dt}{1 + t^2} + \sum _j \dfrac{C_{64}\delta_N^3}{|c_{ij}|^3}.
 \ \end{aligned}
 \end{eqnarray}

Thus,
 \begin{eqnarray}\label{Eq2.38}
 \ \sum_{|i| > 4k_2} \sum_{l\in S''}|\Psi_{iI_l}(0)| \le \dfrac{C_{89}}{l(0)}\sum_{i=1}^{l(0)} \dfrac{1}{i^\frac{1}{4}} + \dfrac{C_{90}}{l(0)^\frac{1}{4}}
 + \dfrac{C_{90}}{l(0)} \le \dfrac{C_{91}}{l(0)^\frac{1}{4}}.
 \end{eqnarray}
Hence, combining \eqref{Eq2.35}, \eqref{Eq2.37}, and \eqref{Eq2.38}, we prove \eqref{FBounded1}.

\end{proof}
\bigskip

\begin{corollary}\label{MLemma2Cor} 
Let $E$ be a compact subset of $\mathbb C$ and let $E_0\subset E\cap \mathcal F$ be a Borel subset. Let $R>0$ such that $K \subset B(0,R)$ and $G = B(0,R) \setminus E$. Let $f\in H^\infty (G)$ be given with $\|f\|_{H^\infty (G)}\le 1$. Suppose that, for $\lambda _0\in \mathbb C$, $\delta > 0$, and a smooth function $\varphi$ with support in $B(\lambda _0, \delta)$, we have 
 \begin{eqnarray}\label{MLemma2CorEq}
 \ \left | \int f(z) \dfrac{\partial \varphi (z)}{\partial  \bar z} d\mathcal L^2(z) \right | \le C_{92}\delta \left \|\dfrac{\partial \varphi (z)}{\partial  \bar z} \right \|_\infty \gamma (B(\lambda _0, 2\delta) \cap E_0).
 \end{eqnarray} 
Then there exists $\hat f\in R^t(K,\mu) \cap L^\infty(\mu )$ such that $\mathcal C(\hat fg\mu)(z) = f(z )\mathcal C(g\mu)(z), ~ z\in G, ~ \gamma-a.a.$ for each $g\perp R^t(K,\mu)$ and $\|\hat f\|_{L^\infty (\mu )} \le C_{93}$. 
\end{corollary}
\smallskip

\begin{proof}
If $\mathcal D = G$, then $H^\infty (G) \subset GC(\mathcal D)$. As in the proof of Theorem \ref{MLemma2}, the subsequence $\{f_N\}$ of $\{f_{\delta_N}\}$ uniformly converges  on any compact subset of $G$. The corollary follows.
\end{proof}
\smallskip

If $E_0 = E$ in Corollary \ref{MLemma2Cor}, then the assumption \eqref{MLemma2CorEq} is no longer necessary. In fact, For $f\in H^\infty (G)$ with $f(z) = 0$ for $z\in G^c$, $\lambda _0\in \mathbb C$, and $\delta > 0$, if $\varphi$ is a smooth function supported in $B(\lambda _0, \delta)$, then $T_\varphi f$ is bounded analytic off the compact subset $\text{spt}(\varphi) \cap E$. Therefore,  
 \[
 \ \begin{aligned}
 \ & |(T_\varphi f)'(\infty)| \\
 \ = & \left | \int f(z) \dfrac{\partial \varphi (z)}{\partial  \bar z} d\mathcal L^2(z) \right |  \\
 \ \le & \|T_\varphi f \|_\infty \gamma(\text{spt}(\varphi) \cap E) \\
 \ \le & 8\|f\|\delta \left \|\dfrac{\partial \varphi (z)}{\partial  \bar z} \right \|_\infty \gamma (B(\lambda _0, \delta) \cap E).
 \ \end{aligned}
 \]
So the assumption \eqref{MLemma2CorEq} is satisfied. Corollary \ref{MLemma2Cor} can be stated as the following. 
\smallskip

\begin{corollary}\label{MLemma2Cor2} 
Let $E\subset \mathcal F$ be a compact subset. Let $R>0$ such that $K \subset B(0,R)$ and $G = B(0,R) \setminus E$. Then for $f\in H^\infty (G)$ with $\|f\|_{H^\infty (G)}\le 1$, there exists $\hat f\in R^t(K,\mu) \cap L^\infty(\mu )$ such that $\mathcal C(\hat fg\mu)(z) = f(z )\mathcal C(g\mu)(z), ~ z\in G, ~ \gamma-a.a.$ for each $g\perp R^t(K,\mu)$ and $\|\hat f\|_{L^\infty (\mu )} \le C_{93}$. 
\end{corollary}
\bigskip

\chapter{Irreducible operator $S_\mu$ on $R^t(K,\mu)$}
\bigskip

In this chapter, we assume that $1\le t <\infty$, $\mu$ is a finite positive Borel measure supported on a compact subset $K\subset \mathbb C$, and $K=\sigma(S_\mu)$. \smallskip

\begin{definition}\label{gammaCDefinition}
Let $E\subset \mathbb C$ be a non-trivial Borel subset ($E \ne \emptyset$). The set $E$ is called $\gamma$-connected if  the following property holds: For any two disjoint open subsets $G_1$ and $G_2$, $E \subset G_1\cup G_2, ~\gamma-a.a.$ implies either $E\subset G_1, ~\gamma-a.a.$ or $E\subset G_2, ~\gamma-a.a.$. 
\end{definition} 
\smallskip

For example, if $K$ is a string of beads set as in Theorem \ref{SOBTheorem} with $K \cap \mathbb R \supset\mathcal R_B \ne \emptyset$, the removable set $\mathcal R$ is $\gamma$-connected though $\text{abpe}(R^t(K,\mu))$ consists of two non-trivial disjoint connected open subsets $G_U$ and $G_L$.

The purpose of this chapter is to prove the following theorem. 

\begin{theorem}\label{IrreducibilityTheorem} 
If $S_\mu$ on $R^t(K, \mu)$ is pure, then  
$S_\mu$ is irreducible, that is, $R^t(K, \mu)$ contains no non-trivial characteristic functions, if and only if the removable set $\mathcal R$ is $\gamma$-connected.  
\end{theorem}
\smallskip

We need several lemmas to complete the proof of Theorem \ref{IrreducibilityTheorem}.

\begin{lemma}\label{Lemma1} 
There exists a Borel partition $\{\Delta_n\}_{n\ge 0}$ of $\text{spt}(\mu )$ such that
 \[
 \ R^t(K,\mu) = L^t(\mu |_{\Delta_0})\oplus \bigoplus_{n=1}^\infty R^t(K_n, \mu |_{\Delta_n}),
 \]
where $K_n = \sigma(S_{\mu |_{\Delta_n}})$ and for $n \ge 1$, $S_{\mu |_{\Delta_n}}$ on $R^t(K_n, \mu |_{\Delta_n})$ is irreducible. 
\end{lemma}
\smallskip

\begin{proof} This is an application of Theorem 1.6 on page 279 of \cite{C91} (also see \cite{co77}) since $R^t(K,\mu)\cap L^\infty(\mu)$ is a weak$^*$ closed subalgebra of $L^\infty(\mu)$ and $R^t(K_0,\mu |_{\Delta_0})\cap L^\infty (\mu |_{\Delta_0})$ is pseudosymmetric subalgebra of $L^\infty (\mu |_{\Delta_0})$. It remains to prove $R^t(K_0,\mu |_{\Delta_0}) =  L^t (\mu |_{\Delta_0})$. Suppose that it is not true, we must have
 \[
 \ R^t(K_0,\mu |_{\Delta_0}) = L^t(\mu |_{\Delta_{00}})\oplus R^t(K_{01}, \mu |_{\Delta_{01}}),
 \]
where $S|_{\mu |_{\Delta_{01}}}$ on $R^t(K_{01}, \mu |_{\Delta_{01}})$ is pure. If $\mathcal R_{01}$ is its removable set, then $\mathcal L^2(\mathcal R_{01}) > 0$. From Corollary \ref{acZero}, there exists $g\perp R^t(K_{01}, \mu |_{\Delta_{01}})$ and $\lambda_0 \in \mathcal R_{01}$ such that 
 \[
 \ \int \dfrac{1}{|z-\lambda_0|}|g(z)|d\mu |_{\Delta_{01}} < \infty
 \]
and $\mathcal C(g\mu |_{\Delta_{01}})(\lambda_0) \ne 0$.
From Lemma \ref{lemmaBasic9}, we conclude that 
 \[
 \ \rho(f)(\lambda_0) = \dfrac{\mathcal C(fg\mu |_{\Delta_{01}})(\lambda_0)}{\mathcal C(g\mu |_{\Delta_{01}})(\lambda_0)}
 \]
is the $\gamma$-limit of $\rho(f)$ at $\lambda_0$ for $f \in R^t(K_{01},\mu |_{\Delta_{01}}) \cap  L^\infty (\mu |_{\Delta_{01}})$.
We claim:
 \begin{eqnarray}\label{Lemma1Eq1}
 \ \begin{aligned}
 \ & \rho(f_1f_2)(\lambda_0) = \rho(f_1)(\lambda_0)\rho(f_2)(\lambda_0),\\
 \ &f_1,f_2 \in R^t(K_{01},\mu |_{\Delta_{01}}) \cap  L^\infty (\mu |_{\Delta_{01}}).
 \ \end{aligned} 
 \end{eqnarray}
In fact, by Proposition \ref{Rhoprop}, $\rho(f_1f_2)(z) = \rho(f_1)(z)\rho(f_2)(z),~\gamma-a.a.$. Therefore, the $\gamma$-limits of $\rho(f_1f_2)(z)$ and $\rho(f_1)(z)\rho(f_2)(z)$ at $\lambda_0$ coincide, which implies \eqref{Lemma1Eq1}.
   
Set
 \[
 \ \mathcal B = \{B\subset \Delta_{01}:~ \chi_B \in R^t(K_{01}, \mu |_{\Delta_{01}})  \cap  L^\infty (\mu |_{\Delta_{01}}),~\rho(\chi_B)(\lambda_0) = 1\}.
 \]
Then $1\in \mathcal B \ne \emptyset$. For $B_1, B_2\in \mathcal B$, by \eqref{Lemma1Eq1}, we see that $\rho(\chi_{B_1\cap B_2})(\lambda_0) = \rho(\chi_{B_1})(\lambda_0)\rho(\chi_{B_2})(\lambda_0) = 1$. We find $\Delta_{01}\supset B_n\supset B_{n+1}$ such that 
 \[
 \ \mu(B_n) \rightarrow b = \inf_{B\in \mathcal B}\mu(B).
 \]
Clearly $B =\cap B_n\in \mathcal B$ and $ b = \mu(B) > 0$.  From Proposition 1.4 on page 278 of \cite{C91}, we get that $\chi_B$ is not a minimal characteristic function. Hence, there exists $B_0\subset B$ with $\mu(B_0) > 0$ and $\mu(B\setminus B_0) > 0$ such that $\chi_{B_0}, ~\chi_{B\setminus B_0}  \in R^t(K_{01},\mu |_{\Delta_{01}})$. Since 
 \[
 \ \rho(\chi_{B})(\lambda_0) = \rho(\chi_{B_0})(\lambda_0) + \rho(\chi_{B\setminus B_0})(\lambda_0)
 \]
 and 
 \[ 
 \ \rho(\chi_{B_0})(\lambda_0) \rho(\chi_{B\setminus B_0})(\lambda_0) = \rho(\chi_{B_0}\chi_{B\setminus B_0})(\lambda_0) = 0,
 \]
we see that $B\setminus B_0\in \mathcal B$ or $B_0\in \mathcal B$, which contradicts to the definition of $b$. Therefore, $R^t(K_0,\mu |_{\Delta_0}) =  L^t (\mu |_{\Delta_0})$.     
\end{proof}
\smallskip

\begin{lemma}\label{Lemma2}
Let $S_\mu$ on $R^t(K,\mu)$ be pure. Suppose that there exists a Borel partition $\{\Delta_i\}_{i\ge 1}$ of $\text{spt}(\mu)$ such that
\[
 \ R^t(K, \mu ) = \bigoplus _{i = 1}^\infty R^t(K_i, \mu |_{\Delta_i}),
 \]
where $K_i = \sigma(S_{\mu |_{\Delta_i}})$, $\mathcal F_i$ is the non-removable boundary of $R^t(K_i, \mu |_{\Delta_i})$, and $\mathcal R_i = K_i \setminus \mathcal F_i$ is the removable set of $R^t(K_i, \mu |_{\Delta_i})$. Then the following hold:

(1)
 \[
 \ \mathcal F \approx \bigcap_{i\ge 1}\mathcal F_i,~ \gamma-a.a..
 \]

(2) $\rho_i(f)(z) = \rho(f)(z), ~\gamma |_{\mathcal R_i}-a.a.$ for $f\in R^t(K_i, \mu |_{\Delta_i})$, where $\rho_i$ is the $\rho$ map for $R^t(K_i, \mu |_{\Delta_i})$, and 
 \[
 \ \mathcal R_i \cap \mathcal R_j \approx \emptyset,~ \gamma-a.a.,\text{ for } i \ne j.
 \]

(3)
 \[
 \ \mathcal R \approx \bigcup_{i\ge 1}\mathcal R_i,~ \gamma-a.a..
 \]   
\end{lemma}
\smallskip

\begin{proof}
(1): Let $\{g_n\}_{n=1}^\infty$ be a dense set of $R^t(K,\mu)^\perp$. Then it is easy to verify that 
 \[
 \ \{h_n\}_{n=1}^\infty = \{ \chi_{\Delta_1}g_1, \chi_{\Delta_1}g_2,\chi_{\Delta_2}g_1,\chi_{\Delta_1}g_3, \chi_{\Delta_2}g_2,\chi_{\Delta_3}g_1,...,\}
 \]
is a dense set of $R^t(K,\mu)^\perp$. For $N$, let $i_N$ be the number of terms, from first $N$ terms, that contain $\chi_{\Delta_i}$. Then it is easy to check
 \[
 \ \mathcal E (h_j\mu, 1 \le j \le N) \subset \mathcal E (\chi_{\Delta_i}g_j\mu, 1 \le j \le i_N). 
 \]
Hence, applying Theorem \ref{FCharacterization}, we get $\mathcal F \subset \mathcal F_i$. Thus,
$\mathcal F \subset \cap_{i\ge 1}\mathcal F_i$. 

Let $\Gamma_n$, $h$ for $\mu$, and $h_i = \chi_{\Delta_i}h$ for $\mu |_{\Delta_i}$ be as in Lemma \ref{GammaExist}. It is easy to verify from Lemma \ref{GammaExist} that 
\[
\ \mathcal {ND}(\mu |_{\Delta_i}) (\approx \mathcal {N}(h_i)) \cap \mathcal {ND}(\mu |_{\Delta_m}) (\approx \mathcal {N}(h_m)) \approx \emptyset, ~\gamma-a.a.
\]
 for $i\ne m$, $\mathcal {ND}(\mu) \approx \cup_i \mathcal {ND}(\mu |_{\Delta_i}), ~\gamma-a.a.$, and $\mathcal {ZD}(\mu) \approx \cap_i \mathcal {ZD}(\mu |_{\Delta_i}), ~\gamma-a.a.$. Therefore, by Theorem \ref{FRProperties} \eqref{FCEq7}, we get
 \[
 \ \mathcal {ZD}(\mu) \cap \bigcap_{i\ge 1}\mathcal F_i \subset \mathcal F.
 \]
For $\gamma$ almost all $\lambda_0\in \mathcal {ND}(\mu) \cap \cap_{i\ge 1}\mathcal F_i$, there exist integers $i_0$ and $m$ such that 
\[
 \ \lambda_0\in \Gamma_m \cap \mathcal {ND}(\mu |_{\Delta_{i_0}}) \cap \bigcap_{i \ne i_0}\mathcal {ZD}(\mu |_{\Delta_i})
 \]
 and 
 \[
 \ h_i(\lambda_0) = \mathcal C(\chi_{\Delta_i}g_j\mu)(\lambda_0) = 0
 \]
for $i \ne i_0$ and $j\ge 1$, which implies 
 \[
 \ v^+ (\chi_{\Delta_i}g_j\mu, \Gamma_m, \lambda_0) = v^- (\chi_{\Delta_i}g_j\mu, \Gamma_m, \lambda_0) = 0
 \]
  for $i \ne i_0$ and $j\ge 1$. On the other hand, $\lambda_0\in \Gamma_m \cap ((\mathcal F_{i_0})_+\cup (\mathcal F_{i_0})_-)$. Hence, $\lambda_0\in \mathcal F$.

(3) follows from (1). 

For (2), if $f = \sum_{n=1}^\infty \frac{1}{2^n} \chi _{\Delta_n}$, then $f \in R^t(K, \mu ) \cap L^\infty(\mu)$. For $g \perp R^t(K, \mu )$, we have $g\chi _{\Delta_n} \perp R^t(K, \mu )$. Thus, from equation \eqref{BasicEq22},
 \[
 \ \rho (f)(z) \mathcal C(g\chi _{\Delta_n}\mu)(z) = \frac{1}{2^n} \mathcal C(g\chi _{\Delta_n}\mu)(z), ~ \gamma-a.a..
 \] 
Therefore, $\rho (f)(z) = \frac{1}{2^n}, ~  \gamma|_{\mathcal R_n}-a.a.$. (2) is implied.
\end{proof}
\smallskip

\begin{lemma}\label{Lemma4} 
Suppose that $S_\mu$ on $R^t(K,\mu)$ is pure. Let $\Delta_0$ be a Borel set, $\chi_{\Delta_0} \in R^t(K, \mu)$, and $K_0 = \sigma(S_{\mu |_{\Delta_0}})$. Let $\mathbb C \setminus K_0 = \cup_{n=1}^\infty U_n$, where $U_n$ is a connected component.  If $S_{\mu |_{\Delta_0}}$ on $R^t(K_0, \mu |_{\Delta_0})$ is irreducible, then the following properties hold: 

(1) $\partial U_n \subset \mathcal F$.

(2) If $K_n = \overline{K\cap U_n}$, then there exists a Borel partition $\{\Delta_n\}_{n\ge 0}$ of $\text{spt}(\mu )$ with $\Delta_n\subset K_n$ such that
 \[
 \ R^t(K,\mu) = R^t(K_0, \mu |_{\Delta_0})\oplus \bigoplus_{n=1}^\infty R^t(K_n, \mu |_{\Delta_n}).
 \]
\end{lemma}
\smallskip

\begin{proof}
(1): From Lemma \ref{Lemma1}, we get
 \[
 \ R^t(K,\mu) = R^t(K_0,\mu |_{\Delta_0}) \oplus R^t(K',\mu |_{\Delta'}).  
 \]
Let $\mathcal F_0$ and $\mathcal F'$ be the non-removable boundaries of $R^t(K_0,\mu |_{\Delta_0})$ and $R^t(K',\mu |_{\Delta'})$, respectively. Set $\mathcal R_0 = K_0 \setminus \mathcal F_0$ and $\mathcal R' = K' \setminus \mathcal F'$. It is clear that $\partial U_n\subset \mathcal F_0$. We claim that $\partial U_n\subset \mathcal F'$. In fact, from Theorem \ref{DensityCorollary}, if there exists $\lambda \in \partial U_n \cap \mathcal R'$, then
 \[
 \ \lim_{\delta\rightarrow 0 } \dfrac{\gamma(B(\lambda, \delta) \setminus \mathcal R')}{\delta} = 0.
 \]
Using Lemma \ref{Lemma2} (2), we see that $\mathcal R_0\cap \mathcal R' = \emptyset, ~ \gamma-a.a.$, so $\mathcal R'\subset \mathcal F_0, ~ \gamma-a.a.$. Hence,
 \[
 \ \lim_{\delta\rightarrow 0 } \dfrac{\gamma(B(\lambda, \delta) \setminus \mathcal F_0)}{\delta} = 0,
 \]
which contradicts to Theorem \ref{Lemma3} as $S_{\mu | _{\Delta _0}}$ on $R^t(K_0,\mu |_{\Delta_0})$ is irreducible. Thus, by Lemma \ref{Lemma2} (1),
 \[
 \ \partial U_n \subset \mathcal F_0 \cap \mathcal F' \approx \mathcal F, ~\gamma-a.a.. 
 \]

(2): By Corollary \ref{MLemma2Cor2}, for $f_n = \chi_{U_n}$, there exists $\hat f_n \in R^t(K, \mu) \cap L^\infty (\mu)$ such that 
 \[
 \ \mathcal C(\hat f_n g\mu )(z) = f_n (z)\mathcal C( g\mu )(z),~ z \in \mathbb C \setminus \partial U_n, \gamma-a.a. 
 \]
for $g\perp R^t(K, \mu)$. Hence, there exists $\Delta _n$ ($U_n \subset \Delta _n \subset \overline{U_n}$) such that, by Lemma \ref{characterizationF}, $\hat f_n = \chi_{\Delta_n}$. Since  
\[
 \ \mathcal C((\hat f_n + \hat f_m) g\mu )(z) = \chi_{U_n\cup U_m} \mathcal C( g\mu )(z),~ z \in \mathbb C \setminus (\partial U_n\cup \partial U_m) , \gamma-a.a., 
 \]
there exists a Borel set $\Delta$ such that $\hat f_n + \hat f_m = \chi _{\Delta}$ (Remark \ref{characterizationFRemark}),
which implies $\Delta _n \cap \Delta _m = \emptyset$ for $n \ne m$. Therefore,
 \[
 \ R^t(K',\mu |_{\Delta'}) = R^t(K_{00},\mu |_{\Delta_{00}}) \oplus \bigoplus _{n=1}^\infty R^t(K_n,\mu |_{\Delta_n}) 
 \]
where $K_{00} = \sigma(S_{\mu |_{\Delta_{00}}})$ and $K_n = \sigma(S_{\mu |_{\Delta_n}})$. Clearly, $K_{00} \subset K_0$. Let $\mathcal F_{00}$ be the non-removable boundary of $R^t(K_{00},\mu |_{\Delta_{00}})$. Set $\mathcal R_{00} = \mathbb C \setminus \mathcal F_{00}$. If $\mathcal R_{00} \ne \emptyset, ~ \gamma-a.a.$, then, using Lemma \ref{Lemma2} (2), we see that $\mathcal R_{00} \subset \mathcal F_0, ~ \gamma-a.a.$. Similar to the proof of (1), we get a contradiction  to Theorem \ref{Lemma3} as $S_{\mu |_{\Delta_{0}}}$ on $R^t(K_{0},\mu |_{\Delta_{0}})$ is irreducible. Thus, $R^t(K_{00},\mu |_{\Delta_{00}}) = 0$.      
\end{proof}

\smallskip

With above lemmas, we are able to prove Theorem \ref{IrreducibilityTheorem}.

\begin{proof} 
(Theorem \ref{IrreducibilityTheorem}) $\Rightarrow$: Suppose that $\mathcal R$ is not $\gamma$-connected. Then there exist open bounded subsets $G_1$ and $G_2$ with $G_1\cap G_2 = \emptyset$, $\mathcal R \subset G_1\cup G_2$, $\mathcal R \cap G_1 \ne \emptyset$, and $\mathcal R \cap G_2 \ne \emptyset$. We may assume $G_1$ is connected.  
 Then from Corollary \ref{MLemma2Cor2}, for $f_1 = \chi_{G_1}$, there exists $\hat f_1 \in R^t(K,\mu)\cap L^\infty (\mu)$ such that 
 \[
 \ \mathcal C(\hat f_1 g \mu)(z) = f_1(z) \mathcal C( g \mu)(z),~ z\in \mathbb C \setminus \partial G_1,~\gamma -a.a.
 \]
 for $g\perp R^t(K,\mu)$. This implies, by Lemma \ref{characterizationF}, that $\hat f_1 = \chi_{\Delta_1}$ and $G_1\cap \mathcal R \approx \mathcal R_{\Delta_1}$ as $\partial G_1\cap \mathcal R \approx \emptyset$ and $\rho_1(\hat f_1)(z) = \chi_{G_1}, ~ \gamma |_{\mathcal R_{\Delta_1}}-a.a.$, where $\mathcal R_{\Delta_1}$ is the removable set for $R^t(K, \mu |_{\Delta_1})$. Similarly, $G_2\cap \mathcal R \approx \mathcal R_{\Delta_1^c}$, where $\mathcal R_{\Delta_1^c}$ is the removable set for $R^t(K, \mu |_{\Delta_1^c})$.  This is a contradiction as $S_\mu$ is irreducible.      

$\Leftarrow$: If $R^t(K, \mu)$ is not irreducible, then from Lemma \ref{Lemma1}, there exists a partition $\{\Delta _1, \Delta _2\}$ of $\text{spt}(\mu)$ such that
 \[
 \ R^t(K,\mu) = R^t(K_1, \mu |_{\Delta_1})\oplus R^t(K_2, \mu |_{\Delta_2}),
 \]
where $K_i = \sigma (S_{\mu  | _{\Delta _i}})$ for $i = 1,2$ and $S_{\mu  | _{\Delta _1}}$ on $R^t(K_1, \mu |_{\Delta_1})$ is irreducible. Using Lemma \ref{Lemma4}, there exists $U_n$ such that $U_n\cap K_2 \ne \emptyset$ ($K_1\ne K$). Let $K\subset B(0, R)$, $G_1 = U_n$, and $G_2 = B(0, R) \setminus \overline{G_1}$. Then, from Lemma \ref{Lemma4}, we see that 
 \[
 \ \mathcal R \subset G_1\cup G_2,~ \mathcal R \cap G_1 \ne \emptyset, \text{ and } \mathcal R \cap G_2 \ne \emptyset.   
 \]
This is a contradiction since $\mathcal R$ is $\gamma$-connected.    
\end{proof}

\bigskip

\chapter{Decomposition of $R^t(K,\mu)$ and the algebra $R^t(K,\mu)\cap L^\infty (\mu)$}
\bigskip

In this chapter, combining the results from previous chapters, we prove a decomposition theorem (Theorem \ref{DecompTheorem}) of $R^t(K, \mu)$ for an arbitrary compact subset $K$ and a finite positive measure $\mu$ supported on $K$, which extends the central results of $P^t(\mu)$ (Thomson's theorem, see Theorem I). We also discuss some applications of our main results.

We assume that $1\le t <\infty$, $\mu$ is a finite positive Borel measure supported on a compact subset $K\subset \mathbb C$, and $K=\sigma(S_\mu)$.

\bigskip

\section{Structure and decomposition of $R^t(K,\mu)$}
\bigskip

In this section, using Theorem \ref{MLemma1} and Theorem \ref{MLemma2}, we prove the following theorem. Recall that $H^\infty (\mathcal R) (=H^\infty_{\mathcal R, \mathcal F}(\mathcal L^2_{\mathcal R }))$ is also defined as in Definition \ref{hSpace}.

\begin{theorem}\label{algebraEq} 
If $S_\mu$ on $R^t(K, \mu)$ is pure, then $\rho(f)\in H^\infty (\mathcal R)$ for all $f\in R^t(K, \mu)\cap L^\infty(\mu )$ and  
\[
 \ \rho: ~ R^t(K, \mu)\cap L^\infty(\mu ) \rightarrow  H^\infty (\mathcal R)
 \]
is an isometric isomorphism and a weak$^*$ homeomorphism.
\end{theorem}

\begin{proof}
For $f\in R^t(K, \mu)\cap L^\infty(\mu )$, $\rho(f)$ is $\gamma$-continuous $\mathcal L^2|_{\mathcal R }-a.a.$ on $\mathcal R$ by Theorem \ref{MTheorem2}. From Theorem \ref{DensityCorollary}, $\mathcal R$ satisfies \eqref{densityAssumption}. Set $F = \rho(f)$. Then $F\in GC(\mathcal R)$. To apply Theorem \ref{MLemma2}, we set $\mathcal D = \mathcal R$ and $E_0 = \mathcal F$. By Theorem \ref{MLemma1}, we see that $F$ satisfies the condition \eqref{MLemma2Eq}. From Theorem \ref{MLemma2}, we conclude that $F\in H^\infty (\mathcal R)$. Therefore, the image of $\rho$ is in $H^\infty (\mathcal R)$.

From Theorem \ref{MLemma2}, there exists $\hat F \in R^t(K, \mu)\cap L^\infty(\mu )$ such that 
 \[
 \ \mathcal C(\hat Fg\mu)(z) = F(z)\mathcal C(g\mu)(z), ~ \mathcal L^2_{\mathcal R }-a.a.
 \]
 for $g\perp R^t(K, \mu)$ and $\|\hat F\|_{L^\infty(\mu )} \le C_{94}  \|F\|_{L^\infty(\mathcal L^2_{\mathcal R })}$.  This implies, from \eqref{BasicEq22}, that $\mathcal C(\hat Fg\mu)(z) = \mathcal C(fg\mu)(z), ~ \mathcal L^2_{\mathcal R }-a.a.$ for $g\perp R^t(K, \mu)$. With Corollary \ref{acZero} and Theorem \ref{FCForR}, we get 
 \[
 \ \mathcal C(\hat Fg\mu)(z) = \mathcal C(fg\mu)(z), ~ \mathcal L^2-a.a., 
 \]
which implies $\hat F = f$ because $S_\mu$ on $R^t(K, \mu)$ is pure. Hence, 
 \[
 \ \|f\|_{L^\infty(\mu )} \le C_{94} \|F\|_{L^\infty(\mathcal L^2_{\mathcal R })} = C_{94} \|\rho(f)\|_{L^\infty(\mathcal L^2_{\mathcal R })}.
 \]
Using Proposition \ref{Rhoprop} (1),
 \[
 \ \|f^n\|_{L^\infty(\mu )} \le C_{94} \|(\rho(f))^n\|_{L^\infty(\mathcal L^2_{\mathcal R })}. 
 \]
Thus, $\|f\|_{L^\infty(\mu )} \le \|\rho (f)\|_{L^\infty(\mathcal L^2_{\mathcal R })}$. So, by Proposition \ref{Rhoprop} (2), we have proved that
 \[
 \ \|\rho(f)\|_{L^\infty(\mathcal L^2_{\mathcal R})} = \|f\|_{L^\infty(\mu)}, ~ f\in R^t(K, \mu)\cap L^\infty(\mu ).
 \]
Clearly the map $\rho$ is injective. For $f\in H^\infty (\mathcal R)$ with $\|f\|_{H^\infty (\mathcal R)} \le 1$, there exists a sequence of $\{f_n\}$ such that $f_n$ is bounded analytic off a compact subset $E_n\subset \mathcal F$, $\|f_n\|_{L^\infty(\mathcal L^2_{\mathcal R})} \le C_{95}$, and $f_n \rightarrow f$ in $L^\infty(\mathcal L^2_{\mathcal R})$ weak$^*$ topology. Using Theorem \ref{MLemma2}, since $f_n$ satisfies the condition \eqref{MLemma2Eq}, we find a function $\hat f_n \in R^t(K, \mu)\cap L^\infty(\mu )$ satisfying
 \[
 \ \mathcal C(\hat f_n g \mu)(z) = f_n(z) \mathcal C( g \mu)(z),~ \mathcal L^2_{\mathcal R}-a.a., ~ \|\hat f_n\|_{L^\infty(\mu)} \le C_{95} 
 \]
for $g\perp R^t(K, \mu)$ and $\rho(\hat f_n) = f_n$. Choose a subsequence $\{\hat {f}_{n_k}\}$ so that $\hat f_{n_k}\rightarrow \hat f$ in $L^\infty(\mu )$ weak$^*$ topology. Then $\mathcal C(\hat f_{n_k} g \mu)(\lambda) \rightarrow \mathcal C(\hat f g \mu)(\lambda)$ for $\lambda$ with $\int \frac{1}{|z-\lambda |}|g(z)| d \mu(z) < \infty$. Hence, 
 \[
 \ \mathcal C(\hat f g \mu)(z) = f(z) \mathcal C( g \mu)(z),~ \mathcal L^2_{\mathcal R}-a.a. 
 \]
for $g\perp R^t(K, \mu)$, which implies $\rho(\hat f) = f$ and  $\rho$ is surjective. Therefore, $\rho$ is bijective      
 isomorphism between two Banach algebras $R^t(K, \mu)\cap L^\infty(\mu )$ and $H^\infty (\mathcal R)$. Clearly $\rho$ is also a weak$^*$ sequentially continuous, an application of Krein-Smulian Theorem shows that $\rho$ is a weak$^*$ homeomorphism.
\end{proof}
\smallskip

\begin{example}\label{algEqExample} 
Let $K$ be as in Example \ref{FCExample} and let $\mu$ be the sum of the arclength measures of the unit circle and all small circles. In this case, $\mathcal F = K^c \cup \mathbb T \cup \cup_{n=1}^\infty \partial B(\lambda_n,\delta_n)$ and $\mathcal R = K\setminus \mathcal F$. Clearly, $H^\infty (\mathcal R) = R^\infty (K, \mathcal L^2_K)$, the weak$^*$ closure of $Rat(K)$ in $L^\infty (\mathcal L^2_K)$. Theorem \ref{algebraEq} shows that $\rho$ 
is an isometric isomorphism and a weak$^*$ homeomorphism from $R^t(K, \mu)\cap L^\infty(\mu )$ to $R^\infty (K, \mathcal L^2_K)$, where $K$ may not have interior such as a Swiss cheese set.   
\end{example}

Combining Theorem \ref{IrreducibilityTheorem} with Theorem \ref{algebraEq}, we get the following decomposition theorem for $R^t(K, \mu)$. 

\begin{theorem}\label{DecompTheorem} 
There exists a Borel partition $\{\Delta_n\}_{n\ge 0}$ of $\text{spt}(\mu )$ and compact subsets $\{K_n\}_{n=1}^\infty$ such that $\Delta_n \subset K_n$ for $n \ge 1$,
 \[
 \ R^t(K,\mu) = L^t(\mu |_{\Delta_0})\oplus \bigoplus_{n=1}^\infty R^t(K_n, \mu |_{\Delta_n}),
 \]
and the following statements are true: 

(1) If $n \ge 1$, then $R^t(K_n, \mu |_{\Delta_n})$ contains no non-trivial characteristic functions. 

(2) If $n \ge 1$ and $\mathcal R_n$ is the removable set for $R^t(K_n, \mu |_{\Delta_n})$, then $\mathcal R_n$ is $\gamma$-connected. 

(3) If $n \ge 1$, then $K_n = \overline{\mathcal R_n}$.

(4) If $n \ge 1$, $m \ge 1$, $n\ne m$,  and $\mathcal F_m$ is the non-removable boundary for $R^t(K_m, \mu |_{\Delta_m})$, then $K_n \subset \mathcal F_m, ~\gamma-a.a$.

(5) If $n \ge 1$, $m \ge 1$, and $n\ne m$, then $K_n \cap K_m \subset \mathcal F, ~\gamma-a.a$.

(6) If $n \ge 1$, then the map $\rho_n$ is an isometric isomorphism and a weak$^*$ homeomorphism from $R^t(K_n, \mu |_{\Delta_n}) \cap L^\infty (\mu |_{\Delta_n})$ onto $H^\infty(\mathcal R_n )$. 
\end{theorem}

\begin{proof} 
(1) follows from Lemma \ref{Lemma1}. (2) follows from Theorem \ref{IrreducibilityTheorem}. (3) follows from Proposition \ref{NFSetIsBig}. 

For (4), if $\mathcal R_m \cap K_n \ne \emptyset, ~ \gamma-a.a.$
for $n\ne m$ and $n,m \ge 1$, then from Theorem \ref{DensityCorollary}, we can find $\lambda \in \mathcal R_m \cap K_n$ such that
 \[
 \ \lim_{\delta \rightarrow 0} \dfrac{\gamma(B(\lambda, \delta)\setminus \mathcal R_m)}{\delta} = 0.
 \] 
From Lemma \ref{Lemma2} (2), we see that $\mathcal R_m \subset \mathcal F_n$. Hence,
 \[
 \ \lim_{\delta \rightarrow 0} \dfrac{\gamma(B(\lambda, \delta)\setminus \mathcal F_n)}{\delta} = 0. 
 \]
This contradicts to Theorem \ref{Lemma3} and (1). Therefore, $K_n \subset \mathcal F_m$ for $n\ne m$ and $n,m \ge 1$. 

(5): (4) implies that 
 \[
 \ K_n\cap K_m \subset \mathcal F_k
 \]
for $n\ne m$ and $n,m \ge 1$. Now (5) follows from Lemma \ref{Lemma2} (1). 

Theorem \ref{algebraEq} implies (6).
\end{proof}
\smallskip

Let $U$ be a connected component of $\text{abpe}(R^t(K, \mu ))$. By Theorem \ref{ABPETheorem}, there exists $n \ge 1$ such that $U\cap \mathcal R_n \ne \emptyset, ~\gamma-a.a.$. If $\partial U \subset \mathcal F$, then $ \mathcal R_n \subset  U \cup {\overline U}^c, ~\gamma-a.a.$. Hence, $ \mathcal R_n \subset  U , ~\gamma-a.a.$ since $\mathcal R_n$ is $\gamma$-connected by Theorem \ref{DecompTheorem} (2). On the other hand, if $r_m \in \text{Rat}(K) \rightarrow \chi_{\Delta_n}$, then $r_m$ uniformly converges to an analytic function on any compact subsets of $U$, which ensures $\mathcal R_m \cap U \approx \emptyset, ~\gamma-a.a.$ for $n\ne m$. So, by Theorem \ref{ABPETheorem} again, $\mathcal R_n = U, ~\gamma-a.a.$. Therefore, we have the following corollary.     
\smallskip

\begin{corollary}\label{DecompCorollary1}
Suppose that $S_\mu$ on $R^t(K, \mu)$ is pure. Let $\text{abpe}(R^t(K, \mu )) = \cup_{i\ge 1}U_i$, where $U_i$ is a connected component. If $\cup _{i\ge 1} \partial U_i \subset \mathcal F,~\gamma-a.a.$, then there is a Borel partition $\{\Delta_i\}_{i=0}^\infty$ of $\mbox{spt}(\mu)$ such that
 \[
 \ R^t(K, \mu ) = R^t(K, \mu|_{\Delta_0}) \oplus \bigoplus _{i = 1}^\infty R^t(K, \mu |_{\Delta_i})
 \]
satisfying:

(a) $S_{\mu |_{\Delta_i}}$ on $R^t(K, \mu |_{\Delta_i})$ is irreducible for $i \ge 1$;

(b) $S_{\mu |_{\Delta_0}}$ on $R^t(K, \mu |_{\Delta_0})$ is pure and $\mbox{abpe}( R^t(K, \mu |_{\Delta_0})) = \emptyset$; 

(c) $\mbox{abpe}( R^t(K, \mu |_{\Delta_i})) = U_i$, $\Delta_i\subset K_i := \overline{U_i}$, $R^t(K, \mu |_{\Delta_i}) = R^t(K_i, \mu |_{\Delta_i})$ for $i \ge 1 ;$

(d) the evaluation map $\rho_i(f) = \rho(f) | _{U_i}$ is an isometric isomorphism and a weak$^*$ homeomorphism from $R^t(K_i, \mu |_{\Delta_i}) \cap L^\infty (\mu |_{\Delta_i})$ onto $H^\infty(U_i)$ for $i \ge 1 .$  
\end{corollary}
\smallskip

The following corollary extends Theorem I and Theorem II that require a stronger condition. That is, both Theorem I and Theorem II assume there exist constants $A>0$ and $\delta_0 > 0 $ such that for all $\lambda \in \partial K$ and $\delta < \delta_0$,
$\gamma (B(\lambda, \delta) \setminus K) \ge A \delta$.
The corollary also solves Problem 5.5 in \cite{ce93}. 
\smallskip

\begin{corollary}\label{DecompCorollary2}
If $\partial K = \partial_1 K, ~\gamma-a.a.$, where $\partial_1 K$ is defined as in \eqref{BOne}, then there is a Borel partition $\{\Delta_i\}_{i=0}^\infty$ of $\mbox{spt}(\mu)$ such that
 \[
 \ R^t(K, \mu ) = L^t(\mu|_{\Delta_0}) \oplus \bigoplus _{i = 1}^\infty R^t(K, \mu |_{\Delta_i})
 \]
satisfying:

(a) $S_{\mu |_{\Delta_i}}$ on $R^t(K, \mu |_{\Delta_i})$ is irreducible for $i \ge 1 ;$

(b) $\mbox{abpe}( R^t(K, \mu |_{\Delta_i})) = U_i$, $U_i$ is connected, $\Delta_i\subset K_i := \overline{U_i}$, $R^t(K, \mu |_{\Delta_i}) = R^t(K_i, \mu |_{\Delta_i})$ for $i \ge 1 ;$ and

(c) the evaluation map $\rho_i(f) = \rho(f) |_{U_i}$ is an isometric isomorphism and a weak$^*$ homeomorphism from $R^t(K_i, \mu |_{\Delta_i}) \cap L^\infty (\mu |_{\Delta_i})$ onto $H^\infty(U_i)$ for $i \ge 1 .$
\end{corollary}

\begin{proof}
Let
 \[
 \ R^t(K,\mu) = L^t(\mu |_{\Delta_0})\oplus \left (R^t(K^0, \mu _0) := \bigoplus_{n=1}^\infty R^t(K_n, \mu |_{\Delta_n}) \right)
 \]
be the decomposition as in Theorem \ref{DecompTheorem}. Let $\mathcal F^0$ be the non-removable boundary for $R^t(K^0, \mu _0)$ and $\mathcal R^0 = K^0 \setminus \mathcal F^0$.  
From the assumption and Proposition \ref{NFSetIsBig} (3), we see that $\partial K^0 \cap \partial K \subset \partial_1 K^0 \subset \mathcal F^0, ~\gamma-a.a.$. 
If $\gamma(\partial K^0 \cap \text{int}(K)\cap \mathcal R^0 ) > 0$, then by Remark \ref{ABPERemark}, we find $\lambda_0 \in \partial K^0 \cap \text{int}(K)\cap \text{abpe}(R^t(K^0, \mu _0))$, which implies $\lambda_0 \in \text{int}(K^0)$ as $K^0 = \sigma(S_{\mu_0})$. This is a contradiction. Hence, $\partial K^0 \subset \mathcal F^0, ~\gamma-a.a.$. By Theorem \ref{ABPETheorem}, we get
 \begin{eqnarray}\label{decompEq1}
 \ \mathcal R^0 = \text{abpe}(R^t(K^0, \mu _0)) = \cup_i U_i, ~\gamma-a.a., 
 \end{eqnarray}
where $U_i$ is a connected component of $\text{abpe}(R^t(K^0, \mu _0))$. So $\partial U_i \subset \mathcal F^0$.
Applying Corollary \ref{DecompCorollary1} to $R^t(K^0, \mu _0)$, we see that, with \eqref{decompEq1}, the summand $R^t(K, \mu |_{\Delta_0})$ as in Corollary \ref{DecompCorollary1} (b) must be zero since its removable set is zero by Lemma \ref{Lemma2} (3). This completes the proof.  
\end{proof}

\bigskip

\section{The algebra $R^t(K, \mu) \cap L^\infty(\mu)$ with non-trivial removable boundary}
\bigskip

In this section, we consider an open bounded set $\Omega$ with $\partial \Omega \cap \mathcal R \ne \emptyset,~\gamma-a.a.$ such that the relation between $H^\infty (\mathcal R)$ (or $R^t(K, \mu) \cap L^\infty(\mu)$) and some subalgebras of $H^\infty(\Omega)$ can be identified.   

Recall for $f\in R^t(K, \mu)$, by Theorem \ref{MTheorem2}, $\rho(f)$ is $\gamma$-continuous $\lambda\in \mathcal R, ~\gamma-a.a.$. We examine if  functions in $H^\infty(\Omega)$ that are $\gamma$-continuous on $\mathcal R$ belong to $H^\infty (\mathcal R)$.

We assume that there exists a Lipschitz graph $\Gamma$ such that
 \begin{eqnarray}\label{NTRemovableEq1}
 \ \mathcal R \cap \partial \Omega \subset \Gamma, ~\gamma-a.a..
\end{eqnarray}

Say $f\in H^\infty (\Omega)$ ($f(z) = 0 $ for $z \in \Omega^c$) is $\gamma$-continuous on $\mathcal R \cap \partial \Omega$ if there exists $a(z)\in L^\infty (\mathcal H^1|_{\mathcal R \cap \partial \Omega})$ such that the function
 \[
 \ f_a(z) = \begin{cases} f(z), & z \in \Omega; \\ a(z), & z \in \mathcal R \cap \partial \Omega; \\  0, & z\in (\Omega \cup (\mathcal R \cap \partial \Omega))^c, \end{cases}
\]
satisfies
\[
 \ \lim_{\delta\rightarrow 0} \dfrac{\gamma(B(\lambda, \delta)\cap \{|f_a(z)-a(\lambda)| > \epsilon \})}{\delta} = 0, ~\lambda\in \mathcal R \cap \partial \Omega,~ \gamma-a.a.
\]
for all $\epsilon > 0$.

For a bounded open set $\Omega$, we define the space $H^\infty_{\mathcal R}(\Omega)$ to be the set of $f\in H^\infty (\Omega)$ for which $f$ is $\gamma$-continuous at almost all $\lambda\in\mathcal R \cap \partial \Omega$ with respect to $\gamma$ (or $\mathcal H^1 |_{\mathcal R \cap \partial \Omega}$). Clearly, $f(z) = f_a(z),~ \mathcal L^2-a.a.$ since $\mathcal L^2(\mathcal R \cap \partial \Omega) = 0$. We will simply use $f$ (instead of $f_a$) for $f \in H^\infty_{\mathcal R}(\Omega)$. It is easy to check, for a smooth function $\varphi$ with compact support, $T_\varphi f \in H^\infty_{\mathcal R}(\Omega)$ whenever $f \in H^\infty_{\mathcal R}(\Omega)$. 

We require at least $f(z) = \chi_{\Omega \cup (\partial \Omega \cap \mathcal R)} \in H^\infty_{\mathcal R}(\Omega)$. Therefore the following condition must be met: For almost all $\lambda \in \mathcal R \cap \partial \Omega$ with respect to $\gamma$,
 \begin{eqnarray}\label{NTRemovableEq2}
 \ \lim_{\delta\rightarrow 0} \dfrac{\gamma(B(\lambda, \delta)\setminus (\Omega \cup (\partial \Omega \cap \mathcal R)))}{\delta} = 0.
\end{eqnarray}

\smallskip

\begin{theorem}\label{GSOBTheorem}
Let $\Omega$ be a bounded open subset of $\mathbb C$. Suppose $S_\mu$ on $R^t(K,\mu)$ is pure. If the conditions \eqref{NTRemovableEq1} and \eqref{NTRemovableEq2} hold, then for $f\in H^\infty_{\mathcal R}(\Omega)$, there exists $\hat f \in R^t(K, \mu) \cap L^\infty(\mu)$ such that $\mathcal C(\hat f g\mu)(z) = f(z) \mathcal C( g\mu)(z),~ z\in \Omega, ~\gamma-a.a.$ for each $g\perp R^t(K,\mu)$ and $\|\hat f\|_{L^\infty(\mu)} \le C_{96} \|f\|_{H^\infty(\Omega)}$. 
\end{theorem}

\begin{proof} 
We need to prove that the inequality \eqref{MLemma2CorEq} holds.
Let $S$ be a square with $S \cap \partial \Omega \ne \emptyset$. Assume that $\gamma (E:=S\cap \mathcal R \cap \partial \Omega) > 0$. So $E\subset \Gamma$.

Let $f\in H^\infty_{\mathcal R}(\Omega )$ be analytic off $S$ and $\|f\| \le 1$. For some fixed $\epsilon > 0$,
let $B_n^\epsilon$ be the set of $\lambda \in E$ satisfying:  for $0 < \delta < \frac{1}{n}$, 

(1) $\gamma (B(\lambda, \delta)\cap \{|f(z) - f(\lambda) | > \epsilon \}) < \epsilon \delta$, and 

(2) $\mathcal H^1 (B(\lambda, \delta)\cap E ) \ge \frac{1}{2}\delta$. 

Notice that if $\lambda$ is a Lebesgue point of $\mathcal H^1 |_{E}$, then
 \begin{eqnarray}\label{GSOBEq3}
 \ \lim_{\delta \rightarrow 0} \dfrac{\mathcal H^1 (B(\lambda, \delta)\cap E)}{\delta} \ge 1
 \end{eqnarray}
(see, for example, 
 Chapter 6 in \cite{M95}).  
 Therefore, together with $\gamma$-continuity of $f$ for (1), we conclude
$\mathcal H^1 (E \setminus \cup_n B_n^\epsilon) = 0$.
Let $n$ be large enough so that $\gamma (E \setminus B_n^\epsilon) < \epsilon$ ($\gamma \approx \mathcal H^1|_\Gamma$ by Lemma \ref{lemmaBasic0} (7)).
Let $\{S_{ij}, \varphi_{ij}\}$ be a partition of unity with the length $l$ of $S_{ij}$ less than $\frac{1}{2n}$. Then
 \[
 \ f = \sum_{2S_{ij} \cap B_n^\epsilon = \emptyset}T_{\varphi_{ij}}f + \sum_{2S_{ij} \cap B_n^\epsilon \ne \emptyset}T_{\varphi_{ij}}f.
 \]
We have the following estimate for $\lambda\in 2S_{ij} \cap B_n^\epsilon \ne \emptyset$,
 \[
 \ \begin{aligned}
 \ &|(T_{\varphi_{ij}}f)'(\infty) | \\
 \ = & \left | \int \bar \partial (\varphi_{ij}) (z) f(z) d \mathcal L^2(z) \right | \\
 \ = &\left | \int \bar \partial ( \varphi_{ij}) (f - f(\lambda)) d \mathcal L^2 \right | \\
 \ \le & \left | \int_{\{|f(z) - f(\lambda)| \le \epsilon \}} \bar \partial (\varphi_{ij})(z) (f(z) - f(\lambda)) d \mathcal L^2(z) \right |\\
 \ & + \left | \int_{\{|f(z) - f(\lambda)| > \epsilon \}} \bar \partial (\varphi_{ij})(z) (f(z) - f(\lambda)) d \mathcal L^2(z)\right | \\
 \ \le & \|\bar \partial (\varphi_{ij})\|_\infty (4\epsilon l^2 + 2\|f\| \mathcal L^2(B(\lambda, l)\cap \{z\in \Omega: |f(z) - f(\lambda)| > \epsilon \})) \\
 \ \le & \|\bar \partial (\varphi_{ij})\|_\infty (4\epsilon l^2 + 8\pi \gamma (B(\lambda, l)\cap \{z\in \Omega: |f(z) - f(\lambda)| > \epsilon \})^2) \\
\ \le & C_{97} \epsilon l
 \ \end{aligned}
 \]
where Lemma \ref{lemmaBasic0} (2) is used. Similarly, $|\beta(T_{\varphi_{ij}}f,c_{ij})| \le C_{97} \epsilon l^2$.

We use the standard Vitushkin scheme as in the proof of Theorem \ref{Lemma3} (or the same argument of the modified Vitushkin scheme as in \eqref{FDelta} and \eqref{FBounded1}) by choosing $g_{ij}$ and $h_{ij}$ for $2S_{ij} \cap B_n^\epsilon \ne \emptyset$ such that $g_{ij}$ and $h_{ij}$ are analytic off a segment $L_{ij} = L_1\cup L_2$, where line segments $L_1$ and $L_2$ are inside $S_{ij}$ with length $\epsilon l$, $dist(L_1, L_2)\ge \frac{l}{2}$, $g_{ij}(\infty) = h_{ij}(\infty) = 0$, $g'_{ij}(\infty) = \epsilon l$, $h'_{ij}(\infty) = 0$, and $\beta(h_{ij}, c_{ij}) = \epsilon l^2$. This constructs $f_l$ as in \eqref{SVEq} such that 
 \[
 \ F_l = \sum_{2S_{ij} \cap B_n^\epsilon = \emptyset}T_{\varphi_{ij}}f + f_l,
 \]
$\|F_l\| \le C_{97}$, and $f'(\infty) = F_l'(\infty)$. Since $F_l$ is analytic off 
 \[
 \ D:=\left (S\cap \partial \Omega \cap \bigcup_{2S_{ij} \cap B_n^\epsilon = \emptyset} \text{spt}(\varphi_{ij}) \right ) \cup \bigcup_{2S_{ij} \cap B_n^\epsilon \ne \emptyset} L_{ij},
 \]  
we have the following estimation:
 \[
 \ \begin{aligned}
 \ & |f'(\infty) | = |F_l'(\infty) | \le C_{97} \gamma (D)  \\
\ \le & C_{97} \gamma \left ((S\cap \partial \Omega \setminus B_n^\epsilon) \cup \bigcup_{2S_{ij} \cap B_n^\epsilon \ne \emptyset} L_{ij} \right )  \\
\ \le & C_{97} A_T\left (\gamma (S\cap \partial \Omega \cap \mathcal F) + \gamma (S\cap \partial \Omega \cap \mathcal R \setminus B_n^\epsilon) + \sum_{2S_{ij} \cap B_n^\epsilon \ne \emptyset} \gamma ((L_{ij}) \right ) \\
\ \le & C_{98}\left (\gamma (S\cap \partial \Omega \cap \mathcal F) + \epsilon + \epsilon \sum_{2S_{ij} \cap B_n^\epsilon \ne \emptyset} \mathcal H^1 (B(c_{ij}, 2l)\cap \mathcal R \cap \partial \Omega) \right ) \\
\ \le & C_{98}\left (\gamma (S\cap \partial \Omega \cap \mathcal F) + \epsilon + 100 \epsilon \mathcal H^1 (\mathcal R \cap \partial \Omega) \right )
 \ \end{aligned}
 \]
where \eqref{GSOBEq3} is used from last step 3 to last step 2. Therefore,
 \begin{eqnarray}\label{GSOBEq31}
 \ |f'(\infty) | \le C_{98}\gamma (S\cap \partial \Omega \cap \mathcal F). 
 \end{eqnarray}
Hence, for a smooth function $\varphi$ supported in $S$ and $f\in H^\infty_{\mathcal R}(\Omega)$ with $\|f\| \le 1$, the function $T_\varphi f\in H^\infty_{\mathcal R}(\Omega)$ is analytic off $S$. We apply \eqref{GSOBEq31} to $\frac{T_\varphi f}{\|T_\varphi f \|}$ and get 
 \[
 \ \left | \int \bar \partial \varphi f d\mathcal L^2\right | \le C_{98}\|T_\varphi f \|\gamma (S\cap \partial \Omega \cap \mathcal F) \le C_{99}\delta \|\bar \partial \varphi\|\gamma (S\cap \partial \Omega \cap \mathcal F). 
 \]
where $\delta$ is the length of $S$. 
Now applying Corollary \ref{MLemma2Cor} for $E = \partial \Omega$ and $E_0 = \partial \Omega \cap \mathcal F$, we prove the theorem. 
\end{proof}
\smallskip

\begin{corollary}\label{GSOBCorollary}
Let $\Omega$ be a bounded open subset of $\mathbb C$. Suppose $S_\mu$ on $R^t(K,\mu)$ is pure. If the condition \eqref{NTRemovableEq1} holds, $K = \overline{\Omega}$, and $\text{abpe}(R^t(K,\mu)) = \Omega$, then the evaluation map
 \[
 \ \rho:~ f\in R^t(K, \mu) \cap L^\infty(\mu) \rightarrow \rho(f)|_{\Omega} \in H^\infty(\Omega) 
 \]
is an isometric isomorphism and a weak$^*$ homeomorphism from $R^t(K, \mu) \cap L^\infty(\mu)$ onto $H^\infty_{\mathcal R}(\Omega) = H^\infty (\mathcal R)$.
\end{corollary}

\begin{proof}
Clearly $H^\infty_{\mathcal R}(\Omega) \supset H^\infty(\mathcal R)$.
Notice that, by Theorem \ref{ABPETheorem}, $\mathcal R \approx \Omega \cup (\mathcal R\cap\partial \Omega), ~ \gamma-a.a.$.  
From Theorem \ref{DensityCorollary}, we see the condition \eqref{NTRemovableEq2} holds. By Theorem \ref{GSOBTheorem}, for $f\in H^\infty_{\mathcal R}(\Omega)$, we let $\hat f$ be as in Theorem \ref{GSOBTheorem}. Then using Theorem \ref{algebraEq} and  Theorem \ref{GSOBTheorem}, we get $\rho(\hat f) \in H^\infty (\mathcal R)$ and $\rho(\hat f)(z) = f(z), ~ \mathcal L^2_{\mathcal R}-a.a.$. Hence, $f\in H^\infty (\mathcal R)$. The corollary now follows from Theorem \ref{algebraEq}.  
\end{proof}
\smallskip

Now under additional assumptions that $\Omega$ contains certain nontangential limit regions, the $\gamma$-continuity can be replaced by certain nontangential limit conditions. Assume that the condition \eqref{NTRemovableEq1} holds and for almost all $\lambda \in \mathcal R \cap \partial \Omega \subset \Gamma$, there exists $\delta_\lambda > 0$ such that 
 \begin{eqnarray}\label{GSOBEq4}
 \ UC(\lambda, \alpha, \delta_\lambda) \cup LC(\lambda, \alpha, \delta_\lambda) \subset \Omega.
 \end{eqnarray}
Under the conditions \eqref{NTRemovableEq1} and  \eqref{GSOBEq4}, we define the upper nontangential limit:
 \begin{eqnarray}\label{GSOBEq5}
 \ f_+(\lambda ) = \lim_{z\in UC(\lambda, \alpha, \delta) \rightarrow \lambda} f(z)
 \end{eqnarray}
if the limit exists.
Similarly, we define the lower nontangential limit:
 \begin{eqnarray}\label{GSOBEq6}
 \ f_-(\lambda ) = \lim_{z\in LC(\lambda, \alpha, \delta) \rightarrow \lambda} f(z)
 \end{eqnarray}
if the limit exists.
Define 
 \[
 \ H^\infty_{\mathcal R} (\Omega, \mathcal{NT}) = \{f\in H^\infty (\Omega):~ f_+(\lambda) = f_-(\lambda),~\mathcal H^1 |_{\mathcal R \cap \partial \Omega}-a.a.\}. 
 \]  
\smallskip

\begin{lemma}\label{NTEquivalent}
If a bounded open set $\Omega$ satisfies \eqref{NTRemovableEq1} and \eqref{GSOBEq4}, then  
 \[
 \ H^\infty_{\mathcal R}(\Omega) = H^\infty_{\mathcal R} (\Omega, \mathcal{NT}).
 \]   
\end{lemma}

\begin{proof}
(Proof of $H^\infty_{\mathcal R}(\Omega) \subset H^\infty_{\mathcal R} (\Omega, \mathcal{NT})$):
For $f\in H^\infty_{\mathcal R}(\Omega)$ and $\lambda \in \mathcal R \cap \partial \Omega$, $A_{f,\epsilon} = \{|f(z) - a(\lambda)| > \epsilon\}$, we have
\[
 \ \lim_{\delta\rightarrow 0}\dfrac{\gamma(UC(\lambda, \alpha, \delta)\cap A_{f,\epsilon})}{\delta} \le \lim_{\delta\rightarrow 0}\dfrac{\gamma(B(\lambda, \delta)\cap A_{f,\epsilon})}{\delta} = 0.
 \]
From Lemma \ref{lemmaBasic11}, we conclude that $f\in H^\infty_{\mathcal R} (\Omega, \mathcal{NT})$. 

(Proof of $H^\infty_{\mathcal R} (\Omega, \mathcal{NT}) \subset H^\infty_{\mathcal R}(\Omega)$):
Let $B_n$ be the set of $\lambda \in \mathcal R \cap \partial \Omega$ such that \eqref{GSOBEq4} holds for all $\delta_{\lambda} < \frac{1}{n}$. From the assumption, we get
 \[
 \ \mathcal H^1 (\mathcal R \cap \partial \Omega \setminus \bigcup_n B_n) = 0. 
 \] 
For a given $\epsilon > 0$, choose $n$ large enough so that $\mathcal H^1 (\mathcal R \cap \partial \Omega \setminus B_n) < \frac{\epsilon}{2}$. Let $D_n \subset B_n$ be a compact subset so that $\mathcal H^1 (B_n \setminus D_n)  < \frac{\epsilon}{2}$. Let $\{B(\lambda_k, \frac{1}{4n})\}$ be an open cover of $D_n$, where $\lambda_k\in D_n$. Set $O_1 = \cup _k B(\lambda_k, \frac{1}{4n})$ and $O_2 = \cup _k B(\lambda_k, \frac{1}{2n})$. Then $D_n\subset O_1 \subset O_2$. Let $\psi$ be a smooth function supported in $O_2$, $0\le \psi\le 1$, and $\psi (z) = 1$ for $z\in O_1$. For $f\in H^\infty_{\mathcal R} (\Omega, \mathcal{NT})$, we have $f = u:= T_\psi f + y := T_{1-\psi} f$. Clearly, $y$ is analytic on $O_1$, hence, $y_+ (z) = y_-(z), ~ \mathcal H^1 |_{D_n}-a.a.$. The function $u$ is analytic off $\overline{O_2}$. Therefore, $u$ is analytic on $G = (\overline{O_2})^c \cup \cup_{\lambda \in D_n} D_n(\lambda)$, where 
$D_n(\lambda) = UC(\lambda, \alpha, \frac{1}{n}) \cup LC(\lambda, \alpha, \frac{1}{n})$.
It is clear that $\mathcal H^1 (\partial G) < \infty$. From Lemma \ref{lemmaBasic0} (6), there exists a function $w$ bounded by $\|u\|$ and supported on $\partial G$ such that $u(z) = \mathcal C(w\mathcal H^1|_{\partial G})(z)$ for $z\in G$. From Theorem \ref{GPTheorem1}, we get, for $\lambda\in D_n$,
 \[
\ \lim_{\delta\rightarrow 0}\dfrac{\gamma(UC(\lambda, \alpha, \delta)\cap \{|u(z) - v^+(w\mathcal H^1, \Gamma, \lambda ) | > \epsilon\})}{\delta}   = 0.
 \]
From Lemma \ref{lemmaBasic11}, 
 $u_+ (\lambda ) = v^+(w\mathcal H^1, \Gamma, \lambda ), ~ \mathcal H^1 |_{D_n}-a.a.$ 
Similarly,
 $u_- (\lambda ) = v^-(w\mathcal H^1, \Gamma, \lambda ), ~ \mathcal H^1 |_{D_n}-a.a.$.
Since $f_+(\lambda ) = f_-(\lambda ), ~ \mathcal H^1 |_{D_n}-a.a.$, we see that 
 \[
 \ v^+(w\mathcal H^1, \Gamma, \lambda ) = v^-(w\mathcal H^1, \Gamma, \lambda ) = \mathcal C(w\mathcal H^1|_{\partial G})(\lambda ), ~\mathcal H^1 |_{D_n}-a.a..
 \]
Certainly, $G$ satisfies the assumption \eqref{NTRemovableEq2} on $D_n$ (same proof as in \eqref{GEDensity}).  Hence, Theorem \ref{GPTheorem1} implies that $u$ is $\gamma$-continuous on $D_n, ~ \mathcal H^1 |_{D_n}-a.a.$. Thus, $f$ is $\gamma$-continuous on $\mathcal R \cap \partial \Omega, ~ \mathcal H^1 |_{\mathcal R \cap \partial \Omega}-a.a.$. Therefore, $f \in H^\infty_{\mathcal R}(\Omega )$.   

\end{proof}
\smallskip

\begin{corollary}\label{GSOBCorollary2}
Let $\Omega$ be a bounded open subset of $\mathbb C$. Suppose $S_\mu$ on $R^t(K,\mu)$ is pure.  If the conditions \eqref{NTRemovableEq1} and \eqref{GSOBEq4} hold, $K = \overline{\Omega}$, and $\text{abpe}(R^t(K,\mu)) = \Omega$, then 
 \begin{eqnarray}\label{GSOBCorollary2Eq}
 \ f_+(\lambda) = f_-(\lambda), ~ \mathcal H^1 |_{\mathcal R \cap \partial \Omega}-a.a.
 \end{eqnarray}
for every $f\in R^t(K, \mu)$ and the evaluation map
$\rho: f\in R^t(K, \mu) \cap L^\infty(\mu) \rightarrow \rho(f)|_{\Omega} \in H^\infty(\Omega)$
is an isometric isomorphism and a weak$^*$ homeomorphism from $R^t(K, \mu) \cap L^\infty(\mu)$ onto $H^\infty_{\mathcal R}(\Omega, \mathcal {NT})$.
\end{corollary}

\begin{proof}
For every $f\in R^t(K, \mu)$,
since $\rho(f)(z) = f(z)$ is analytic on $\Omega$ and $\rho(f)$ is $\gamma$- continuous on $\mathcal R ~ \gamma-a.a.$ (Theorem \ref{MTheorem2}), 
\eqref{GSOBCorollary2Eq} follows from the same proof of $H^\infty_{\mathcal R}(\Omega) \subset H^\infty_{\mathcal R} (\Omega, \mathcal{NT})$ in Lemma \ref{NTEquivalent}. Applying Corollary \ref{GSOBCorollary} and Lemma \ref{NTEquivalent}, we prove the corollary.
\end{proof}
\smallskip

Theorem \ref{SOBTheorem} directly follows from the above corollary. In fact, $\Omega = G_U\cup G_L$, $\Gamma$ is the real line,  and $\mathcal R\cap \partial \Omega$ is a subset of $E$.

\bigskip

{\bf Acknowledgments.} The authors would like to thank Professor Akeroyd for providing many useful discussions and comments, in particular, regarding examples of string of beads sets.
\bigskip

\appendix

\backmatter
\bibliographystyle{amsalpha}

\printindex

\end{document}